\documentclass[11pt]{amsart}
\UseRawInputEncoding
\usepackage{amssymb,mathrsfs,graphicx,enumerate,color}
\usepackage{amsthm,amsfonts,amssymb,epsfig,graphics,amsmath,amsbsy,enumerate}
\usepackage{colortbl}
\definecolor{black}{rgb}{0.0, 0.0, 0.0}
\definecolor{red}{rgb}{1.0, 0.5, 0.5}

\title[Stability of composite waves of viscous shock and rarefaction]
{Time-asymptotic stability of composite waves  of  viscous shock and  rarefaction  for  barotropic Navier-Stokes equations}

\author[Kang]{Moon-Jin Kang}
\address[Moon-Jin Kang]{\newline Department of Mathematical Sciences,
\newline
Korea Advanced Institute of Science and Technology, Daejeon 34141, Korea}
\email{moonjinkang@kaist.ac.kr}

\author[Vasseur]{Alexis F. Vasseur}
\address[Alexis F. Vasseur]{\newline Department of Mathematics, \newline The University of Texas at Austin, Austin, TX 78712, USA}
\email{vasseur@math.utexas.edu}

\author[Wang]{Yi Wang}
\address[Yi Wang]{\newline Institute of Applied Mathematics, AMSS, CAS, Beijing 100190, P. R. China
\newline
and School of Mathematical Sciences, University of Chinese Academy of Sciences,
\newline Beijing 100049, P. R. China}
\email{wangyi@amss.ac.cn}

\newtheorem{theorem}{Theorem}[section]
\newtheorem{lemma}{Lemma}[section]

\newtheorem{proposition}{Proposition}[section]
\newtheorem{remark}{Remark}[section]

\newcommand{\bbr}{\mathbb {R}}
\newcommand{\R}{\mathbb {R}}

\newcommand{\mb}{\mathbf}

\newcommand{\wt}{\widetilde}

\newcommand{\deltas}{\delta_S}
\newcommand{\deltar}{\delta_R}

\numberwithin{figure}{section}
\newcommand{\beq}{\begin{equation}}
\newcommand{\eeq}{\end{equation}}
\newcommand{\bsp}{\begin{split}}
\newcommand{\esp}{\end{split}}

\newcommand{\di}{\displaystyle}

\newcommand{\ds}{\delta_S}

\def\eps{\varepsilon }

\newcommand\adots{\mathinner{\mkern2mu\raise1pt\hbox{.}
\mkern3mu\raise4pt\hbox{.}\mkern1mu\raise7pt\hbox{.}}}

\def\charf {\mbox{{\text 1}\kern-.30em {\text l}}}
\def\lam{\lambda}  % the "Boltzmann" scaling
 \newcommand{\bmat}{\begin{pmatrix}} 
  \newcommand{\emat}{\end{pmatrix}}
  
\newcommand{\s}{\sigma}

\begin{document}
%%%%%%%%%%%%%%%%
\bibliographystyle{plain}

\date{\today}

\subjclass[2010]{}

\thanks{\textbf{Acknowledgment.}  M.-J. Kang was partially supported by the NRF-2019R1C1C1009355.
A. Vasseur was partially supported by the NSF grant: DMS 1614918. Y. Wang is supported by NSFC grants No. 12090014 and 11688101.
}

\begin{abstract}
We prove the  time-asymptotic stability of composite waves consisting of the  superposition of a viscous shock  and a rarefaction  for the one-dimensional compressible barotropic Navier-Stokes equations.  Our result solves a long-standing problem first mentioned in 1986 by Matsumura and Nishihara in \cite{MN86}. The same authors introduced it officially as an open problem in 1992 in \cite{MN} and it was again described as very challenging open problem in 2018 in the survey paper \cite{MatsumuraBook}. The main difficulty is due to the incompatibility of the  standard anti-derivative method, used to study the stability of viscous shocks, and  the energy method used for  the stability of rarefactions. Instead of the anti-derivative method, our proof uses the $a$-contraction with shifts theory recently developed by two of the authors. This method is energy based, and can seamlessly handle the superposition of waves of different kinds. 
%by introducing the suitable shift function to the viscous shock wave and the weight function, which solves an open problem proposed by Matsumura and Nishihara \cite{MN} since 1992. Here quite different from the previous classical results, we do not introduce the anti-derivative variable such that we can tackle the superposition of the viscous shock wave and rarefaction wave.
\end{abstract}

\maketitle \centerline{\date}

\tableofcontents

\section{Introduction}
\setcounter{equation}{0}

Consider the  one-dimensional compressible  barotropic Navier-Stokes equations. 
In the Lagrangian  mass coordinates, the system is described as 
\begin{align}\label{NS}
\begin{aligned}
\left\{ \begin{array}{ll}
        v_t- u_x =0,\qquad \quad\quad x\in\bbr,\ t\geq 0,\\
       u_t  +p(v)_x=  (\mu\frac{u_x}{v})_x, \\
        \end{array} \right.
\end{aligned}
\end{align}
where  the unknown functions $v=v(t,x)>0$, and  $u(t,x)$  represent respectively the specific volume, and the velocity of the gas. 
The   pressure function $p$ is  given by the well-known $\gamma-$law 
\[
p(v)=bv^{-\gamma},
\]
where $b>0, \gamma>1$ are both constants depending on the fluid, and  the constant  $\mu>0$ corresponds to the viscosity coefficient. Without loss of generality, we  normalize two of the   constants as $\mu=1$ and $b=1$.
The system is then endowed with initial values:
$$
(v,u)(t=0, x)=(v_0(x), u_0(x)), \qquad x\in \bbr.%\rightarrow(v_\pm, u_\pm), \quad {\rm as}\quad x\rightarrow\pm\infty.
$$
We  consider initial values with fixed end states $(v_\pm, u_\pm)\in\bbr^+\times\bbr$, that is such that 
\begin{equation}\label{in}
(v_0(x), u_0(x)) \rightarrow(v_\pm, u_\pm), \quad {\rm as}\quad x\rightarrow\pm\infty.
\end{equation}
On top of its physical relevance, system \eqref{NS} can be seen as the typical example of viscous conservation laws involving a physical viscosity. 
\vskip0.3cm
The large-time behavior of  solutions to \eqref{NS}, with initial values verifying \eqref{in},  is closely related to the Riemann problem of the associated  Euler equations: 
\begin{align}
\begin{aligned}\label{E}
\left\{ \begin{array}{ll}
        v_t- u_x =0,\\
       u_t  +p(v)_x= 0, \\
        \end{array} \right.
\end{aligned}
\end{align}
with the Riemann initial data
\begin{align}
\begin{aligned}\label{Ei}
(v,u)(t=0,x)=\left\{ \begin{array}{ll}
        (v_-,u_-),\quad x<0,\\
       (v_+,u_+),\quad x>0, \\
        \end{array} \right.
\end{aligned}
\end{align}
corresponding to the end states \eqref{in}. % , that is such that $(v_L,u_L)=(v_-,u_-)$ and $(v_R,u_R)=(v_+,u_+)$.
In the scalar case (where the system \eqref{NS} is replaced by a single viscous equation), the time-asymptotic stability of the viscous waves, and their link to the inviscid problem was first proved in 1960 by Ilin-Oleinik \cite{O} (see also Sattinger \cite{S}). 
The case for systems as  \eqref{NS} is far more difficult (see \cite{MatsumuraBook}). 
\vskip0.3cm
One of the motivation for the study of  large-time behavior of solutions to compressible  Navier-Stokes equation for Riemann initial data was to obtain insights about  inviscid limit to the Euler equation.
In 2005  \cite{BB} , Bianchini-Bressan showed, for small BV initial values, the convergence at the inviscid limit of solution to  parabolic system with ``artificial viscosity"  to the unique solution of the associated hyperbolic system.
However, to this day, the result is still unknown for the physical Navier-Stokes system, even in the barotropic case \eqref{NS}.

\vskip0.3cm
\noindent
 {\bf Riemann problem for the inviscid model: } Let us first describe the well-known solution  of the Riemann problem for the inviscid model \eqref{E}-\eqref{Ei}, first proposed and solved by Riemann \cite{Riemann} in 1860s. 
%Let us first remind   briefly the reader  the construction of those solutions. 
This system of conservation laws is   strictly hyperbolic. This means that the derivative of the flux function $(-u,p(v))$ with respect to the conserved variables, about a fixed  state $(v,u)\in \R^+\times \R$:%, which  the matrix obtained from the linearization of \eqref{E} , which corresponds to the derivative of the flux function:
$$
\left(
\begin{array}{cr}
0&-1\\
p'(v)&0
\end{array}
\right)
$$
is diagonalizable with real distinct eigenvalues. Note that this matrix defined the waves generated by the linearization of the system \eqref{E} about this fixed state $(v,u)\in \R^+\times \R$.
Its  eigenvalues $\lambda_1=-\sqrt{-p'(v)}<0$ and  $\lambda_2=\sqrt{-p'(v)}>0$ generate both characteristic fields which are  genuinely nonlinear. Therefore, the self-similar solution, so called {\it Riemann solution}, of the Riemann problem is determined by a combination of at most two  elementary solutions from the following four families: 1-rarefaction; 2-rarefaction; 1-shock and 2-shock (see for instance \cite{Dafermos1}). These families are completely defined through their associated curves in the states plane $\R^+\times \R$.
For any  $(v_R,u_R)\in \R^+\times \R$, the 1-rarefaction curve $R_1(v_R,u_R)$ corresponds to the integral curve of the  first eigenvalue $\lambda_1$, and  is defined by  
\begin{equation}\label{(1.9)}
 R_1 (v_R, u_R):=\Bigg{ \{} (v, u)\Bigg{ |}v<v_R ,~u=u_R-\int^v_{v_R}
 \lambda_1(s) ds\Bigg{ \}}.
\end{equation}
The 2-rarefaction curve $R_2$ can be defined in the same way from the second eigenvalue  $\lambda_2$.  
%Especially, if $(v_-,u_-)= R_1 (v_+, u_+)$, the solution of  \eqref{Ei} is given as by a rarefaction (see section \ref{sec-preli}). 
For any initial values of the Riemann problem \eqref{Ei} with $(v_-,u_-)=(v_L,u_L)$, $(v_+,u_+)=(v_R,u_R)$, such that $(v_L,u_L)\in R_1(v_R,u_R)$, the solution $(v^r,u^r)$ of \eqref{E} is the 1-rarefaction wave defined as 
%If $w_- < w_m$, then \eqref{BE} has the rarefaction wave fan $w^r(t, x) = w^r(x/t)$ given by
\begin{equation} \label{BESintro}
\lambda_1(v^r(t, x))  = \begin{cases}
\lambda_1(v_L) , \qquad x < \lambda_1(v_L)t, \\
\frac{x}{t}, \qquad  \lambda_1(v_L)t \leq x \leq  \lambda_1(v_R)t, \\
\lambda_1(v_R), \qquad x > \lambda_1(v_R)t,
\end{cases}
\end{equation}
together with 
%Set $w_- = \lambda_1(v_-), w_m = \lambda_1(v_m)$. It is easy to check that the 1-rarefaction wave $(v^r, u^r)(t, x) = (v^r, u^r)(x/t)$ to the Riemann problem \eqref{E} - \eqref{Ei} is given explicitly by
\begin{equation}\label{rareintro}
\begin{array}{ll}
 % &\lambda_1(v^r(x/t)) = w^r(x/t), \\[1mm]
  &z_1(v^r(t,x), u^r(t,x))=  z_1 (v_L, u_L)=z_1 (v_R,u_R),
\end{array}
\end{equation}
where $z_1(v,u)=u+\int^v \lambda_1(s)ds$ is called the 1-Riemann invariant to the Euler equation \eqref{E}.  The case of 2-rarefaction wave is treated similarly from the second eigenvalue $\lambda_2$.
 \vskip0.3cm
We can now define the shock  curves using the Rankine-Hugoniot condition, as the one-parameter family of all the $(v,u)$  such that there exists $\sigma$ with:
\begin{equation}\label{RH}
%S_2(v_+,u_+):=\bigg\{(v,u)\bigg{ |}
\begin{array}{l} -\sigma(v_R-v)-(u_R-u)=0,\\[2mm]
\di  -\sigma(u_R-u)+(p(v_R)-p(v))=0.
\end{array}
%\bigg\}.\label{(1.8)}
\end{equation}
The general theory shows that this condition defines actually 2 curves that meet at the point $(v_R,u_R)$, one for the value $\sigma=-\sqrt{-\frac{p(v_R)-p(v)}{v_R -v}}$ (the 1-shock curve $S_1(v_R,u_R)$ which corresponds to admissible shocks for $v>v_R$), and one for the value $\sigma=\sqrt{-\frac{p(v_R)-p(v)}{v_R -v}}$ (the 2-shock curve $S_2(v_R,u_R)$ with admissible shocks for $v<v_R$).
\vskip0.1cm
Whenever $(v_L,u_L)\in S_1(v_R,u_R)\cup S_2(v_R,u_R)$, the solution $(v^s,u^s)$ to \eqref{E}-\eqref{Ei} with $(v_-,u_-)=(v_L,u_L)$, $(v_+,u_+)=(v_R,u_R)$, is given by the discontinuous traveling wave defined as
\begin{equation} \label{BERintro}
 (v^s,u^s)(t,x)= \begin{cases}
(v_L,u_L) , \qquad x < \sigma t, \\
%\frac{x}{t}, \qquad  \lambda(v_-)t \leq x \leq  \lambda(v_+)t, \\
(v_R,u_R), \qquad x > \sigma t.
\end{cases}
\end{equation}
\vskip0.3cm
For the general case of any  states $(v_-,u_-), (v_+,u_+)\in \R^+\times \R$, it can be shown that there exists a (unique) intermediate state $(v_m,u_m)\in \R^+\times \R$ such that $(v_m,u_m)$ is on a curve of the second families from $(v_+,u_+)$ (either $R_2(v_+,u_+)$ or $S_2(v_+,u_+)$), and $(v_-,u_-)$ is on a curve of the first families from $(v_m,u_m)$ (either $R_1(v_m,u_m)$ or $S_1(v_m,u_m)$). The solution $(v,u)$ of \eqref{E}-\eqref{Ei} is then obtained by the juxtaposition of the two associated waves
$$
(v,u)(t,x)=(v_1,u_1)(t,x)+(v_2,u_2)(t,x)-(v_m,u_m).
$$
The wave  $(v_1,u_1)$ is 1-rarefaction fan solution to \eqref{BESintro}-\eqref{rareintro} if $(v_-,u_-)\in R_1(v_m,u_m)$, or 1-shock solution to \eqref{BERintro} if $(v_-,u_-)\in S_1(v_m,u_m)$, with $(v_L,u_L)=(v_-,u_-)$, $(v_R,u_R)=(v_m,u_m)$.   The wave  $(v_2,u_2)$ is  2-shock solution to 
\eqref{BERintro} if $(v_m,u_m)\in S_2(v_+,u_+)$, or, 2-rarefaction fan solution if $(v_m,u_m)\in R_2(v_+,u_+)$, both with the end states   $(v_L,u_L)=(v_m,u_m)$, $(v_R,u_R)=(v_+,u_+)$. Note that the cases of single wave are included as   degenerate cases when $(v_-,u_-)=(v_m,u_m)$, or $(v_+,u_+)=(v_m,u_m)$. 
 \vskip0.3cm
 \noindent
 {\bf Previous time-asymptotic results for the viscous model:}
The time-asymptotic behavior of the viscous solution to \eqref{NS} depends on whether the associated Riemann solution to the associated inviscid model \eqref{E}-\eqref{Ei} involves shock waves or rarefaction waves. In the case where \eqref{Ei} is a shock, the viscous counterpart   for \eqref{NS}, called viscous shock, is the traveling wave $(\wt v^S(x-\sigma t), \wt  u^S(x-\sigma t))$ defined by the following ODE:
\begin{equation}\label{VSintro}
\left\{
\begin{array}{ll}
\di -\sigma (\wt v^S)^\prime-( \wt u^S)^\prime=0,\\[3mm]
\di -\sigma( \wt u^S)^\prime+(p( \wt v^S))^\prime=\Big(\frac{( \wt u^S)^\prime}{ \wt v^S}\Big)^\prime,\\[3mm]
\di ( \wt v^S, \wt u^S)(-\infty)=(v_L, u_L), \qquad ( \wt v^S, \wt u^S)(+\infty)=(v_R, u_R).
\end{array}
\right.
\end{equation}
 Matsumura-Nishihara \cite{MN-S}  proved the stability of  the viscous shock waves \eqref{VSintro} for  the compressible Navier-Stokes equations \eqref{NS}. Independently, Goodman showed in \cite{G} the same result of a general system with artificial diffusion. This corresponds to the case where diffusion is added to all the equations of the system. 
In both papers, the proof were done under the  zero mass condition which is crucial for using the so called anti-derivative method. 
 Then Liu \cite{Liu}, Szepessy-Xin \cite{SX} and Liu-Zeng \cite{LZ} removed the crucial zero mass condition in \cite{MN-S, G} by introducing the constant shift on the viscous shock and the diffusion waves and the coupled diffusion waves in the transverse characteristic fields. Masica-Zumbrun \cite{MZ} proved the spectral stability of viscous shock to 1D compressible Navier-Stokes system under a spectral condition, which is slightly weaker than the zero mass condition. The case of the superposition  of two shocks for the Navier-Stokes-Fourier system was treated by  Huang-Matsumura in \cite{HM}. Finally, the asymptotic stability of viscous shocks for Navier-stokes systems with degenerated viscosities was studied in Matsumura-Wang \cite{matwan}, and generalized to a larger class of viscosity in  \cite{VY} using the BD entropy introduced by Bresch-Desjardins in  \cite{BD-06}.
\vskip0.3cm
The treatment of stability of rarefactions is performed with very different techniques based on direct energy methods. Matsumura-Nishihara \cite{MN86, MN} first proved the time-asymptotic stability of the rarefaction waves for  the compressible and isentropic Navier-Stokes equations \eqref{NS}. It was later  generalized to the Navier-Stokes-Fourier system by Liu-Xin \cite{LX} and Nishihara-Yang-Zhao \cite{NYZ}.
\vskip0.3cm
\noindent
{\bf The case of the juxtaposition of a shock and a rarefaction:}
However, the time-asymptotic stability of the superposition of a viscous shock wave and a rarefaction wave has been an open problem up to now. The main difficulty is that the classical anti-derivative method used for 
the stability of shocks does not match  well with the energy method classically used for the stability of rarefactions. The problem of the stability of such a superposition of a rarefaction and a viscous shock was first mentioned in 1986 by Matsumura and Nishihara in \cite{MN86}. The same authors introduced it officially as an open problem in 1992 in \cite{MN} and Matsumura  described it again as very challenging open problem in 2018 in the survey paper \cite{MatsumuraBook}. Our main theorem is proving this conjecture.

\begin{theorem}\label{thm:main}
For a given constant state $(v_+,u_+)\in\bbr_+\times\bbr$, there exist constants $\delta_0, \eps_0>0$ such that the following holds true.\\
For any $(v_m,u_m)\in S_2(v_+,u_+)$ and $(v_-,u_-)\in R_1(v_m,u_m)$ such that 
$$|v_+-v_m|+|v_m-v_-|\leq \delta_0,$$ 
denote $(v^r,u^r)(\frac xt)$ the 1-rarefaction solution to \eqref{E} with end states $(v_-,u_-)$ and $(v_m,u_m)$, and $(\tilde{v}^S,\tilde{u}^S)(x-\sigma t)$ the 2-viscous shock solution of  \eqref{VSintro} with end states $(v_m,u_m)$ and $(v_+,u_+)$.
%Consider the composite wave as in \eqref{SW}:
%\[
%\wt W(t,x) := \left(v^r(\frac xt)+\wt v^S(x-\s t)-v_m, u^r(\frac xt)+\wt u^S(x-\s t)-u_m %\right) .
%\]
Let $(v_0, u_0)$ be  any initial data such that 
\begin{equation}\label{i-p}
\sum_{\pm}\Big( \|v_0-v_\pm\|_{L^2(\bbr_\pm)} +  \|u_0-u_\pm\|_{L^2(\bbr_\pm)} \Big) + \| v_{0x}\|_{L^2(\bbr)}+ \|u_{0x}\|_{L^2(\bbr)} < \eps_0,
\end{equation}
where $\bbr_- :=-\bbr_+ = (-\infty,0)$.\\
 Then,  the compressible Navier-Stokes system \eqref{NS} admits a unique global-in-time solution $(v,u)$. Moreover, there exists an absolutely continuous shift $\mb{X}(t)$ such that
 \begin{align}
\begin{aligned}\label{ext-main}
&v(t,x)- \Big( v^r(\frac xt)+\wt v^S(x-\s t-\mb X(t))-v_m \Big) \in C(0,+\infty;H^1(\bbr)),\\
& u(t,x)- \Big( u^r(\frac xt)+\wt u^S(x-\s t-\mb X(t))-u_m \Big) \in C(0,+\infty;H^1(\bbr)),\\
%& u_{xx}(t,x) \in L^2(0,T; L^2(\bbr)),\quad \mbox{for any } T>0.
& u_{xx}(t,x)-\wt u^S_{xx}(x-\s t-\mb X(t))\in L^2(0,+\infty; L^2(\bbr)).
\end{aligned}
\end{align}
In addition, as $t\to+\infty$,
\begin{equation}\label{con}
\sup_{x\in\bbr}\Big|(v,u)(t,x)-  \Big(v^r(\frac xt)+\wt v^S(x-\s t-\mb X(t))-v_m,u^r(\frac xt)+\wt u^S(x-\s t-\mb X(t))-u_m \Big) \Big| \to 0,
\end{equation}
and
\begin{equation}\label{as}
\lim_{t\rightarrow+\infty} |\dot {\mb X}(t) |=0.
\end{equation}
\end{theorem}

\begin{remark}
Theorem \ref{thm:main} states that if the two far-field states $(v_\pm, u_\pm)$ in \eqref{in} are connected by the superposition of shock  and rarefaction waves, then the solution to the compressible Navier-Stokes equations \eqref{NS} tends to the composition wave of the self-similar rarefaction wave and the viscous shock wave with the shift $\mb{X}(t)$, which solves the open problem proposed by Matsumura-Nishihara \cite{MN} since 1992.
\end{remark}

\begin{remark}
The shift function $\mb{X}(t)$ (defined in \eqref{X(t)}) is proved to satisfy the time-asymptotic behavior \eqref{as}, which implies that 
$$
\lim_{t\rightarrow+\infty}\frac{{\mb X}(t)}{t}=0,
$$
that is, the shift function ${\mb X}(t)$ grows at most sub-linearly w.r.t. the time $t$ and the shifted viscous shock wave still keeps the original traveling wave profile time-asymptotically.

\end{remark}

\begin{remark}
Note that our result in Theorem \ref{thm:main} also holds true in the case of a single viscous shock, that is, $\deltar\equiv0$. In this case, Theorem \ref{thm:main} provides an alternative proof for stability of a single viscous shock. Our proof is far simpler than the ones of Masica-Zumbrun \cite{MZ}, or Liu-Zeng \cite{LZ}. Moreover, our approach does not have the disadvantages of the anti-derivative method, as the necessity to consider zero mass initial perturbations for instance. This simplification is what allows us to consider the combination of waves of different kinds. Therefore, our approach follows exactly the suggestion of Matsumura in \cite[Section 4.2, page 2540]{MatsumuraBook} to find a simpler proof, for the stability of viscous shock, than the ones in \cite{MZ} or \cite{LZ}, in order to attack many other open problems.  Note however, that our simplification comes at the cost of  less precise information, especially on the shift ${\mb X}(t)$. 
\end{remark}

\begin{remark}
The extension of Theorem \ref{thm:main}  to general smooth viscosity function $\mu=\mu(v)>0$ and general pressure function $p(v) > 0$ satisfying $p'(v) < 0, p''(v) > 0$ follows without meaningful added difficulties, since we consider small $H^1$-perturbations. For the sake of clarity, and to simplify slightly the arguments, we made the choice to write the paper in the physical relevant context of constant viscosities and power  pressure laws. 
\end{remark}

The main new ingredient of our proof is the use of the method of $a$-contraction with shifts  \cite{Kang-V-NS17} to track the stability of the viscous shock. The method is based on the relative entropy introduced by Dafermos \cite{Dafermos4} and DiPerna \cite{DiPerna}. It is  energy based, and so meshes seamlessly with the treatment of the rarefaction.  
\vskip0.3cm
\noindent
{\bf The method of $a$-contraction with shifts:}
The method of $a$-contraction with shifts was developed in \cite{KVARMA} (see also \cite{LV}) to study the stability of extremal shocks for inviscid system of conservation laws, as for example, the Euler system \eqref{E}.  
Consider the entropy of the system (which is actually the physical energy) defined for any state $U=(v,u)$ as:
$$
\eta(U)=\frac{u^2}{2} + Q(v),\qquad Q(v)=\frac{1}{(\gamma-1) v^{\gamma-1}}.
$$
We then consider the relative entropy defined in \cite{Dafermos4} for any two states $U=(v,u)$, $\overline{U}=(\bar v,\bar u)$:
$$
\eta(U|\overline U)=\frac{|u-\overline{u}|^2}{2}+ Q(v|\bar{v}), \qquad  Q(v|\bar{v})=Q(v)-Q(\bar{v})-Q'(\bar v)(v-\bar v).
$$
Note that $Q$ is convex, and so $\eta(U|\overline U)$ is nonnegative and equal to zero if and only if $U=\overline U$. Therefore $\eta(U|\overline U)$ can be used as a pseudo-distance between $U$ and $\overline U$. 
It can be shown that rarefactions $\overline U$ (that is solutions to  \eqref{BESintro}-\eqref{rareintro}) have a contraction property for this pseudo-metric (see for instance \cite{Vasseur_Book}). Indeed, for any weak entropic solution $U$ to \eqref{E},  it can be shown that
$$
\frac{d}{dt} \int_\R \eta(U|\overline U)\,dx\leq 0.
$$
 The contraction property is not true  if $\overline U$ is a shock (that is traveling waves  \eqref{BERintro} verifying the Rankine-Hugoniot conditions \eqref{RH}). However, the contraction property can be recovered  up to a shift, after weighting the relative entropy (see  \cite{KVARMA}). Indeed, there exists weights $a_-,a_+>0$ (depending only on the shock $\overline U$) such that for any weak entropic solution $U$ of \eqref{E} (verifying a mild condition called strong trace property)  there exists a Lipschitz  shift function $t\to X(t)$ such that 
 $$
\frac{d}{dt}\left\{ a_-\int_{-\infty}^{X(t)}\eta(U|\bar{U})\,dx+a_+\int_{X(t)}^\infty\eta(U|\bar U)\,dx\right\}\leq 0.
 $$
This was first proved in the scalar case by Leger \cite{Leger} for $a_-=a_+$. It has been shown in \cite{Serre-Vasseur} that the contraction with $a_-=a_+$ is usually false for most systems. Therefore the weighting via the coefficients $a_-,a_+$ is essential. Note that in the case of the full Euler system, the a-contraction property up to shifts is  true for all the single wave patterns, including the 1-shocks, 3-shocks (see \cite{Vasseur-2013}), and  the 2-contact discontinuities (see \cite{SV-16}). Although the $a$-contraction property with shifts holds for general extremal shocks, it is not always true for intermediate shocks  (see \cite{Kang} for instance). 
\vskip0.3cm
The first extension of the method  to viscous models was done in the 1D scalar case \cite{Kang-V-1}  (see also \cite{Kang19}) and then in the multi-D case \cite{KVW}. The case of the barotropic Navier-Stokes equation \eqref{NS} was treated in \cite{Kang-V-NS17}. The $a$-contraction property takes place in variables associated to the BD entropy (see \cite{BD-06}): $U=(v,h)$, where $h$ is the effective velocity defined as $h=u-(\mathrm{ln} \ v)_x$. In these variables, system \eqref{NS} with $ \mu=1$ is transformed as 
\begin{align}
\begin{aligned}\label{hNS-1intro}
\left\{ \begin{array}{ll}
 v_t -h_x= (\ln v)_{xx}, \\ % -\sigma v_x
        h_t+p(v)_x=0. % -\sigma h_x
        \end{array} \right.
\end{aligned}
\end{align}
 The only nonlinear term of the hyperbolic system \eqref{E} is the pressure which is a function of $v$. The system \eqref{hNS-1intro} is then better than \eqref{NS} since the diffusion is in the variable $v$ corresponding to the nonlinear term $p(v)$. It was shown in \cite{Kang-V-NS17} that  there exists a monotonic function $x\to a(x)$ (with limits $a_\pm$ at $\pm\infty$), depending only on the viscous shock $\bar{U}=(\bar v,\bar h)$ solution to \eqref{VSintro} (in the $(v,u)$ variables), such that for any solution $U$ to \eqref{hNS-1intro}, there exists a shift function $t\to \mb{X}(t)$ with
 $$
 \frac{d}{dt}\int_\R a(x-\mb{X}(t))\eta(U(t,x)|\bar U(x-\mb{X}(t)))\,dx \leq 0.
$$
\vskip0.3cm
The strategy of this paper is to apply the $a$-contraction method to the composite wave made of a shock wave and a  rarefaction wave. The weight function $a$ and the shift $\mb{X}(t)$ is applied only on the  shock wave. The combination of the viscous shock and the rarefaction is not an exact solution to \eqref{NS}. This introduces some errors that can be controlled thanks to the separation of the waves. Because of the shift, the separation of the waves is not automatic. We  show, however, that it is still true, and  that the shock cannot artificially stick to the rarefaction. This  provides an "almost" $a$-contraction in the effective variables $(v, h)$.  We then recover the classical control on the $H^1$ norm of the perturbation in the classical variables $(v,u)$.

\vskip0.3cm
\noindent
{\bf The $a$-contraction with shift theory for a small viscous shock:} 
Note that the $a$-contraction result of   \cite{Kang-V-NS17} provides a uniform stability for viscous shocks with respect to the strength of the viscosity. This is used in  \cite{KV-Inven} to obtain the stability of inviscid shocks of \eqref{E}
among any inviscid limits of \eqref{NS}. Since the conjecture of Matsumura \cite{MatsumuraBook} does not mention the uniform stability with respect to the viscosity, we choose in this paper to restrict ourselves to the classical framework and show the stability with $\mu=1$ fixed. This allows us to simplify some of the arguments of  \cite{Kang-V-NS17} in this context. The method is even more powerful in this restricted framework and should be developed in the foreseeable future to a  large family of problems. Let us describe  the fundamental ideas in this context. 
\vskip0.3cm
\noindent{\it A Poincar\'e type inequality and the scalar case:}
At its core,  the method of $a$-contraction with shift in the viscous cases relies on the following Poincar\'e type inequality (see \cite[Lemma 2.9]{Kang-V-NS17}).  
\begin{lemma}\label{lem-poin}
For any $f:[0,1]\to\bbr$ satisfying $\int_0^1 y(1-y)|f'|^2 dy<\infty$, 
\beq\label{poincare}
\int_0^1\Big|f-\int_0^1 f dy \Big|^2 dy\le \frac{1}{2}\int_0^1 y(1-y)|f'|^2 dy.
\eeq
\end{lemma}
The eigenfunctions of the associated Euler-Lagrange equation to this minimization problem are the Legendre polynomials, and their eigenvalues are given explicitly.  As a consequence, the inequality is sharp. The weighted $H^1$ norm of this inequality comes naturally when considering the following Burgers equation (see \cite{Kang-V-1}):
\begin{equation}\label{burgers}
\partial_tu+\partial_x(u(1-u))=\partial_x^2u,
\end{equation}
and its viscous shock profile $\tilde u$ defined as
$$
\partial_x(\tilde u(1-\tilde u))=\partial_x^2\tilde u, \qquad \lim_{x\to-\infty}\tilde u(x)=0,\qquad  \lim_{x\to+\infty}\tilde u(x)=1.
$$
This shock does not depend on time (it is a stationary wave). Integrating in $x$, and denoting $\tilde u'=\partial_x \tilde u$ gives
\begin{equation}\label{eqtilde}
\tilde u'(x) =\tilde u(x)(1-\tilde u(x)).
\end{equation}
Consider now the relative entropy associated to the entropy $\eta(u)=u^2/2$ between a generic solution $u$ of \eqref{burgers} and the shifted shock $\tilde u^{-\mb{X}}(t,x)=\tilde{u}(x-\mb{X}(t))$ for an arbitrary shift $\mb{X}(t)$:
$$
\eta(u|\tilde u^{-\mb{X}})(t,x)= \frac{|u(t,x)-\tilde u(x-\mb{X}(t))|^2}{2}.
$$
The shifted shock verifies the equation
$$
\partial_t[\tilde u^{-\mb{X}}]+\dot{\mb{X}}  (\tilde u')^{-\mb{X}}+ \partial_x(\tilde u^{-\mb{X}}(1-\tilde u^{-\mb{X}}))=\partial_x^2[\tilde u^{-\mb{X}}].
$$
Multiplying the difference of \eqref{burgers} and the shifted shock equation by $(u-\tilde u^{-\mb{X}})$, we can show that
\begin{equation}\label{inburgers}
\frac{d}{dt} \int_\R \eta(u|\tilde u^{-\mb{X}})(t,x)\,dx=\dot{\mb{X}}(t)\int_\R \tilde u'(u^\mb{X}-\tilde u)\,dx +\int_\R \tilde u' |u^\mb{X}-\tilde u|^2-\int_\R |\partial_x (u^\mb{X}-\tilde u)|^2\,dx.
\end{equation}
Note that, at the final step,  we  made the change of variable  $x\to x+\mb{X}(t)$ flipping the shift from the shock $\tilde u$ to the function $u^\mb{X}(t,x)=u(t,x+\mb{X}(t))$. 
We now fix the speed of the shift as 
$$
\dot{\mb{X}}(t)=- \int_\R \tilde u'(u^\mb{X}-\tilde u)\,dx,
$$
which defines the shift $t\mapsto \mb{X}(t)$ thanks to the Cauchy-Lipschitz theorem. 
We claim that,  for  this shift, $\int_\R \eta(u|\tilde u^{-\mb{X}})(t,x)\,dx$ is non-increasing in time.  This statement will be proved, if we can show that for any function $g\in H^1(\R)$:
\begin{equation}\label{eqgeneral}
-\bar g^2+\int_\R \tilde u'(x) |g(x)|^2\,dx-\int_\R |g'(x)|^2\,dx\leq 0,
\end{equation}
where $\bar g=\int_\R \tilde u'(x)g(x)\,dx$.
Indeed, for any fixed time $t>0$, denote $g(x)=(u^\mb{X}-\tilde u)(t,x)$. The inequality \eqref{eqgeneral} for this specific function $g$ applied to \eqref{inburgers} shows that at all these times:
$$
\frac{d}{dt} \int_\R \eta(u|\tilde u^{-\mb{X}})(t,x)\,dx\leq 0.
$$
Therefore, the contraction up to a shift is reduced to the Poincar\'e type inequality \eqref{eqgeneral}. Because $\int_\R \tilde u'\,dx=1$, it is equivalent to 
$$
\int_\R \tilde u'(x) |g(x)-\bar g|^2\,dx-\int_\R |g'(x)|^2\,dx\leq 0.
$$

Let us rewrite this inequality in the natural variable associated to the shock: 
$$
y=\tilde u(x), \qquad dy=\tilde u'(x) \,dx, \qquad f(y)=g(x).
$$
This change of variable is possible since $\tilde u$ is an increasing function from 0 to 1.
We have also
$$
g'(x)=\tilde u'(x) f'(y),\qquad \bar g=\int_0^1 f(y)\,dy,
$$
and so \eqref{eqgeneral} is equivalent to 
$$
\int_0^1\Big|f-\int_0^1 f dy \Big|^2 dy\le \int_0^1 \tilde u'(x)|f'|^2 dy.
$$
But thanks to \eqref{eqtilde}, $\tilde u'(x) =\tilde u(x)(1-\tilde u(x))=y(1-y)$. Hence  \eqref{poincare} implies \eqref{eqgeneral} since $1/2<1$.
\vskip0.3cm
\noindent {\bf The case of  the Navier-Stokes system:} If we perform the same idea on the Navier-Stokes system  \eqref{hNS-1intro} in the BD effective variables $U=(v,h)$, but without weight function $a$, we are obtaining (after Taylor expansion, using the smallness of the shock and of the perturbation) the inequality
$$
\frac{d}{dt}\int_\R \eta(U|\tilde U^{-\mb{X}})\,dx\approx \dot{\mb{X}}(t) \mb{Y}(U) +\int_\R \partial_x[p'(\tilde v)] |v^\mb{X}-\tilde v|^2 \,dx-\int_\R \frac{1}{v}|\partial_x(v^\mb{X}-\tilde v)|^2\,dx,
$$
with 
$$
\mb{Y}(t)\approx \int_\R \partial_x(\nabla \eta(\tilde U))\cdot(U^\mb{X}-\tilde U)\,dx.
$$
Thanks to the BD effective variables, the first equality is very similar to the scalar case. Especially, the dissipation is in the $v$  variable only, as the ``bad" quadratic term. However, the $\mb{Y}$ term involves now a linear combination of $v^\mb{X}-\tilde v$ and $h^\mb{X}-\tilde h$. Therefore, whatever the choice of $\dot{\mb{X}}$, we cannot control any weighted mean value of $v^\mb{X}-\tilde v$ from this term as in the scalar case.
\vskip0.3cm
The point of the method is that the flux of the relative entropy (which disappears when integrating in $x$) is better behaved. On top of a ``bad" quadratic term in $|v^\mb{X}-\tilde v|^2$, it involves a  ``good" (meaning with a good sign) quadratic term involving  a linear combination of $v^\mb{X}-\tilde v$ and $h^\mb{X}-\tilde h$.  The weight function $a$ is used to activate those flux terms. 
Note that the  linear combination involved in the flux terms  is independent of the linear combination involved in the $\mb{Y}(t)$ term. Therefore the use of both the weight and the shift allows to control the  weighted mean value of  $v^\mb{X}-\tilde v$ needed to use the Poincar\'e Lemma \ref{lem-poin}. The weight function $a$ is chosen such that $\partial_xa$ is proportional to $\partial_x[p'(\tilde v)] $ which is the analogue of $\tilde u'$ for the scalar case, and is a natural weight associated to the shock layer. Its strength, however, is enhanced by a factor bigger than the size of the shock $ \lambda\gg \delta$, in order to make the relative entropy flux term dominant.

%Finally let us mention  two recent   applications of the  $a$-contraction with shift theory in the multi-dimension settings \cite{KVW, KVW2}. 

\vskip0.3cm
The rest of the paper is organized as follows. We begin with preliminaries in Section \ref{sec-preli}. It includes known properties on the rarefaction and on the viscous shock, together with simple properties on the behavior of the 
pressure functional and the relative entropy. % We also include a simple Poincar\'e inequality used in the theory of $a$-contraction for viscous models.  
The general set up is laid out in Section \ref{sec-thm}. We introduce an a priori estimates result  in Proposition \ref{prop2}. Then we show by a continuing argument how this proposition implies Theorem \ref{thm:main}. The last two sections are dedicated to the proof of Proposition \ref{prop2}. The $a$-contraction argument is set up in Section \ref{sec-acontraction} where  global a priori estimates are proved in the variables $(v,h)$. From these global estimates, we deduce global a priori estimates in the variables $(v,u)$ in Section \ref{sec-vu}, concluding the proof of Proposition \ref{prop2}.

\section{Preliminaries} \label{sec-preli}
\setcounter{equation}{0}
We gather in this section some well-known results which will be useful in the rest of the paper.

\subsection{Relative quantities}

As explained in the introduction, the $a$-contraction with shifts theory is based on the relative entropy, and the specific volume variable $v\in \R^+$ is of particular importance. 
 For any function $F$ defined on  $\R^+$, we define the associated relative quantity defined for  $v,w\in \R^+$ as 
$$
F(v|w)=F(v)-F(w)-F'(w)(v-w).
$$
We gather, in the following lemma, useful explicit inequalities on the relative quantities associated to the pressure $p(v)=v^{-\gamma}$, and the internal energy $Q(v)=v^{1-\gamma}/(1-\gamma)$. The proofs are simply based on Taylor expansions, and can be found in  \cite[Lemmas 2.4, 2.5 and 2.6]{Kang-V-NS17}.

\begin{lemma}\label{lem-useful}
For given constants $\gamma>1$, and $v_->0$, their exist constants $C, \delta_*>0$, such that the following holds true.\\
1) For any $v,w$ such that  $0<w<2v_-, 0<v\le 3v_-$,
\beq\label{rel_Q}
 |v-w|^2\le C Q(v|w),
 \eeq
% and 
 \beq\label{rel_p}
 |v-w|^2\le C p(v|w).
 \eeq
2) For any $v, w>v_-/2$,
\beq\label{pressure2}
|p(v)-p(w)| \le C |v-w|.
\eeq
3) For any $0<\delta<\delta_*$, and for any $(v, w)\in \bbr_+^2$ 
satisfying $|p(v)-p(w)|<\delta$, and  $|p(w)-p(v_-)|<\delta$, the following holds true:
\begin{align}
\begin{aligned}\label{p-est1}
p(v|w)&\le \bigg(\frac{\gamma+1}{2\gamma} \frac{1}{p(w)} + C\delta \bigg) |p(v)-p(w)|^2,
\end{aligned}
\end{align}
\beq\label{Q-est11}
Q(v|w)\ge \frac{p(w)^{-\frac{1}{\gamma}-1}}{2\gamma}|p(v)-p(w)|^2 -\frac{1+\gamma}{3\gamma^2} p(w)^{-\frac{1}{\gamma}-2}(p(v)-
p(w))^3,
\eeq
\beq\label{Q-est1}
Q(v|w)\le \bigg( \frac{p(w)^{-\frac{1}{\gamma}-1}}{2\gamma} +C\delta  \bigg)|p(v)-p(w)|^2.
\eeq
\end{lemma}

\subsection{Rarefaction wave}

We now recall important properties of the 1-rarefaction waves. Consider a $(v_m, u_m)$ in \eqref{in}, and $(v_-,u_-)\in R_1(v_m, u_m)$.
%Since we are concerned with the time-asymptotic stability of the superposition of the rarefaction wave and viscous shock wave, that is, $(v_\pm, u_\pm)$ in \eqref{in} satisfies that 
%\begin{equation}
%(v_-,u_-)\in R_1-S_2 (v_+, u_+),
%\end{equation}
%i.e,, there exists a unique intermediate state $(v_m, u_m)\in \bbr^+\times\bbr$ such that $(v_-,u_-)\in R_1(v_m, u_m)$ and $(v_m,u_m)\in S_2 (v_+, u_+)$ where the rarefaction wave curve $R_1(v_m, u_m)$ and the shock wave curve $S_2 (v_+, u_+)$ are defined in \eqref{(1.9)} and \eqref{(1.8)}, respectively.
Set $w_- = \lambda_1(v_-), w_m = \lambda_1(v_m)$, and
consider the Riemann problem for the inviscid Burgers equation:
\begin{equation} \label{BE}
\begin{cases}
\displaystyle w_t + ww_x = 0, \\
\displaystyle w(0, x) = w_0^r(x) = \begin{cases}
w_-,  \quad x < 0,\\
w_m,  \quad x > 0.
\end{cases}
\end{cases}
\end{equation}
If $w_- < w_m$, then \eqref{BE} has the rarefaction wave fan $w^r(t, x) = w^r(x/t)$ given by
\begin{equation} \label{BES}
w^r(t, x) = w^r(\frac{x}{t}) = \begin{cases}
w_- , \qquad x < w_-t, \\
\frac{x}{t}, \qquad w_-t \leq x \leq w_mt, \\
w_m , \qquad x > w_mt.
\end{cases}
\end{equation}
 It is easy to check that the 1-rarefaction wave $(v^r, u^r)(t, x) = (v^r, u^r)(x/t)$ to the Riemann problem \eqref{E}-\eqref{Ei}, defined in \eqref{BESintro}-\eqref{rareintro},  is given explicitly by
\begin{equation}\label{rare}
\begin{array}{ll}
  &\lambda_1(v^r(\frac xt)) = w^r(\frac xt), \\[1mm]
  &z_1(v^r(\frac xt), u^r(\frac xt))=  z_1 (v_-, u_-)=z_1 (v_m,u_m).
\end{array}
\end{equation}
%where $z_1(v,u)=u-\int^v \lambda_1(s)ds$ is 1-Riemann invariant to the Euler equations \eqref{E}.
 The self-similar 1-rarefaction wave $(v^r, u^r)(x/t)$ is Lipschitz continuous and satisfies the Euler system a.e. for $t>0$,
\begin{align}
\begin{aligned}\label{E-ra}
\left\{ \begin{array}{ll}
        v^r_t- u^r_x =0,\\[1mm]
       u^r_t  +p(v^r)_x= 0. 
        \end{array} \right.
\end{aligned}
\end{align}
Let $\delta_R:=|v_m-v_-|$ denote the strength of the rarefaction wave. Notice that $\delta_R\sim |u_m-u_-|$ by $\eqref{rare}_2$.

\subsection{Viscous shock wave}
We turn to the 2-viscous shock wave connecting $(v_m, u_m)$ and $(v_+,u_+)$ such that $(v_m, u_m)\in S_2(v_+,u_+)$. Recall the Rankine-Hugoniot condition \eqref{RH}
%\begin{equation}\label{RH}
%\begin{array}{ll}
%\di -\sigma(v_+-v_m)-(u_+-u_m)=0,\\
%\di -\sigma(u_+-u_m)+(p(v_+)-p(v_m))=0,
%\end{array}
%\end{equation}
and the Lax entropy condition
\begin{equation}\label{ec}
\lambda_2(v_+)<\sigma<\lambda_2(v_m).
\end{equation}
The Riemann problem \eqref{E}-\eqref{Ei} admits a unique 2-shock solution
\begin{align}
\begin{aligned}\label{Shock}
(v^s,u^s)(t,x)=\left\{ \begin{array}{ll}
        (v_m,u_m),\quad x<\sigma t,\\[2mm]
       (v_+,u_+),\quad x>\sigma t. \\
        \end{array} \right.
\end{aligned}
\end{align}

By \eqref{RH}, it holds that
\begin{equation}
\sigma=\sqrt{-\frac{p(v_+)-p(v_m)}{v_+-v_m}}.
\end{equation}

By introducing a new variable $\xi=x-\sigma t$, the the 2-viscous shock wave $(\wt v^S, \wt u^S)(\xi)$ satisfies the ODE
\begin{equation}\label{VS}
\left\{
\begin{array}{ll}
\di -\sigma (\wt v^S)^\prime-( \wt u^S)^\prime=0,\qquad\qquad ^\prime=\frac{d}{d\xi},\\[3mm]
\di -\sigma( \wt u^S)^\prime+(p( \wt v^S))^\prime=\Big(\frac{( \wt u^S)^\prime}{ \wt v^S}\Big)^\prime,\\[3mm]
\di ( \wt v^S, \wt u^S)(-\infty)=(v_m, u_m), \qquad ( \wt v^S, \wt u^S)(+\infty)=(v_+, u_+).
\end{array}
\right.
\end{equation}

The properties of the 2-viscous shock wave $( \wt v^S, \wt u^S)(\xi)$ can be listed as follows. The proof of this lemma can be found in \cite{MN-S} or \cite{G} (see also \cite{Kang-V-NS17}).

\begin{lemma} \label{lemma1.3}
For any state  $(v_+,u_+)$, there exists a constant $C>0$ such that the following is true. For any end state such that $(v_m, u_m)\in R_2(v_+,u_+)$, there exists a unique solution  $( \wt v^S, \wt u^S)(\xi)$ to \eqref{VS}. Let $\delta_S$ denote the strength of the shock as $\delta_S:=|p(v_+)-p(v_m)|\sim|v_+-v_m|\sim |u_+-u_m|$. It holds that
	$$
	 \wt u^S_\xi<0, \qquad  \wt v^S_\xi >0,
	$$
	and
\begin{align*}
& |\wt v^S(\xi) -v_m|\leq C\deltas\ e^{-C\delta_S |\xi|}, \quad \xi<0,\\[1mm]
& |\wt v^S(\xi) -v_+|\leq C\deltas\ e^{-C\delta_S |\xi|}, \quad \xi>0,\\[1mm]
    &|( \wt v^S_\xi, \wt u^S_\xi)|\leq C\deltas^2\ e^{-C\delta_S |\xi|}, \quad\forall\xi\in\bbr,\\[1mm]
    &|( \wt v^S_{\xi\xi}, \wt u^S_{\xi\xi})|\leq C\delta_S |( \wt v^S_\xi, \wt u^S_\xi)|, \quad \forall\xi\in\bbr.
  \end{align*}		
\end{lemma}

\subsection{Composite waves of viscous shock and rarefaction}
Given the end states $(v_\pm, u_\pm)\in\bbr^+\times\bbr$ in \eqref{in}, we consider the case that there exists a unique intermediate state $(v_m, u_m)$ such that 
\begin{equation}
(v_-,u_-)\in R_1(v_m, u_m), \qquad  (v_m, u_m)\in S_2 (v_+, u_+).
\end{equation}
%that is, there exists a unique intermediate state $(v_m, u_m)\in \bbr^+\times\bbr$ such that $(v_-,u_-)\in R_1(v_m, u_m)$ and $(v_m,u_m)\in S_2 (v_+, u_+)$.
We will consider a superposition wave:
\beq\label{SW}
\left( v^r(\frac xt)+\wt v^S(x-\s t)-v_m, u^r(\frac xt)+\wt u^S(x-\s t)-u_m \right),
\eeq
where $( v^r,  u^r)(\frac xt)$ is the 1-rarefaction wave defined in \eqref{rare} and $( \wt v^S, \wt u^S)(\xi)$ is the 2-viscous shock wave defined in Lemma \ref{lemma1.3}.

%\subsection{Ideas of Proof}
%Define a shift function such that
%\begin{equation}\label{xa}
%|\dot{\mb{X}}(t)|\sim  |\int_0^1{\rm w} dy|.
%\end{equation}

\section{Set-up of the problem, and proof of Theorem \ref{thm:main}}\label{sec-thm}
\setcounter{equation}{0}

\subsection{Construction of approximate rarefaction wave}
As in \cite{MN86}, we will consider a smooth approximate solution of the 1-rarefaction wave, by using the smooth solution to the Burgers equation:
\begin{equation} \label{ABE}
\begin{cases}
\displaystyle w_t + ww_x = 0, \\
\displaystyle w(0, x) = w_0(x) = \frac{w_m + w_-}{2} + \frac{w_m - w_-}{2} \tanh x.
\end{cases}
\end{equation}
Then, by the characteristic methods, the solution $w(t, x)$ of the problem \eqref{ABE} has the following properties and their proofs can be found in \cite{MN86}.
\begin{lemma} \label{lemma2.1}
  Suppose $w_m > w_-$ and set $\tilde{w} = w_m - w_-$. Then the problem \eqref{ABE} has a unique smooth global solution $w(t, x)$ such that
   
  (1)~$w_- < w(t, x) < w_m, ~w_x >0$ for $x \in \bbr$ and $t \geq 0$.
   
  (2)~The following estimates hold for all $t > 0$ and $p \in [1, + \infty]$:

  \begin{align*}
    &\|w_x(t, \cdot)\|_{L^p} \leq C \min(|\tilde{w}|, |\tilde{w}|^{1/p}t^{-1+1/p}), \\
    &\|w_{xx}(t, \cdot)\|_{L^p} \leq C \min(|\tilde{w}|,t^{-1}). 
   % &\|w_{xxx}(t, \cdot)\|_{L^p} \leq C \min(|\tilde{w}|,t^{-1}).
  \end{align*}
  
  (3)~ If $w_m<0,$ then it holds that $\forall x\geq 0, \forall t\geq 0$,
   \begin{align*}
    &|w(t, x)-w_m| \leq \tilde{w} e^{-2(|x|+|w_m|t)}, \\
    &|(w_x, w_{xx})(t, x)|\leq C\tilde{w} e^{-2(|x|+|w_m|t)}.   
  \end{align*}
  
(4)~  It holds that $\forall x\le w_-t,\,  \forall t\geq 0$,
   \begin{align*}
    &|w(t, x)-w_-| \leq \tilde{w}e^{-2|x-w_-t|}, \\
    &|(w_x, w_{xx})(t, x)|\leq C\tilde{w} e^{-2|x-w_-t|}.
  \end{align*}

  (5)~$\di\lim_{t \to +\infty} \sup_{x \in \bbr} |w(t, x) - w^r(\frac xt)| = 0$.
\end{lemma}

We now construct the smooth approximate 1-rarefaction wave $(\wt v^R, \wt  u^R)(t, x)$ of the 1-rarefaction wave fan $(v^r, u^r)(\frac xt)$ by
\begin{align}
  \begin{aligned} \label{AR}
   &\lambda_1(v_-) = w_-,\ \lambda_1(v_m) = w_m,\\
    &\lambda_1(\wt   v^R)(t, x) = w(1+t, x), \\
    &z_1(\wt   v^R, \wt   u^R)(t, x) = z_1 (v_-, u_-)=z_1 (v_m,u_m),
  \end{aligned}
\end{align}
where $w(t,x)$ is the smooth solution to the Burgers equation in \eqref{ABE}.
One can easily check that the above approximate rarefaction wave $(\wt   v^R,\wt   u^R)$ satisfies the system:
\begin{equation}  \label{ARW}
  \begin{cases}
    \displaystyle \wt   v^R_t -\wt u^R_x = 0, \\
    \displaystyle \wt  u^R_t + p(\wt  v^R)_x = 0. \\
  \end{cases}
\end{equation}
The following lemma comes from Lemma \ref{lemma2.1} (cf. \cite{MN86}).
\begin{lemma} \label{lemma1.2}
	The smooth approximate 1-rarefaction wave $(\wt  v^R, \wt  u^R)(t,x)$ defined in \eqref{AR} satisfies the following properties. Let $\delta_R$ denote the rarefaction wave strength as $\delta_R := |v_m - v_-|\sim |u_m-u_-|$.\\
	
	(1)~$ \wt  u^R_x = \frac{2}{(\gamma + 1)\wt  v^R}w_x > 0$ and $ \wt  v^R_x = \frac{( \wt  v^R)^{\frac{\gamma+1}{2}}}{\sqrt\gamma} \wt  u^R_x>0,$ for all $x \in \bbr$ and $t \geq 0$.
	
	(2)~The following estimates hold for all $t \geq 0$ and $p \in [1, + \infty]$:
	\begin{align*}
	&\|( \wt  v^R_x, \wt  u^R_x)\|_{L^p} \leq C \min\{\delta_R, \delta_R^{1/p}(1+t)^{-1+1/p}\}, \\
	&\|( \wt  v^R_{xx}, \wt  u^R_{xx})\|_{L^p} \leq C \min\{\delta_R, (1+t)^{-1}\}, \\
	&|\wt  u^R_{xx}| \leq C |\wt  u^R_{x}| ,\quad \forall x\in\bbr.
	\end{align*}
	
	(3)~ For $x\geq 0, t\geq 0,$ it holds that
	\begin{align*}
    &|( \wt  v^R, \wt  u^R)(t, x)-(v_m,u_m)| \leq C\delta_R \ e^{-2(|x|+|\lambda_1(v_m)|t)}, \\
    &|( \wt  v^R_x, \wt  u^R_x)(t, x)|\leq C\delta_R\ e^{-2(|x|+|\lambda_1(v_m)|t)}.
  \end{align*}
  
  (4)~ For $x\le\lambda_1(v_-) t$ and $ t\geq 0,$ it holds that
	\begin{align*}
    &|( \wt  v^R, \wt  u^R)(t, x)-(v_-,u_-)| \leq C\delta_R \ e^{-2|x-\lambda_1(v_-)t|}, \\
    &|( \wt  v^R_x, \wt  u^R_x)(t, x)|\leq C\delta_R\ e^{-2|x-\lambda_1(v_-)t|}.
  \end{align*}
  
(5)~$\di \lim_{t \to +\infty} \sup_{x \in \bbr} |( \wt  v^R, \wt  u^R)(t, x) - (v^r, u^r)( \frac xt)| = 0$.
\end{lemma}

\subsection{Local in time estimates on the solution} 
For simplification of our analysis, we rewrite the compressible Navier-Stokes system \eqref{NS} into the following system, based on the change of variable associated to the speed of propagation of the shock $(t,x)\mapsto (t, \xi=x-\s t)$: 
 \begin{align}
\begin{aligned}\label{NS-1}
\left\{ \begin{array}{ll}
        v_t-\sigma v_\xi- u_\xi =0,\\
       u_t -\sigma u_\xi +p(v)_\xi=  (\frac{u_\xi}{v})_\xi. \\
        \end{array} \right.
\end{aligned}
\end{align}
We will consider stability of the solution to \eqref{NS-1} around the superposition wave of the approximate rarefaction wave and  the viscous shock wave shifted by $\mb X(t)$ (to be defined in \eqref{X(t)}) : 
\beq\label{shwave}
(\wt v_{-\mb X}, \wt u_{-\mb X}) (t,\xi) := \left(\wt v^R(t,\xi+\sigma t)+\wt v^S(\xi -\mb X(t))-v_m,\wt u^R(t,\xi+\sigma t)+\wt u^S(\xi-\mb X(t))-u_m \right) .
\eeq

For any initial $H^1$ perturbation of the superposition of waves \eqref{shwave}, there exists a global strong solution to \eqref{NS-1} (see for instance \cite{MV-sima}). We will use a standard method of continuation to show the global in time control of this perturbation. 
For that, we first recall local in time estimates for strong solutions to \eqref{NS} (and so also for \eqref{NS-1}). They can be found in  \cite{Solo} (see also \cite[Proposition 2.2]{MV-sima}).% we ensure the local existence of strong solutions for \eqref{NS} (thus, for \eqref{NS-1}) as follows.
\begin{proposition} \label{prop:soln}
Let $\underline v$ and $\underline u$ be smooth monotone functions such that
\beq\label{sm}
\underline v(x) = v_\pm \quad\mbox{and}\quad \underline u(x) = u_\pm\quad\mbox{for } \pm x \ge 1.
\eeq
For any constants $M_0, M_1,  \underline \kappa_0,  \overline \kappa_0, \underline\kappa_1, \overline\kappa_1$ with $M_1>M_0>0$ and $ \overline \kappa_1>\overline \kappa_0>  \underline \kappa_0>\underline\kappa_1>0$, there exists a constant $T_0>0$ such that if 
\begin{align*}
\begin{aligned}
&\|v_0-\underline v\|_{H^1(\bbr)} +\|u_0 -\underline u\|_{H^1(\bbr)}  \le M_0,\\
&0< \underline \kappa_0 \leq v_0(x)\leq  \overline \kappa_0, \qquad  \forall x\in \bbr, \\
\end{aligned}
\end{align*}
then \eqref{NS-1} has a unique solution $(v,u)$ on $[0,T_0]$ such that 
\begin{align*}
\begin{aligned}
&v -\underline v \in C([0,T_0];H^1(\bbr)), \\
& u -\underline u \in C([0,T_0];H^1(\bbr)) \cap  L^2(0,T_0;H^2(\bbr)).
\end{aligned}
\end{align*}
and
\[
\|v-\underline v\|_{L^\infty(0,T_0;H^1(\bbr))} + \|u-\underline u\|_{L^\infty(0,T_0;H^1(\bbr))} \le M_1.
\]
Moreover:
\beq\label{bddab}
\underline  \kappa_1 \le v(t,x) \le \overline  \kappa_1,\qquad \forall (t,x)\in [0,T_0]\times \bbr.
\eeq
\end{proposition}
%In fact, by \cite{MV-sima},  the local strong solution $(v,u)$ for \eqref{NS} in Proposition \ref{prop:soln} is a global one in any finite time interval $[0,T]$.

\subsection{Construction of shift} For the continuation argument, the main tool is the a priori estimates of  Proposition \ref{prop2}. These estimates depend on the shift function, and for this reason, we are giving its definition right now. The definition depends on the weight function  $a:\bbr\to\bbr$ defined in \eqref{weight}. For now, we will only use the fact that $\|a\|_{C^1(\bbr)}\leq 2$. 
We then define the shift  $\mb X$ as a solution to the ODE:
\begin{equation}\label{X(t)}
\left\{
\begin{array}{ll}
\di \dot{\mb{X}}(t)=-\frac{M}{\delta_S}\Big[\int_{\mathbb{R}}\frac{a(\xi-\mb{X})}{\sigma} \partial_\xi  \wt h^S(\xi-\mb{X}) (p(v)-p(\wt v_{-\mb X})) d\xi\\[4mm]
\di \qquad \qquad\qquad  -\int_{\mathbb{R}}a(\xi-\mb{X}) \partial_\xi p(\wt v^S(\xi-\mb{X})) (v-\wt v_{-\mb X}) d\xi\Big],\\
\di \mb X(0)=0,
\end{array}
\right.
\end{equation}
where $M$ is the specific constant chosen as $M:=\frac{5(\gamma+1)\sigma_m^3}{8\gamma p(v_m)}$ with $\s_m:=\sqrt{-p'(v_m)}$, which will be used in the proof of Lemma \ref{lem-sharp} (see \eqref{definitionM}). \\

The following lemma ensures that \eqref{X(t)} has a unique absolutely continuous solution defined on any interval in time $[0,T]$ for which   \eqref{bddab} is verified.
\begin{lemma}\label{lem:xex}
For any $c_1,c_2>0$, there exists a constant $C>0$ such that the following is true.  For any $T>0$, and any   function $v\in L^\infty ((0,T)\times \R)$ %such that for some positive constants $c_1,c_2, T$,
verifying
\beq\label{odes}
c_1 \le v(t,x)\le c_2,\qquad \forall (t,x)\in [0,T]\times \bbr,
\eeq
 the ODE \eqref{X(t)} has a unique absolutely continuous solution $\mb X$ on $[0,T]$. Moreover, %and there exists a constant $C>0$ such that
\beq\label{roughx}
|{\mb X}(t)| \le Ct,\quad \forall t\le T.
\eeq
\end{lemma}
\begin{proof}
We will use the following lemma as a simple adaptation of the well-known Cauchy-Lipschitz theorem.  \\

 \begin{lemma} \cite[Lemma A.1]{CKKV} \label{lem_ckkv}
 Let $p>1$ and $T>0$. Suppose that a function 
 $F:[0,T]\times\bbr\rightarrow\bbr$  satisfies 
 $$\sup_{x\in\bbr }|F(t,x)|\leq f(t) \quad\mbox{and}\quad
\sup_{x,y\in\bbr,x\neq y }\Big|\frac{F(t,x)-F(t,y)}{x-y}\Big|\leq g(t)
\quad \mbox{for } t\in[0,T] $$ for some functions $f \in L^1(0,T)$ and $\, g\in L^p(0,T)$. Then for any $x_0\in\bbr$, there exists a unique absolutely continuous function $\mb{X}:[0,T]\rightarrow \bbr$ satisfying
\beq\label{ode_eq}\left\{ \begin{array}{ll}
        \dot{\mb{X}}(t) = F(t,\mb{X}(t))\quad\mbox{for \textit{a.e.} }t\in[0,T],\\
       \mb{X}(0)=x_0.\end{array} \right.\eeq
 \end{lemma}

To apply the above lemma, let $F(t,\mb{X})$ denote the right-hand side of the ODE \eqref{X(t)}. \\
Then the sufficient conditions of the above lemma are verified thanks to the facts that $\|a\|_{C^1(\bbr)}\leq 2$, $\|\wt v^S\|_{C^2(\bbr)}\leq \max\{1,v_+\}$, and $\|\tilde v^S_\xi\|_{L^1} \le C\delta_S$.
Indeed, using \eqref{odes}, we find that for some constant $C>0$,
\beq\label{f1t}
\sup_{\mb{X}\in\bbr}|F(t,\mb{X})| \le \frac{C}{\deltas}  \||p(v)|+|p(\wt v_{-\mb{X}})|+|v|+|\wt v_{-\mb{X}}| \|_{L^\infty(\bbr)} \int_\bbr |\wt v^S_\xi| d\xi \le C,
\eeq
and
\begin{align*}
\begin{aligned}
\sup_{\mb{X}\in \bbr} |\partial_{\mb{X}}F(t,\mb{X})| &\le  \frac{C}{\deltas} \|a\|_{C^1}\||p(v)|+|p(\wt v_{-\mb{X}})|+|v|+|\wt v_{-\mb{X}}| \|_{L^\infty(\bbr)}  \int_\bbr |\wt v^S_\xi| d\xi \le C.
\end{aligned}
\end{align*}
Especially, since $|\dot{\mb X}(t)| \le C$ by \eqref{f1t}, we have \eqref{roughx}.
\end{proof}

\subsection{A priori estimates}
First, it follows from \eqref{ARW} that $(v,u)(t,\xi):=(\wt v^R(t,\xi+\sigma t),\wt u^R(t,\xi+\sigma t)) $ verifies 
\begin{equation}  \label{rarexi}
  \begin{cases}
    \displaystyle v_t -\s  v_\xi -u_\xi = 0, \\
    \displaystyle u_t -\s  u_\xi + p(v)_\xi = 0.
  \end{cases}
\end{equation}
Therefore, using \eqref{VS} and \eqref{rarexi} we find that the approximated combination of waves $(\wt v_{-\mb X}, \wt u_{-\mb X}) $ defined in \eqref{shwave} solves the system: 
\begin{align}
\begin{aligned}\label{sws-u}
\left\{ \begin{array}{ll}
       \di (\wt v_{-\mb X})_t-\sigma(\wt v_{-\mb X})_\xi+\dot{\mb{X}}(t)(\wt v^S)^{-\mb X}_\xi-(\wt u_{-\mb X})_\xi =0,\\[2mm]
      \di (\wt u_{-\mb X})_t-\sigma(\wt u_{-\mb X})_\xi+\dot{\mb{X}}(t) (\wt u^S)^{-\mb X}_\xi+(p(\wt v_{-\mb X}))_\xi=\left(\frac{(\wt u_{-\mb X})_\xi}{\wt v_{-\mb X}}\right)_\xi+F_1 +F_2,
        \end{array} \right.
\end{aligned}
\end{align}
where $(\wt v^S)^{-\mb X}_\xi:=\wt v^S_\xi(\xi -\mb X(t))$, $(\wt u^S)^{-\mb X}_\xi:=\wt u^S_\xi(\xi -\mb X(t))$ and
\beq\label{ff1f2}
F_1=\left(\frac{ ({\wt u}^S_\xi)^{-\mb{X}} }{ ({\wt v}^S)^{-\mb{X}}}\right)_\xi - \left(\frac{(\wt u_{-\mb X})_\xi}{\wt v_{-\mb X}}\right)_\xi ,\quad F_2=\big[p(\wt v_{-\mb X})-p(\wt v^R)-p\big((\wt v^S)^{-\mb X}\big)\big]_\xi.
\eeq
Note that the shift $\mb{X}(t)$ is performed only in the shock layer. The terms $F_1$ and $F_2$ are the wave interactions due to nonlinearity of the viscosity and the pressure and error terms due to the inviscid rarefaction.\\

We now state the key step for the proof of Theorem \ref{thm:main}. 

\begin{proposition} \label{prop2}
For a given point $(v_+,u_+)\in\bbr^+\times\bbr$, there exist positive constants $C_0, \delta_0,\eps_1$ such that the following holds.\\
Suppose that $(v,u)$ is the solution to \eqref{NS-1} on $[0,T]$ for some $T>0$,  and $(\wt v_{-\mb X}, \wt u_{-\mb X})$ is defined in \eqref{shwave} with $\mb X$ being the absolutely continuous solution to \eqref{X(t)} with weight function $a$ defined in \eqref{weight}. Assume that both the rarefaction and shock waves strength satisfy $\deltar, \deltas<\delta_0$ and that 
\begin{align*}
\begin{aligned}
&v -\wt v_{-\mb X} \in C([0,T];H^1(\bbr)), \\
& u -\wt u_{-\mb X} \in C([0,T];H^1(\bbr)) \cap  L^2(0,T;H^2(\bbr)),
\end{aligned}
\end{align*}
and 
\beq\label{apri-ass}
\|v-\wt v_{-\mb X}\|_{L^\infty(0,T;H^1(\bbr))} + \|u-\wt u_{-\mb X}\|_{L^\infty(0,T;H^1(\bbr))} \le \eps_1.
\eeq
Then,  for all $t\le T$,
\begin{align}
\begin{aligned}\label{finest}
&\sup_{t\in[0,T]}\Big[\|v-\wt v_{-\mb X}\|_{H^1(\bbr)} +\|u-\wt u_{-\mb X}\|_{H^1(\bbr)}\Big] +\sqrt{\deltas\int_0^t|\dot{\mb{X}}|^2 ds} \\
&\qquad\quad +\sqrt{\int_0^t \big(\mathcal{G}^S(U)+\mathcal{G}^R(U)+D(U)+D_1(U)+D_2(U) \big) ds}\\
&\quad\le C_0\left( \|v_0-\wt v(0,\cdot)\|_{H^1(\bbr)} +\|u_0 -\wt u(0,\cdot)\|_{H^1(\bbr)} \right) + C_0\deltar^{1/6} ,
\end{aligned}
\end{align}
where $C_0$ is independent of $T$ and
\begin{align}
\begin{aligned}\label{maingood}
%&G_1(U):= \frac{\lam}{\deltas}\int_\bbr  |(\wt v^S)_\xi^{-\mb X} |  \left| h-\wt h-\frac{p(v)-p(\wt v)}{\sigma}\right|^2 d\xi,\\
&\mathcal{G}^S(U):=\int_\bbr | v^S_\xi(\xi -\mb X(t)) | |v-\wt v_{-\mb X}|^2 d\xi,\\
& \mathcal{G}^R(U):= \int_\bbr  |\wt u^R_\xi | | v-\wt v_{-\mb X} |^2 d\xi, \\
& D(U):= \int_\bbr  |\partial_\xi \big(p(v)-p(\wt v_{-\mb X})\big)|^2 d\xi,\\
&D_1(U):=  \int_\bbr \big|(u-\wt u_{-\mb X})_{\xi} \big|^2  d\xi,\\
&D_2(U):=  \int_\bbr \big|(u-\wt u_{-\mb X})_{\xi\xi} \big|^2  d\xi.
\end{aligned}
\end{align}
In addition, by \eqref{X(t)},
\beq\label{xprop}
|\dot{\mb{X}}(t)|\leq C_0\|(v-\wt v_{-\mb X})(t,\cdot)\|_{L^\infty(\bbr)},\qquad t\le T.
\eeq
\end{proposition}

We postpone the proof of this key proposition to Sections \ref{sec-acontraction} and \ref{sec-vu}. We are proving in the rest of this section  how Proposition \ref{prop2} implies Theorem \ref{thm:main}.

\subsection{Global in time estimates on the perturbations} 
We first prove   \eqref{ext-main} from  Theorem \ref{thm:main} by using Proposition \ref{prop:soln} and Proposition \ref{prop2} and a continuation argument.\\
Let us consider the positive constants $\delta_0, \eps_1, C_0$ of Proposition \ref{prop2}.  The constant $\delta_0$ control the maximum size of the shock and the rarefaction, and  can be chosen even smaller if needed. First, by \eqref{sm} in Proposition \ref{prop:soln}, the smooth and monotone functions $\underline v(x), \underline u(x)$ especially satisfy for some $C_*>0$,
\begin{align}
\begin{aligned}\label{underini}
&\sum_{\pm}\Big( \|\underline v-v_\pm\|_{L^2(\bbr_\pm)} +  \|\underline u-u_\pm\|_{L^2(\bbr_\pm)} \Big) + \|\partial_x\underline v\|_{L^2(\bbr)} + \|\partial_x\underline u\|_{L^2(\bbr)} \\
&\qquad\qquad \le C(|v_+-v_-|+|u_+-u_-|) \le C_*(\deltar+\deltas) (\le 2C_*\delta_0).
\end{aligned}
\end{align}
This together with Lemmas \ref{lemma1.3} and \ref{lemma1.2} then implies that for some $C_1>0$,
\begin{align}
\begin{aligned}\label{fffest}
&\|\underline v(\cdot)-\wt v(0,\cdot)\|_{H^1(\bbr)} +\|\underline u(\cdot)-\wt u(0,\cdot)\|_{H^1(\bbr)} \\
&\le \sum_{\pm}\Big( \|\underline v-v_\pm\|_{L^2(\bbr_\pm)} +  \|\underline u-u_\pm\|_{L^2(\bbr_\pm)} \Big)+\|\wt v^R(0)-v_m\|_{L^2(\bbr_+)}\\
&\quad +\|\wt v^S-v_+\|_{L^2(\bbr_+)}   +\|\wt v^R(0)-v_-\|_{L^2(\bbr_-)} +\|\wt v^S-v_m\|_{L^2(\bbr_-)} \\
&\quad + \|\partial_x\underline v\|_{L^2(\bbr)} +\|\partial_x\wt v^R(0)\|_{L^2(\bbr)} + \|\wt v^S_\xi\|_{L^2(\bbr)}\\
&\quad   +\|\wt u^R(0)-u_m\|_{L^2(\bbr_+)}+\|\wt u^S-u_+\|_{L^2(\bbr_+)} +\|\wt u^R(0)-u_-\|_{L^2(\bbr_-)}\\
&\quad   +\|\wt u^S-u_m\|_{L^2(\bbr_-)}  + \|\partial_x\underline u\|_{L^2(\bbr)} +\|\partial_x\wt u^R(0)\|_{L^2(\bbr)} + \|\wt u^S_\xi\|_{L^2(\bbr)}\\
&\le C_1(\deltar+\sqrt{\deltas}).
\end{aligned}
\end{align}
By smallness of $\delta_0$, we observe that for any $\deltas, \deltar \in (0,\delta_0)$, 
\beq\label{10}
\frac{\frac{\eps_1}{2}-C_0 \deltar^{1/6} }{C_0+1} -C_1(\deltar+\sqrt{\deltas})-C_*(\deltar+\deltas) >0.
\eeq
Let $\eps_0$ be the above positive constant:
\[
\eps_0:=\eps_*-C_*(\deltar+\deltas),\quad{\rm and} \ \ \eps_*:=\frac{\frac{\eps_1}{2}-C_0 \deltar^{1/6} }{C_0+1} -C_1(\deltar+\sqrt{\deltas}),
\]
where note that $\eps_0$ can be chosen independently on $\delta_S, \delta_R$, for example, as $\eps_0=\frac{\eps_1}{4(C_0+1)}$.\\
The specific constants $\eps_0, \eps_*$ will be used to apply Propositions \ref{prop:soln} and \ref{prop2} as below.\\
Consider any initial data $(v_0,u_0)$  verifying the hypothesis \eqref{i-p} of Theorem \ref{thm:main}, that is, 
\beq\label{genini}
\sum_{\pm}\Big( \|v_0-v_\pm\|_{L^2(\bbr_\pm)} +  \|u_0-u_\pm\|_{L^2(\bbr_\pm)} \Big) + \|v_{0x}\|_{L^2(\bbr)}+ \|u_{0x}\|_{L^2(\bbr)} < \eps_0,
\eeq
which together with \eqref{underini} yields
\begin{align}
\begin{aligned}\label{inie}
&\|v_0-\underline v\|_{H^1(\bbr)} +\|u_0 -\underline u\|_{H^1(\bbr)}  \\
&\quad \le \sum_{\pm}\Big( \|v_0-v_\pm\|_{L^2(\bbr_\pm)} +  \|u_0-u_\pm\|_{L^2(\bbr_\pm)}+\|\underline v-v_\pm\|_{L^2(\bbr_\pm)} +  \|\underline u-u_\pm\|_{L^2(\bbr_\pm)} \Big) \\
&\qquad + \|v_{0x}\|_{L^2(\bbr)}+ \|u_{0x}\|_{L^2(\bbr)}+ \|\underline v_x\|_{L^2(\bbr)} + \|\underline u_x\|_{L^2(\bbr)}\\
&\quad \le \eps_0+C_*(\deltar+\deltas) = \eps_*.
\end{aligned}
\end{align}
Especially, this together with Sobolev embedding implies that 
\beq\label{betv}
\|v_0-\underline v\|_{L^\infty(\bbr)} \le C\eps_*,
\eeq
which together with smallness of $\eps_*$ implies that
\[
\frac{v_-}{2}< v_0(\xi) < 2v_+,\quad  \forall \xi\in \bbr.
\]
 Since $\eps_*$ satisfies $0<\eps_*<\frac{\eps_1}{2}$ by \eqref{10}, Proposition \ref{prop:soln} with \eqref{inie} and \eqref{betv} implies that there exists $T_0>0$ such that 
\eqref{NS-1} has a unique solution $(v,u)$ on $[0,T_0]$ satisfying
\beq\label{341}
\|v-\underline v\|_{L^\infty(0,T_0;H^1(\bbr))} + \|u-\underline u\|_{L^\infty(0,T_0;H^1(\bbr))} \le \frac{\eps_1}{2},
\eeq
and
\[
\frac{v_-}{3}< v(t,\xi) < 3v_+,\quad  \forall (t,\xi)\in [0,T_0]\times \bbr.
\]
Then, using the same argument as in \eqref{fffest}, and then using Lemmas \ref{lemma1.2} and \ref{lem:xex}, we find that for all $t\in [0,T_0]$,
\begin{align*}
\begin{aligned}
&\|\underline v-\wt v_{-\mb X}(t,\cdot)\|_{L^2(\bbr)} + \|\underline u-\wt u_{-\mb X}(t,\cdot)\|_{L^2(\bbr)}\\
&\le  \sum_{\pm}\Big( \|\underline v-v_\pm\|_{L^2(\bbr_\pm)} +  \|\underline u-u_\pm\|_{L^2(\bbr_\pm)} \Big)+\|\wt v^R(t,\cdot+\s t)-v_m\|_{L^2(\bbr_+)} +\|(\wt v^S)^{-\mb X}-v_+\|_{L^2(\bbr_+)}  \\
&\quad +\|\wt v^R(t,\cdot +\s t)-v_-\|_{L^2(\bbr_-)} +\|(\wt v^S)^{-\mb X}-v_m\|_{L^2(\bbr_-)} +\|\wt u^R(t,\cdot+\s t)-u_m\|_{L^2(\bbr_+)}  \\
&\quad + \|\partial_x\underline v\|_{L^2(\bbr)} +\|\partial_x\wt v^R(t)\|_{L^2(\bbr)} + \|(\wt v^S)^{-\mb X}_\xi\|_{L^2(\bbr)} \\
&\quad +\|(\wt u^S)^{-\mb X}-u_+\|_{L^2(\bbr_+)} +\|\wt u^R(t,\cdot+\s t)-u_-\|_{L^2(\bbr_-)} +\|(\wt u^S)^{-\mb X}-u_m\|_{L^2(\bbr_-)}  \\
&\quad + \|\partial_x\underline u\|_{L^2(\bbr)} +\|\partial_x\wt u^R(t)\|_{L^2(\bbr)} + \|(\wt u^S)^{-\mb X}_\xi\|_{L^2(\bbr)}\\
&\le C\deltar\sqrt{1+(\s-\lambda_1(v_-))t}+C\sqrt{\deltas}(1 +\sqrt{|\mb X(t)|}) \\
&\le C\sqrt{\delta_0}(1+\sqrt{t}).
\end{aligned}
\end{align*}
Indeed, some estimates above are computed as follows:
\begin{align*}
\begin{aligned}
&\int_0^\infty |\wt v^S(\xi-\mb X(t)) -v_+|^2 d\xi =\int_{-\mb X(t)}^\infty |\wt v^S(\xi) -v_+|^2 d\xi \\
&\quad\le  \int_0^\infty  C\delta_S^2 e^{-C\deltas|\xi|} d\xi + \int_0^{|\mb X(t)|}|\wt v^S(\xi) -v_+|^2 d\xi  \le C\deltas (1+|\mb X(t)|) ,\\
&\int_{-\infty}^0 |\wt v^R(t,\xi+\s t) -v_-|^2 d\xi = \int_{-\infty}^{\s t} |\wt v^R(t,x) -v_-|^2 dx \\
&\quad = \int_{\lambda_1(v_-) t}^{\s t} |\wt v^R(t,x) -v_-|^2 dx +  \int_{-\infty}^{\lambda_1(v_-) t} |\wt v^R(t,x) -v_-|^2 dx\\
&\quad\le \delta_R^2 (\s-\lambda_1(v_-))t + C \delta_R^2  \int_{-\infty}^{\lambda_1(v_-) t} e^{-4|x-\lambda_1(v_-)t|} dx \le  C \delta_R^2\big(1+(\s-\lambda_1(v_-))t\big),\\
&\int_0^\infty |\wt v^R(t,\xi+\s t) -v_m|^2 d\xi \le   C \delta_R^2\int_{\s t}^\infty e^{-4|x|} dx \le   C \delta_R^2.\\
\end{aligned}
\end{align*}
Using smallness of $\delta_0$, and choosing $T_1\in (0, T_0)$ small enough such that $C\sqrt{\delta_0}(1+\sqrt{T_1})\le \frac{\eps_1}{2}$, we have
\beq\label{342}
\|\underline v-\wt v_{-\mb X}\|_{L^\infty(0,T_1;H^1(\bbr))} + \|\underline u-\wt u_{-\mb X}\|_{L^\infty(0,T_1;H^1(\bbr))} \le \frac{\eps_1}{2}.
\eeq
Therefore, \eqref{341} and \eqref{342} imply that
\[
\|v-\wt v_{-\mb X}\|_{L^\infty(0,T_1;H^1(\bbr))} + \|u-\wt u_{-\mb X}\|_{L^\infty(0,T_1;H^1(\bbr))} \le \eps_1.
\]
Especially, since $\mb X$ is absolutely continuous, and
\[
v -\underline v, u -\underline u \in C([0,T_1];H^1(\bbr)),
\]
we have
\[
v -\wt v_{-\mb X}, u -\wt u_{-\mb X} \in C([0,T_1];H^1(\bbr)).
\]
We now consider the maximal existence time:
\[
T_M := \sup \left \{ t>0~\Big|~\sup_{[0,t]} \left(\|v-\wt v_{-\mb X}\|_{H^1(\bbr)} + \|u-\wt u_{-\mb X}\|_{H^1(\bbr)} \right) \le \eps_1 \right \}.
\]
If $T_M<\infty$, then the continuation argument implies that
\beq\label{343}
 \sup_{[0,T_M]} \left(\|v-\wt v_{-\mb X}\|_{H^1(\bbr)} + \|u-\wt u_{-\mb X}\|_{H^1(\bbr)} \right) =\eps_1.
\eeq
But, since it follows from \eqref{fffest} and \eqref{inie} that
\[
 \|v_0-\wt v(0,\cdot)\|_{H^1(\bbr)} +\|u_0 -\wt u(0,\cdot)\|_{H^1(\bbr)} < \frac{\frac{\eps_1}{2}-C_0 \deltar^{1/6} }{C_0+1},
\]
it holds from Proposition \ref{prop2} that
\[
 \sup_{[0,T_M]} \left(\|v-\wt v_{-\mb X}\|_{H^1(\bbr)} + \|u-\wt u_{-\mb X}\|_{H^1(\bbr)} \right) \le C_0  \frac{\frac{\eps_1}{2}-C_0 \deltar^{1/6} }{C_0+1} + C_0\deltar^{1/6} \le \frac{\eps_1}{2},
\]
which contradicts the above equality \eqref{343}. \\
Therefore, $T_M=\infty$, which together with Proposition \ref{prop2} implies
\begin{align}
\begin{aligned}\label{fcru}
&\sup_{t>0}\big(\|v-\wt v_{-\mb X}\|_{H^1(\bbr)} +\|u-\wt u_{-\mb X}\|_{H^1(\bbr)} \big) +\sqrt{\deltas\int_0^\infty |\dot{\mb{X}}|^2 ds} \\
&\qquad\quad +\sqrt{\int_0^\infty \big(\mathcal{G}^S(U)+\mathcal{G}^R(U)+D(U)+D_1(U)+D_2(U) \big) ds}\\
&\quad \le C_0\left( \|v_0-\wt v(0,\cdot)\|_{H^1(\bbr)} +\|u_0 -\wt u(0,\cdot)\|_{H^1(\bbr)} \right) + C_0\deltar^{1/6} < \infty,
\end{aligned}
\end{align}
and
\beq\label{xcru}
|\dot{\mb{X}}(t)|\leq C_0\|(v-\wt v_{-\mb X})(t,\cdot)\|_{L^\infty(\bbr)}, \quad t>0.
\eeq
In addition, since the rarefaction wave $(v^r, u^r)$ is Lipschitz continuous in $x$ for all $t>0$ and from Lemma \ref{lemma1.2}, 
we have
 \begin{align*}
&v(t,x)- \Big( v^r(\frac xt)+\wt v^S(x-\s t-\mb X(t))-v_m \Big) \in C([0,+\infty);H^1(\bbr)),\\
& u(t,x)- \Big( u^r(\frac xt)+\wt u^S(x-\s t-\mb X(t))-u_m \Big) \in C([0,+\infty);H^1(\bbr)).
\end{align*}
Since $(u-\wt u_{-\mb X})_{\xi\xi}\in L^2(0,+\infty; L^2(\bbr))$ by \eqref{fcru}, and $(\wt u^R)_{\xi\xi}\in L^2(0,+\infty; L^2(\bbr))$ by Lemma \ref{lemma1.2},
we have
\[
u_{xx}(t,x)-\wt u^S_{xx}(x-\s t-\mb X(t))\in L^2(0,+\infty; L^2(\bbr)),
\]
which implies the desired result \eqref{ext-main}.\\
Especially, since the right-hand side of \eqref{fcru} is small enough, we find that (by Sobolev embedding as before)
\beq\label{vabdd}
\frac{v_-}{3}< v(t,\xi) < 3v_+,\quad  \forall (t,\xi)\in [0,\infty)\times \bbr.
\eeq
These and the above estimates \eqref{fcru}-\eqref{xcru} are useful to prove the long-time behaviors \eqref{con}-\eqref{as} as follows.

\subsection{Time-asymptotic behavior, and end of the proof of Theorem \ref{thm:main}} We now want to prove \eqref{con} and \eqref{as}.
Consider a function $g$ defined on $(0,\infty)$ by
\[
g(t):=\|(v-\wt v_{-\mb X})_\xi\|_{L^2(\bbr)}^2 +\|(u-\wt u_{-\mb X})_\xi\|_{L^2(\bbr)}^2 .
\]
The aim is to show the classical estimate:
\begin{equation}\label{asym}
\int_0^\infty\big[|g(t)|+|g^\prime(t)| \big]dt<\infty.
\end{equation}
Since
\begin{align*}
\begin{aligned}
(p(v)-p(\wt v_{-\mb X}))_\xi &= p'(v)(v-\wt v_{-\mb X})_\xi + (\wt v_{-\mb X})_\xi (p'(v)-p'(\wt v_{-\mb X}))\\
&=p'(v)(v-\wt v_{-\mb X})_\xi + \big(\wt v^R_\xi + \wt v^S_\xi (\xi-\mb X(t)) \big) (p'(v)-p'(\wt v_{-\mb X})),
\end{aligned}
\end{align*}
the uniform bound \eqref{vabdd} yields
\beq\label{8*}
|(v-\wt v_{-\mb X})_\xi | \le C |(p(v)-p(\wt v_{-\mb X}))_\xi | +C\big(|\wt v^R_\xi| + |\wt v^S_\xi (\xi-\mb X(t))| \big) |v-\wt v_{-\mb X}|.
\eeq
Thus, it follows from \eqref{fcru}, \eqref{8*} and $ |\wt u^R_\xi | \sim  |\wt v^R_\xi |$ that
$$
\int_0^\infty |g(t)| dt\leq C\int_0^\infty \big(\mathcal{G}^S(U)+\mathcal{G}^R(U)+D(U)+D_1(U) \big) dt<\infty,
$$
which proves the first part of \eqref{asym}. \\
To show the second part of \eqref{asym}, we combine the systems \eqref{NS-1} and \eqref{sws-u} as follows:
\begin{align}
\begin{aligned}\label{combt}
&(v-\wt v_{-\mb X})_t-\sigma(v-\wt v_{-\mb X})_\xi -\dot{\mb{X}}(t)(\wt v^S)^{-\mb X}_\xi -(u-\wt u_{-\mb X})_\xi =0,\\
& \di (u-\wt u_{-\mb X})_t-\sigma(u-\wt u_{-\mb X})_\xi-\dot{\mb{X}}(t) (\wt u^S)^{-\mb X}_\xi+(p(v)-p(\wt v_{-\mb X}))_\xi\\
&\qquad\qquad =\left(\frac{u_\xi}{v} - \frac{(\wt u_{-\mb X})_\xi}{\wt v_{-\mb X}}\right)_\xi -F_1 -F_2.
\end{aligned}
\end{align}
Using \eqref{combt} and the integration by parts, we have
\beq\label{bound}
\begin{array}{ll}
\di \int_0^\infty |g^\prime(t)|dt=\int_0^\infty2\left|\int (v-\wt v_{-\mb X})_\xi(v-\wt v_{-\mb X})_{\xi t} d\xi+ \int (u-\wt u_{-\mb X})_\xi(u-\wt u_{-\mb X})_{\xi t} d\xi\right| dt\\[4mm]
\di \leq \int_0^\infty\left| \s \int \partial_\xi ((v-\wt v_{-\mb X})^2_\xi) d\xi + 2\int (v-\wt v_{-\mb X})_\xi \left[\dot{\mb{X}}(t)(\wt v^S)^{-\mb X}_{\xi\xi} +(u-\wt u_{-\mb X})_{\xi\xi}  \right] d\xi\right| dt  \\[4mm]
\di \quad + \int_0^\infty\bigg|\s \int \partial_\xi ((u-\wt u_{-\mb X})^2_\xi) d\xi +2\int (u-\wt u_{-\mb X})_\xi \dot{\mb{X}}(t) (\wt u^S)^{-\mb X}_{\xi\xi} d\xi  \\[4mm]
\di\qquad\qquad+2\int (u-\wt u_{-\mb X})_{\xi\xi} \bigg[-(p(v)-p(\wt v_{-\mb X}))_{\xi} +\left(\frac{u_\xi}{v} - \frac{(\wt u_{-\mb X})_\xi}{\wt v_{-\mb X}}\right)_{\xi} - F_1 - F_2 \bigg] d\xi\bigg| dt \\[4mm]
\di \le 2\int_0^\infty\int \bigg( |(v-\wt v_{-\mb X})_\xi | \left[ |\dot{\mb{X}}(t)| |(\wt v^S)^{-\mb X}_{\xi\xi}| +|(u-\wt u_{-\mb X})_{\xi\xi}|\right]  + |(u-\wt u_{-\mb X})_\xi | |\dot{\mb{X}}(t)| |(\wt u^S)^{-\mb X}_{\xi\xi}|   \\[4mm]
\di\qquad\qquad+ |(u-\wt u_{-\mb X})_{\xi\xi}| \bigg[|(p(v)-p(\wt v_{-\mb X}))_{\xi}| +\Big|\left(\frac{u_\xi}{v} - \frac{(\wt u_{-\mb X})_\xi}{\wt v_{-\mb X}}\right)_{\xi}\Big| + |F_1| + |F_2| \bigg] \bigg) d\xi dt\\[5mm]
\di \le C\int_0^\infty \big(|\dot{\mb{X}}(t)|^2+\mathcal{G}^S(U)+\mathcal{G}^R(U)+D(U)+D_1(U)+D_2(U) \big) dt\\[4mm]
\di \quad +C\int_0^\infty \int\Big[\Big|\Big(\frac{u_\xi}{v} - \frac{(\wt u_{-\mb X})_\xi}{\wt v_{-\mb X}}\Big)_{\xi}\Big|^2 + |F_1|^2 + |F_2|^2\Big]d\xi dt.
\end{array}
\eeq
For the last three terms above, we get further estimates as follows. \\
Using \eqref{vabdd} with Lemmas \ref{lemma1.3} and \ref{lemma1.2}, one has
\begin{align*}
\nonumber&\int_0^\infty \int\Big|\Big(\frac{u_\xi}{v} - \frac{(\wt u_{-\mb X})_\xi}{\wt v_{-\mb X}}\Big)_{\xi}\Big|^2d\xi dt \\\nonumber
&=\int_0^\infty \int\bigg| \frac{1}{v}\left( u -\wt u_{-\mb X} \right)_{\xi\xi}+ (\wt u_{-\mb X} )_{\xi\xi}\left(\frac{1}{v}-\frac{1}{\wt v_{-\mb X} }\right) -\frac{u_\xi}{v^2} \left(v -\wt v_{-\mb X} \right)_\xi \\\nonumber
&\qquad-\frac{(\wt v_{-\mb X})_\xi}{v^2} \left(u -\wt u_{-\mb X} \right)_\xi  + (\wt v_{-\mb X})_{\xi} (\wt u_{-\mb X})_{\xi} \left(\frac{1}{v^2}-\frac{1}{(\wt v_{-\mb X})^2 }\right)   \bigg| ^2d\xi dt \\
&\le C \int_0^\infty \int \Big[|\left( u -\wt u_{-\mb X} \right)_{\xi\xi}|^2 + (|(\wt u^R)_{\xi}|^2+|(\wt u^S)^{-\mb X}_{\xi}|^2)|v-\wt v_{-\mb X}|^2 \\\nonumber
&\qquad +  |(u-\wt u_{-\mb X})_\xi|^2 |\left(v -\wt v_{-\mb X} \right)_\xi|^2 +  (|(\wt v^R)_{\xi}|^2+|(\wt v^S)^{-\mb X}_{\xi}|^2) |\left(u -\wt u_{-\mb X} \right)_\xi|^2 \\\nonumber
&\qquad +  (|(\wt u^R)_{\xi}|^2+|(\wt u^S)^{-\mb X}_{\xi}|^2) |\left(v -\wt v_{-\mb X} \right)_\xi|^2\Big]d\xi dt.
\end{align*}
Then, using \eqref{maingood}, we have
\begin{align*}
\nonumber&\int_0^\infty \int\Big|\Big(\frac{u_\xi}{v} - \frac{(\wt u_{-\mb X})_\xi}{\wt v_{-\mb X}}\Big)_{\xi}\Big|^2d\xi dt \\
&\le C\int_0^\infty \Big(\mathcal{G}^S(U)+\mathcal{G}^R(U)+D(U)+D_1(U)+D_2(U)\Big)dt\\\nonumber
&\qquad + C\int_0^\infty\|\left(u -\wt u_{-\mb X} \right)_\xi\|_{L^\infty(\mathbb{R})}^2\int|(v-\wt v_{-\mb X})_\xi|^2 d\xi dt.
\end{align*}
Using the interpolation inequality and \eqref{fcru}, the last term above is estimated as
\[
\begin{array}{ll}
\di C\int_0^\infty\|\left(u -\wt u_{-\mb X} \right)_\xi\|_{L^\infty(\mathbb{R})}^2\int|(v-\wt v_{-\mb X})_\xi|^2 d\xi dt\\[4mm]
\di \le C\int_0^\infty\|\left(u -\wt u_{-\mb X} \right)_\xi\|_{L^2(\bbr)}\|\left(u -\wt u_{-\mb X} \right)_{\xi\xi}\|_{L^2(\bbr)} \|(v-\wt v_{-\mb X})_\xi\|_{L^2(\bbr)}^2 dt\\[4mm]
\di \le C\int_0^\infty\Big[\|\left(u -\wt u_{-\mb X} \right)_{\xi\xi}\|_{L^2(\bbr)}^2+\|\left(u -\wt u_{-\mb X} \right)_\xi\|_{L^2(\bbr)}^2\|\left(v -\wt v_{-\mb X} \right)_\xi\|_{L^2(\bbr)}^4\Big] dt\\[4mm]
\di \le C\int_0^\infty\Big[\|\left(u -\wt u_{-\mb X} \right)_{\xi\xi}\|_{L^2(\bbr)}^2+\|\left(u -\wt u_{-\mb X} \right)_\xi\|_{L^2(\bbr)}^2 \Big] dt\\[4mm]
\di \le C\int_0^\infty \Big(\mathcal{G}^S(U)+\mathcal{G}^R(U)+D(U)+D_2(U)\Big)dt<\infty.
\end{array}
\]
Similarly, using Lemmas \ref{lemma1.3} and \ref{lemma1.2} with recalling $\wt v_{-\mb X}=\wt v^R+(\wt v^S)^{-\mb X} -v_m$, we have
\begin{align*}
\begin{aligned}
& \int_0^\infty \int|F_1|^2d\xi dt=\int_0^\infty \int\Big|\Big(\frac{ ({\wt u}^S_\xi)^{-\mb{X}} }{ ({\wt v}^S)^{-\mb{X}}}\Big)_\xi  - \left(\frac{(\wt u_{-\mb X})_\xi}{\wt v_{-\mb X}} \right)_\xi \Big|^2d\xi dt \\
&\quad  \le C\int_0^\infty \int\bigg( |(\wt u^R)_{\xi\xi}|+|(\wt u^R)_{\xi}| |(\wt v^R)_{\xi}|  +( |(\wt u^S)^{-\mb{X}}_{\xi\xi}|+|(\wt u^S)^{-\mb{X}}_{\xi}||(\wt v^S)^{-\mb{X}}_{\xi}|) |\wt v^R-v_m| \\
&\quad\qquad\qquad\qquad +|(\wt u^R)_\xi | |(\wt v^S)^{-\mb{X}}_\xi | +|(\wt v^R)_\xi | |(\wt u^S)^{-\mb{X}}_\xi | \bigg)^2d\xi dt\\
&\quad  \le C\int_0^\infty \left(\|(\wt u^R)_{\xi\xi}\|_{L^2(\bbr)}^2 + \|(\wt u^R)_{\xi}\|_{L^4(\bbr)}^4 + \||(\wt v^S)^{-\mb{X}}_{\xi}| |\wt v^R-v_m| +|(\wt v^R)_\xi | |(\wt u^S)^{-\mb{X}}_\xi | \|_{L^2(\bbr)}^2 \right) dt,\\
\end{aligned}
\end{align*}
and 
\begin{align*}
\int_0^\infty \int|F_2|^2d\xi dt &=\int_0^\infty \int\left| \big[p(\wt v_{-\mb X})-p(\wt v^R)-p\big((\wt v^S)^{-\mb X}\big)\big]_\xi \right|^2d\xi dt \\
& \le C\int_0^\infty \| |\wt v^R_\xi| |(\wt v^S)^{-\mb X}-v_m| +|(\wt v^S)^{-\mb X}_\xi|  |\wt v^R -v_m| \|_{L^2(\bbr)}^2 dt .
\end{align*}
Notice that the right-hand sides above are finite by Lemma \ref{lemma1.2} and Lemma \ref{lemma2.2}.
Thus, the above estimates with \eqref{fcru} imply the proof of the second part of \eqref{asym}. \\
Therefore, we have \eqref{asym}, which implies
\[
\lim_{t\rightarrow+\infty} \big(\|(v-\wt v_{-\mb X})_\xi\|_{L^2(\bbr)}^2 +\|(u-\wt u_{-\mb X})_\xi\|_{L^2(\bbr)}^2 \big) =0.
\]
This together with the interpolation inequality and \eqref{fcru} implies
\begin{equation}\label{asym-b}
\lim_{t\rightarrow+\infty} \big(\|v-\wt v_{-\mb X}\|_{L^\infty(\bbr)} +\|u-\wt u_{-\mb X}\|_{L^\infty (\bbr)} \big) =0,
\end{equation}
which together with Lemma \ref{lemma1.2} (5) implies \eqref{con}. In addition, by \eqref{xcru} and \eqref{asym-b}, it holds that
\begin{equation}
|\dot{\mb{X}}(t)|\leq C_0\|(v-\wt v_{-\mb X})(t,\cdot)\|_{L^\infty(\bbr)} \rightarrow 0\quad {\rm as}\quad t\rightarrow +\infty,
\end{equation}
which proves \eqref{as}. Thus we complete the proof of Theorem \ref{thm:main}.\\

Hence, the remaining part of this paper is dedicated to the proof of Proposition \ref{prop2}.\\
 
\noindent$\bullet$ {\bf Notations:} In what follows, we use the following notations  for notational simplicity. \\
1. $C$ denotes a positive $O(1)$-constant which may change from line to line, but which is independent of the small constants $\delta_0, \eps_1, \deltas,\deltar$, $\lambda$ (to appear in \eqref{weight}) and the time $T$.\\
2. For any function $f : \bbr^+\times \bbr\to \bbr$ and any time-dependent shift $\mb X(t)$, 
\[
f^{\pm \mb X}(t, \xi):=f(t,\xi\pm \mb X(t)).
\]
3. We omit the dependence on $\mb X$ for \eqref{shwave} as follows:
\[
(\wt v, \wt u) (t,\xi) := \Big(\wt v^R(t,\xi+\s t)+\wt v^S(\xi -\mb X(t))-v_m,\wt u^R(t,\xi+\s t)+\wt u^S(\xi-\mb X(t))-u_m \Big) .
\]
For simplicity, we also omit the arguments of the waves without confusion: for example, 
\begin{align*}
\begin{aligned}
&\wt v^R:=\wt v^R(t,\xi+\s t),\quad  (\wt v^R)^{\mb X}:=\wt v^R(t,\xi+\s t + \mb X(t)),\\
&\wt v^{\mb X}:=\wt v^R(t,\xi+\s t + \mb X(t))+\wt v^S(\xi)-v_m.
\end{aligned}
\end{align*}

\section{Energy estimates for the system of $(v,h)$-variables}\label{sec-acontraction}
\setcounter{equation}{0}
We introduce a new effective velocity 
\beq\label{h}
h:= u-(\ln v)_\xi.
\eeq
 Then, the system \eqref{NS-1} is transformed into
\begin{align}
\begin{aligned}\label{hNS-1}
\left\{ \begin{array}{ll}
 v_t-\sigma v_\xi -h_\xi= (\ln v)_{\xi\xi}, \\
        h_t-\sigma h_\xi+p(v)_\xi =0.\\
        \end{array} \right.
\end{aligned}
\end{align}
We set $\wt h^S:= \wt u^S -(\ln \wt v^S)_\xi$. Then, it follows from \eqref{VS} that
\begin{equation}\label{hVS}
\left\{
\begin{array}{ll}
\di -\sigma (\wt v^S)^\prime-(\wt h^S)^\prime=(\ln \wt v^S)'',\\[3mm]
\di -\sigma(\wt h^S)^\prime+(p(\wt v^S))^\prime=0,\\[3mm]
\di (\wt v^S,\wt h^S)(-\infty)=(v_m, u_m),\qquad  (\wt v^S,\wt h^S)(+\infty)=(v_+, u_+).
\end{array}
\right.
\end{equation}
Set
\beq\label{huvx}
\wt h(t,\xi):=\wt u^R(t,\xi) + (\wt h^S)^{-\mb X}(\xi) - u_m,\quad \mbox{for } t\in[0,T].
\eeq
Then, it holds from \eqref{rarexi} and \eqref{hVS} that
\begin{align}
\begin{aligned}\label{SWS}
\left\{ \begin{array}{ll}
       \di \wt v_t-\sigma\wt v_\xi+\dot{\mb{X}}(t)(\wt v^S)^{-\mb{X}}_\xi-\wt h_\xi =(\ln \wt v)_{\xi\xi}+F_3,\\[3mm]
      \di \wt h_t-\sigma \wt h_\xi+\dot{\mb{X}}(t)(\wt h^S)^{-\mb{X}}_\xi +(p(\wt v))_\xi= F_2, \\[3mm]
        \end{array} \right.
\end{aligned}
\end{align}
where $F_2$ is defined in \eqref{ff1f2} 
\beq\label{f1f2}
F_3=\big(\ln (\wt v^S)^{-\mb X} -\ln \wt v\big)_{\xi\xi}.
\eeq

This section is dedicated to the proof of the following lemma.\\
\begin{lemma}\label{lem-zvh}
Under the hypotheses of Proposition \ref{prop2}, there exists $C>0$ (independent of $\delta_0, \eps_1, T$) such that for all $t\in (0,T]$,
\begin{align}
\begin{aligned}\label{esthv}
&\int_{\bbr} \left(\frac{|h-\wt h|^2}{2} +Q(v|\wt v)\right) d\xi +\deltas\int_0^t|\dot{\mb{X}}|^2 ds +\int_0^t \left( G_1(U)+ G^S(U) + D(U) \right) ds\\
&\quad\le C\int_{\bbr} \left(\frac{|h(0,\xi)-\wt h(0,\xi)|^2}{2} +Q(v_0|\wt v(0,\xi))\right)  d\xi  + C \deltar^{1/3},
\end{aligned}
\end{align}
where $h(0,\xi):= u_0(\xi)-(\ln v_0)_\xi(\xi)$, and
\begin{align}
\begin{aligned}\label{good1}
&G_1(U):= \frac{\lam}{\deltas}\int_\bbr  |(\wt v^S)_\xi^{-\mb X} |  \left| h-\wt h-\frac{p(v)-p(\wt v)}{\sigma}\right|^2 d\xi,\\
&G^S(U):=\int_\bbr |(\wt v^S)_\xi^{-\mb X} | |p(v)-p(\wt v)|^2 d\xi,\\
& D(U):= \int_\bbr  |\partial_\xi \big(p(v)-p(\wt v)\big)|^2 d\xi.
\end{aligned}
\end{align}

\end{lemma}

\subsection{Wave interaction estimates}
We here present useful estimates for the error terms $F_1, F_2, F_3$ introduced in \eqref{ff1f2} and \eqref{f1f2}. 
First, we notice that the a priori assumption \eqref{apri-ass} with the Sobolev embedding and \eqref{pressure2} implies
\beq\label{smp1}
\|p(v)-p(\wt v)\|_{L^\infty((0,T)\times\bbr)}\le C\|v-\wt v\|_{L^\infty((0,T)\times\bbr)} \le C\eps_1.
\eeq
This smallness together with \eqref{X(t)}, \eqref{apri-ass} and \eqref{pressure2} yields that
\beq\label{dxbound}
|\dot{\mb X}(t)| \le \frac{C}{\deltas}  \||p(v)-p(\wt v)|+|v-\wt v| \|_{L^\infty(\bbr)} \int_\bbr (\wt v^S)^{-\mb X}_\xi d\xi \le C\|v-\wt v\|_{L^\infty(\bbr)}.
\eeq
This especially proves \eqref{xprop}, and 
will be used to get the wave interaction estimates in Lemma \ref{lemma2.2}.

\begin{lemma}\label{lemma2.2}
Let $\mb X$ be the shift defined by \eqref{X(t)}. Under the same hypotheses as in Proposition \ref{prop2}, the following holds: $\forall t\le T$,\\
\begin{align*}
&\|(\wt v^S)^{-\mb{X}}_{\xi} (\wt v^R-v_m) \|_{L^1(\bbr)} + \| (\wt v^R)_\xi (\wt v^S)^{-\mb{X}}_\xi \|_{L^1(\bbr)} \le C\deltar \deltas e^{-C \deltas t},\\
&\|(\wt v^S)^{-\mb{X}}_{\xi} (\wt v^R-v_m) \|_{L^2(\bbr)} + \| (\wt v^R)_\xi (\wt v^S)^{-\mb{X}}_\xi \|_{L^2(\bbr)} \le C\deltar \deltas^{3/2} e^{-C \deltas t},\\
&\|(\wt v^R)_\xi ((\wt v^S)^{-\mb X}-v_m)\|_{L^2(\bbr)} \le C\deltar \deltas e^{-C \deltas t} .
\end{align*}
\end{lemma}
\begin{proof}
First, by \eqref{dxbound} with \eqref{smp1}, it holds that
$$
|\dot{\mb{X}}(t)|\leq C\eps_1,\qquad 0\le t\le T,
$$
which together with $\mb X(0)=0$ yields
$$
|\mb{X}(t)|\leq C\eps_1 t,\qquad 0\le t\le T.
$$
Let us take $\eps_1$ so small such that the above bound is less than $ \frac{\s t}{4}$, that is,
\[
C\eps_1t < \frac{\s t}{4}.
\]
Then, since
\begin{align*}
\forall \xi<-\frac{\s t}{2},\quad & \xi-\mb X(t) < -\frac{\s t}{2} +C\eps_1 t <-\frac{\s t}{4}<0 \quad\mbox{and} \\
&| \xi-\mb X(t)| \ge |\xi| -|\mb X(t)| > \frac{\s t}{2} -C\eps_1 t > \frac{\s t}{4},
\end{align*}
it holds from Lemma \ref{lemma1.3} that
\begin{align*}
\forall \xi<-\frac{\s t}{2},\quad |\wt v^S(\xi-\mb X(t)) -v_m | &\le C\deltas e^{-C\deltas |\xi -\mb X(t)|} \\
&\le C\deltas \exp\left(-\frac{C\deltas |\xi -\mb X(t)|}{2} \right) \exp\left(-\frac{C\deltas \s t}{8} \right) .
\end{align*}
Likewise, by Lemma \ref{lemma1.3},
\begin{align*}
\forall \xi<-\frac{\s t}{2},\quad |\partial_\xi \wt v^S(\xi-\mb X(t))| &\le C\deltas^2 e^{-C\deltas |\xi -\mb X(t)|} \\
&\le C\deltas^2 \exp\left(-\frac{C\deltas |\xi -\mb X(t)|}{2} \right) \exp\left(-\frac{C\deltas \s t}{8} \right) .
\end{align*}
On the other hand, since
\[
\forall \xi\ge -\frac{\s t}{2},\quad x=\xi +\s t \ge \frac{\s t}{2} \ge 0, 
\]
it holds from Lemma \ref{lemma1.2} that
\[
\forall \xi\ge -\frac{\s t}{2},\quad |\wt v^R(t,\xi+\s t) -v_m | + |\partial_\xi \wt v^R(t,\xi+\s t)| \le C\deltar e^{-2( |\xi +\s t| + |\lambda_1(v_m)| t )},
\] 
where note that $|\lambda_1(v_m)|>0$ is $O(1)$-constant, since $\frac{v_+}{2}\le v_m\le v_+$.\\
Therefore, using the above estimates together with the bounds: (by Lemmas \ref{lemma1.3} and \ref{lemma1.2})
\begin{align*}
\forall \xi,\quad & |\wt v^R(t,\xi+\s t) -v_m | + |\partial_\xi \wt v^R(t,\xi+\s t)| \le C\deltar,\\
& |\wt v^S(\xi-\mb X(t)) -v_m | \le C\deltas,\qquad |\partial_\xi \wt v^S(\xi-\mb X(t))|  \le C\deltas^2,\\
& \|\partial_\xi \wt v^R(t,\cdot+\s t)\|_{L^1(\bbr)} \le C\deltar,\quad\forall t,
\end{align*}
we have
\begin{align*}
\big|(\wt v^S)^{-\mb{X}}_{\xi}\big| \big(|\wt v^R-v_m| + |(\wt v^R)_\xi |\big)  \le   \begin{cases}
    \displaystyle C\deltar\deltas^2 e^{-C\deltas |\xi-\mb X(t)|} e^{-C \deltas t},\quad \mbox{if } \xi<-\frac{\s t}{2}, \\
    \displaystyle C\deltar\deltas^2 e^{-C |\xi+\s t|} e^{-C t},\qquad \mbox{if } \xi\ge -\frac{\s t}{2},
  \end{cases}
\end{align*}
and
\begin{align*}
|(\wt v^R)_\xi| |(\wt v^S)^{-\mb X}-v_m| \le   \begin{cases}
    \displaystyle C|(\wt v^R)_\xi| \deltas e^{-C\deltas |\xi-\mb X(t)|} e^{-C \deltas t},\quad \mbox{if } \xi<-\frac{\s t}{2}, \\
    \displaystyle C\deltar\deltas e^{-C |\xi+\s t|} e^{-C t},\qquad \mbox{if } \xi\ge -\frac{\s t}{2}.
  \end{cases}
\end{align*}
Hence, this with the smallness of $\deltas$ implies that 
\begin{align*}
\int_\bbr \Big|\big|(\wt v^S)^{-\mb{X}}_{\xi}\big| \big(|\wt v^R-v_m| + |(\wt v^R)_\xi |\big)\Big| d\xi &\le C\deltar \deltas e^{-C \deltas t} \int_\bbr \deltas \left(  e^{-C\deltas |\xi-\mb X(t)|} +   e^{-C |\xi+\s t|}   \right) d\xi \\
&\le C\deltar \deltas e^{-C \deltas t},
\end{align*}
\begin{align*}
\int_\bbr \Big|\big|(\wt v^S)^{-\mb{X}}_{\xi}\big| \big(|\wt v^R-v_m| + |(\wt v^R)_\xi |\big)\Big|^2 d\xi &\le C\deltar^2 \deltas^3 e^{-C \deltas t} \int_\bbr \deltas \left(  e^{-C\deltas |\xi-\mb X(t)|} +   e^{-C |\xi+\s t|}   \right) d\xi \\
&\le C\deltar^2 \deltas^3 e^{-C \deltas t},
\end{align*}
and
\begin{align*}
\int_\bbr |(\wt v^R)_\xi|^2 |(\wt v^S)^{-\mb X}-v_m|^2 d\xi 
&\le C\deltar\deltas^2 e^{-C\deltas t} \int_{\bbr} |(\wt v^R)_\xi| d\xi +C\deltar^2\deltas^2 e^{-Ct} \int_{\bbr} e^{-C |\xi+\s t|}  d\xi  \\
&\le C\deltar^2 \deltas^2 e^{-C \deltas t}.
\end{align*}
\end{proof}

\subsection{Construction of weight function}
We define the weight function $a$ by
\begin{equation}\label{weight}
a(\xi):=1+\frac{\lam}{\deltas}(p(v_m)-p(\wt v^S(\xi))),
\end{equation}
where the constant $\lam$ is chosen to be so small but far bigger than $\deltas$ such that 
\beq\label{lamsmall}
\deltas\ll \lam \le C\sqrt{\deltas}.
\eeq
Notice that 
\begin{equation}\label{a-bound}
1<a(\xi)<1+\lam,
\end{equation}
and 
\begin{equation}\label{a-prime}
a^\prime(\xi)=-\frac{\lam}{\deltas}p^\prime(\wt v^S)\wt v^S_\xi>0,
\end{equation}
and so,
\beq\label{d-weight}
|a'|\sim \frac{\lam}{\deltas} |\wt v^S_\xi|.
\eeq

\subsection{Relative entropy method} \label{ssec:ent}
We rewrite \eqref{hNS-1} into the  viscous hyperbolic system of conservation laws:
\beq\label{system-0}
\partial_t U +\partial_\xi A(U)= { \big(\ln v\big)_{\xi\xi} \choose 0},
\eeq
where 
\[
U:={v \choose h},\quad A(U):={-\s v -h \choose -\s h+p(v)}.
\]
Consider the entropy $\eta(U):=\frac{h^2}{2}+Q(v)$ of \eqref{system-0}, where $Q(v)=\frac{v^{-\gamma+1}}{\gamma-1}$, i.e., $Q'(v)=-p(v)$.\\
To write the above viscous term in terms of the derivative of the entropy:
\beq\label{nablae}
\nabla\eta(U)={-p(v)\choose h},
\eeq
 we observe that
\[
\big(\ln v\big)_{\xi\xi}=\left(\frac{(-p(v))_\xi}{-p'(v)v}\right)_{\xi},
\]
especially, by $-p'(v)v=\gamma p(v)$,
\[
\big(\ln v\big)_{\xi\xi}=\left(\frac{(-p(v))_\xi}{\gamma p(v)}\right)_{\xi} .
\]
Thus, using the nonnegative matrix 
\[
M(U):={\frac{1}{\gamma p(v)} \quad 0 \choose\quad\  0\quad\ 0},
\]
the above system \eqref{system-0} can be rewritten as
\beq\label{system-vh}
\partial_t U +\partial_\xi A(U)= \partial_\xi\Big(M(U) \partial_\xi\nabla\eta(U) \Big).
\eeq

Let
\begin{align}
\begin{aligned}\label{tilvh}
\wt U(t,\xi):= {\di \wt v(t,\xi) \choose \di \wt h(t,\xi)} = {\wt v^R(t,\xi)+(\wt v^S)^{-\mb{X}}(\xi)-v_m \choose \wt u^R(t,\xi)+(\wt h^S)^{-\mb{X}}(\xi)-u_m}.
\end{aligned}
\end{align}
Note that \eqref{SWS} can be written as
\beq\label{tilueq}
\partial_t \wt U  +\partial_\xi A(\wt U)= \partial_\xi\Big(M(\wt U) \partial_\xi\nabla\eta(\wt U) \Big) -\dot{\mb{X}} \partial_\xi \big((\wt U^S)^{-\mb{X}} \big)+ \bmat{F_3}\\ {F_2}\emat ,
\eeq
where $F_2, F_3$ are defined in \eqref{ff1f2}, \eqref{f1f2} respectively.
%\[
%F_3=\big(\ln (\wt v^S)^{-\mb X} -\ln \wt v\big)_{\xi\xi},\quad F_2=\big(p(\wt v)-p(\wt v^R)-p((\wt v^S)^{-\mb X})\big)_\xi .
%\]
Consider the relative entropy functional defined by
\beq\label{defent}
\eta(U|V)=\eta(U)-\eta(V) -\nabla\eta(V) (U-V),
\eeq
and the relative flux defined by
\beq\label{defa}
A(U|V)=A(U)-A(V) -\nabla A(V) (U-V).
\eeq
Let $G(\cdot;\cdot)$ be the flux of the relative entropy defined by
\beq\label{defg}
G(U;V) = G(U)-G(V) -\nabla \eta(V) (A(U)-A(V)),
\eeq
where $G$ is the entropy flux of $\eta$, i.e., $\partial_{i}  G (U) = \sum_{k=1}^{2}\partial_{k} \eta(U) \partial_{i}  A_{k} (U),\quad 1\le i\le 2$.\\
By a straightforward computation, for the system \eqref{system-0}, we have
\begin{align}
\begin{aligned}\label{relative_e}
&\eta(U|\wt U)=\frac{|h-\wt h |^2}{2} + Q(v|\wt v),\\
& A(U|\wt U)={0 \choose p(v|\wt v)},\\
&G(U;\wt U)=(p(v)-p(\wt v)) (h-\wt h)-\s \eta(U|\wt U),
\end{aligned}
\end{align}
where the relative pressure is defined as
\begin{equation}\label{pressure-relative}
p(v|w)=p(v)-p(w)-p'(w)(v-w).
\end{equation}

Below, we will estimate the relative entropy (weighted by $a(\xi)$ defined in \eqref{weight}) of the solution $U$ of \eqref{system-vh} w.r.t. the shifted wave \eqref{tilvh} as follows: 
\[
a^{-\mb X}(\xi)\eta\big(U(t,\xi)|\wt U(t,\xi) \big).
\]

\begin{lemma}\label{lem-zerovh}
Let $a$ be the weight function defined by \eqref{weight}. Let $U$ be a solution to \eqref{system-vh}, and $\wt U$ the shifted wave satisfying \eqref{tilvh}.
Then,
\begin{align}
\begin{aligned}\label{ineq-0}
\frac{d}{dt}\int_{\bbr} a^{-\mb X}(\xi)\eta\big(U(t,\xi)|\wt U(t,\xi) \big) d\xi =\dot{\mb X} (t) \mb{Y}(U) +\mathcal{J}^{bad}(U) - \mathcal{J}^{good}(U),
\end{aligned}
\end{align}
where
\begin{align}
\begin{aligned}\label{ybg-first}
&\mb Y(U):= - \int_{\bbr} \!a_\xi^{-\mb X} \eta(U|\wt U ) d\xi +\int_\bbr a^{-\mb X} \nabla^2\eta(\wt U) (\wt U^S)_\xi^{-\mb{X}}  (U-\wt U) d\xi,\\
&\mathcal{J}^{bad}(U):= \int_\bbr a_\xi^{-\mb X} \big(p(v)-p(\wt v)\big) \big(h-\wt h\big) d\xi + \s\int_\bbr a^{-\mb X}  (\wt v^S)^{-\mb{X}}_\xi p(v|\wt v) d\xi  \\
& \qquad\quad -\int_\bbr a_\xi^{-\mb X}\frac{p(v)-p(\wt v)}{\gamma p(v)}  \partial_\xi \big(p(v)-p(\wt v)\big)  d\xi +\int_\bbr a_\xi^{-\mb X} \big(p(v)-p(\wt v)\big)^2  \frac{\partial_{\xi} p(\wt v)}{\gamma p(v)p(\wt v)}   d\xi\\
& \qquad\quad   -\int_\bbr a^{-\mb X} \partial_\xi \big(p(v)-p(\wt v)\big)  \frac{p(\wt v)-p(v)}{\gamma p(v)p(\wt v)}\partial_{\xi} p(\wt v)  d\xi+\int_\bbr a^{-\mb X} (p(v)-p(\wt v)) F_3 d\xi \\
& \qquad\quad  -\int_\bbr a^{-\mb X} (h-\wt h) F_2 d\xi, \\
&\mathcal{J}^{good}(U):= \frac{\sigma}{2}\int_\bbr a_\xi^{-\mb X}\left| h-\wt h\right|^2 d\xi  +\sigma  \int_\bbr  a_\xi^{-\mb X} Q(v|\wt v) d\xi + \int_\bbr a^{-\mb X}  \wt u^R_\xi  p(v|\wt v) d\xi \\
& \qquad\quad +\int_\bbr \frac{a^{-\mb X}}{\gamma p(v)} |\partial_\xi \big(p(v)-p(\wt v)\big)|^2 d\xi.
\end{aligned}
\end{align}
\end{lemma}
\begin{remark}
Since $a'(\xi) >0$ and $u^R_\xi >0$ by Lemma \ref{lemma1.2}, $-\mathcal{J}^{good}$ consists of good terms, while $\mathcal{J}^{bad}$ consists of bad terms. 
\end{remark}

\begin{proof}
By the definition of the relative entropy with \eqref{system-vh} and \eqref{defent}, we first have
\begin{align*}
&\frac{d}{dt}\int_{\bbr}  a^{-\mb X}(\xi) \eta\big(U(t,\xi)|\wt U(t,\xi) \big) d\xi =-\dot{\mb X} (t)\int_{\bbr} \!a_\xi^{-\mb X} \eta(U|\wt U ) d\xi  \\
&\quad+\int_\bbr \!\!  a^{-\mb X} \bigg[\Big(\nabla\eta(U)-\nabla\eta(\wt U)\Big) \partial_t U  -\nabla^2\eta(\wt U) (U-\wt U) \partial_t \tilde U \bigg] d\xi\\
&=-\dot{\mb X} (t)\int_{\bbr} \!a_\xi^{-\mb X} \eta(U|\wt U ) d\xi  + \int_\bbr \!\!  a^{-\mb X}\bigg[\Big(\nabla\eta(U)-\nabla\eta(\wt  U)\Big)\!\Big(\!\!\!-\partial_\xi A(U)+ \partial_\xi\Big(M(U)\partial_\xi\nabla\eta(U) \Big) \Big)\\
&\qquad -\nabla^2\eta(\wt U) (U-\wt U) \left(-\partial_\xi A(\wt U)+ \partial_\xi\Big(M(\wt U)\partial_\xi\nabla\eta(\wt U) \Big)-\dot{\mb{X}} \partial_\xi \big((\wt U^S)^{-\mb{X}} \big)+ \bmat{F_3}\\ {F_2}\emat \right)  \bigg] d\xi.\\
\end{align*}
Using the definitions \eqref{defa} and \eqref{defg} with the same computation as in \cite[Lemma 4]{Vasseur_Book}) (see also \cite[Lemma 2.3]{Kang-V-NS17}), we have
\[
\frac{d}{dt}\int_{\bbr} a^{-\mb X}(\xi)\eta\big(U(t,\xi)|\wt U(t,\xi) \big) d\xi  =\dot{\mb X} (t) \mb Y(U)  +\sum_{i=1}^6 I_i,
\]
\begin{align}
\begin{aligned}\label{genrhs}
&I_1:=-\int_\bbr a^{-\mb X} \partial_\xi G(U;\wt U) d\xi,\\
&I_2:=- \int_\bbr a^{-\mb X} \partial_\xi \nabla\eta(\wt U) A(U|\wt U) d\xi,\\
&I_3:=\int_\bbr a^{-\mb X} \Big( \nabla\eta(U)-\nabla\eta(\wt U)\Big) \partial_\xi \Big(M(U) \partial_\xi \big(\nabla\eta(U)-\nabla\eta(\wt U)\big) \Big)  d\xi, \\
&I_4:=\int_\bbr a^{-\mb X} \Big( \nabla\eta(U)-\nabla\eta(\wt U)\Big) \partial_\xi \Big(\big(M(U)-M(\wt U)\big) \partial_\xi \nabla\eta(\wt U) \Big)  d\xi, \\
&I_5:=\int_\bbr a^{-\mb X}(\nabla\eta)(U|\wt U)\partial_\xi \Big(M(\wt U) \partial_\xi \nabla\eta(\wt U) \Big)  d\xi,\\
&I_6:=-\int_\bbr a^{-\mb X} \nabla^2\eta(\wt U) (U-\wt U)   \bmat{F_3}\\ {F_2}\emat   d\xi.
\end{aligned}
\end{align}
Using \eqref{relative_e} and \eqref{nablae}, we have
\begin{align*}
\begin{aligned}
I_1&=\int_\bbr  a_\xi^{-\mb X} G(U;\wt U) d\xi = \int_\bbr a_\xi^{-\mb X} \Big(\big(p(v)-p(\wt v)\big) \big(h-\wt h\big)  -\sigma \eta(U|\wt U)\Big)  d\xi\\
&= \int_\bbr a_\xi^{-\mb X} \big(p(v)-p(\wt v)\big) \big(h-\wt h\big) d\xi -\frac{\sigma}{2}\int_\bbr a_\xi^{-\mb X}\left| h-\wt h\right|^2 d\xi  -\sigma  \int_\bbr  a_\xi^{-\mb X} Q(v|\wt v) d\xi ,\\
I_2&=-\int_\bbr a^{-\mb X} \wt h_\xi p(v|\wt v) d\xi.
\end{aligned}
\end{align*}
By integration by parts, we have
\begin{align*}
\begin{aligned}
I_3&=\int_\bbr a^{-\mb X} \big(p(v)-p(\wt v)\big)\partial_{\xi}\Big(\frac{1}{\gamma p(v)} \partial_{\xi}\big(p(v)-p(\wt v)\big) \Big) d\xi \\
&=-\int_\bbr \frac{a^{-\mb X}}{\gamma p(v)} |\partial_\xi \big(p(v)-p(\wt v)\big)|^2 d\xi -\int_\bbr a_\xi^{-\mb X}\frac{p(v)-p(\wt v)}{\gamma p(v)}  \partial_\xi \big(p(v)-p(\wt v)\big)  d\xi,\\
I_4&= \int_\bbr a^{-\mb X} \big(p(v)-p(\wt v)\big)\partial_{\xi}\Big( \frac{p(\wt v)-p(v)}{\gamma p(v)p(\wt v)}  \partial_{\xi} p(\wt v) \Big) d\xi \\
&=\int_\bbr a_\xi^{-\mb X} \big(p(v)-p(\wt v)\big)^2  \frac{\partial_{\xi} p(\wt v)}{\gamma p(v)p(\wt v)}   d\xi  -\int_\bbr a^{-\mb X} \partial_\xi \big(p(v)-p(\wt v)\big)  \frac{p(\wt v)-p(v)}{\gamma p(v)p(\wt v)}\partial_{\xi} p(\wt v)  d\xi.
\end{aligned}
\end{align*}
Using \eqref{nablae} and
\beq\label{sec-ent}
\nabla^2\eta(U)= \begin{pmatrix}
  -p'(v) & 0 \\
  0 & 1 
 \end{pmatrix},
\eeq
we have
\[
I_5 =-\int_\bbr a^{-\mb X} p(v|\wt v) (\ln\wt v)_{\xi\xi} d\xi,
\]
and
\[
I_6= \int_\bbr a^{-\mb X} p'(\wt v) (v-\wt v) F_3 d\xi -\int_\bbr a^{-\mb X} (h-\wt h) F_2 d\xi .
\]
Especially, since
\[
I_6= \underbrace{-\int_\bbr a^{-\mb X} p(v|\wt v) F_3 d\xi}_{=:I_7} +\int_\bbr a^{-\mb X} (p(v)-p(\wt v)) F_3 d\xi -\int_\bbr a^{-\mb X} (h-\wt h) F_2 d\xi ,
\]
we use \eqref{huvx} and \eqref{hVS} to have
\begin{align*}
\begin{aligned}
I_2+I_5+I_7&=- \int_\bbr a^{-\mb X} \left(\wt u^R_\xi + (\wt h^S)^{-\mb{X}}_\xi  +(\ln \wt v^S)^{-\mb X}_{\xi\xi} \right) p(v|\wt v) d\xi\\
&=-\int_\bbr a^{-\mb X}  \left(\wt u^R_\xi -\s   (\wt v^S)^{-\mb{X}}_\xi \right) p(v|\wt v) d\xi.
\end{aligned}
\end{align*}
Therefore, we have
\begin{align*}
\begin{aligned}
&\frac{d}{dt}\int_{\bbr}  a^{-\mb X}(\xi) \eta\big(U(t,\xi)|\wt U(t,\xi) \big) d\xi \\
& =\dot{\mb X} (t) \mb{Y}(U) + \int_\bbr a_\xi^{-\mb X} \big(p(v)-p(\wt v)\big) \big(h-\wt h\big) d\xi -\frac{\sigma}{2}\int_\bbr a_\xi^{-\mb X}\left| h-\wt h\right|^2 d\xi  -\sigma  \int_\bbr  a_\xi^{-\mb X} Q(v|\wt v) d\xi  \\
&\quad -\int_\bbr a^{-\mb X}  \left(\wt u^R_\xi -\s (\wt v^S)^{-\mb{X}}_\xi \right) p(v|\wt v) d\xi +\int_\bbr a^{-\mb X} (p(v)-p(\wt v)) F_3 d\xi -\int_\bbr a^{-\mb X} (h-\wt h) F_2 d\xi \\
&\quad -\int_\bbr \frac{a^{-\mb X}}{\gamma p(v)} |\partial_\xi \big(p(v)-p(\wt v)\big)|^2 d\xi -\int_\bbr a_\xi^{-\mb X}\frac{p(v)-p(\wt v)}{\gamma p(v)}  \partial_\xi \big(p(v)-p(\wt v)\big)  d\xi\\
&\quad +\int_\bbr a_\xi^{-\mb X} \big(p(v)-p(\wt v)\big)^2  \frac{\partial_{\xi} p(\wt v)}{\gamma p(v)p(\wt v)}   d\xi  -\int_\bbr a^{-\mb X} \partial_\xi \big(p(v)-p(\wt v)\big)  \frac{p(\wt v)-p(v)}{\gamma p(v)p(\wt v)}\partial_{\xi} p(\wt v)  d\xi.
\end{aligned}
\end{align*}
\end{proof}

\subsection{Maximization in terms of $h-\wt h$}\label{sec:mini}

On the right-hand side of \eqref{ineq-0}, we will use Lemma \ref{lem-poin} for the diffusion term in order to control the bad terms only related to the perturbation $p(v)-p(\wt v)$ (or $v-\wt v$).  Therefore, we will rewrite $\mathcal{J}^{bad}$ into the maximized representation in terms of $h-\wt h$ in the following lemma.

\begin{lemma}\label{lem-max}
Let $a:\bbr\to\bbr^+$ be as in \eqref{weight}, and $\wt U$ be the shifted wave as in \eqref{tilvh}. 
Then, for any $U\in \bbr^+\times\bbr$, 
\begin{align}
\begin{aligned}\label{ineq-1}
\mathcal{J}^{bad} (U) -\mathcal{J}^{good} (U)= \mathcal{B}(U)- \mathcal{G}(U),
\end{aligned}
\end{align}
where
\begin{align}
\begin{aligned}\label{badgood}
&\mathcal{B}(U):= \frac{1}{2\sigma} \int_\bbr a_\xi^{-\mb X} \big|p(v)-p(\wt v)\big|^2 d\xi + \s\int_\bbr a^{-\mb X}  (\wt v^S)^{-\mb{X}}_\xi p(v|\wt v) d\xi  \\
& \qquad\quad -\int_\bbr a_\xi^{-\mb X}\frac{p(v)-p(\wt v)}{\gamma p(v)}  \partial_\xi \big(p(v)-p(\wt v)\big)  d\xi +\int_\bbr a_\xi^{-\mb X} \big(p(v)-p(\wt v)\big)^2  \frac{\partial_{\xi} p(\wt v)}{\gamma p(v)p(\wt v)}   d\xi\\
& \qquad\quad   -\int_\bbr a^{-\mb X} \partial_\xi \big(p(v)-p(\wt v)\big)  \frac{p(\wt v)-p(v)}{\gamma p(v)p(\wt v)}\partial_{\xi} p(\wt v)  d\xi+\int_\bbr a^{-\mb X} (p(v)-p(\wt v)) F_3 d\xi \\
& \qquad\quad  -\int_\bbr a^{-\mb X} (h-\wt h) F_2 d\xi, \\
&\mathcal{G}(U):= \frac{\sigma}{2}\int_\bbr a_\xi^{-\mb X}\left| h-\wt h-\frac{p(v)-p(\wt v)}{\sigma}\right|^2 d\xi  +\sigma  \int_\bbr  a_\xi^{-\mb X} Q(v|\wt v) d\xi + \int_\bbr a^{-\mb X}  \wt u^R_\xi  p(v|\wt v) d\xi \\
& \qquad\quad +\int_\bbr \frac{a^{-\mb X}}{\gamma p(v)} |\partial_\xi \big(p(v)-p(\wt v)\big)|^2 d\xi.
\end{aligned}
\end{align}

\begin{remark}\label{rem:0}
Since $\sigma a_\xi >0$ and $a>0$, $-\mathcal{G}$ consists of four good terms. 
\end{remark}

\end{lemma}
\begin{proof}
Let $J_1$ and $J_2$ be  the first terms of $\mathcal{J}^{bad}(U)$ and $-\mathcal{J}^{good}(U)$ respectively:
\begin{align*}
\begin{aligned}
&J_1:= \int_\bbr a_\xi^{-\mb X} \big(p(v)-p(\wt v)\big) \big(h-\wt h\big) d\xi,\\
&J_2:= -\frac{\sigma}{2}\int_\bbr a_\xi^{-\mb X}\left| h-\wt h\right|^2 d\xi.
\end{aligned}
\end{align*}
Applying the quadratic identity $\alpha z^2+ \beta z =\alpha(z+\frac{\beta}{2\alpha})^2-\frac{\beta^2}{4\alpha}$ with $z:=h-\wt h$ to the integrands of $J_1+J_2$, we have
\begin{align*}
\begin{aligned}
- \frac{\sigma}{2} \left| h-\wt h\right|^2 +\big(p(v)-p(\wt v)\big) (h-\wt h)  = - \frac{\sigma}{2} \left| h-\wt h -\frac{p(v)-p(\wt v)}{\sigma}\right|^2+ \frac{1}{2\sigma}|p(v)-p(\wt v)|^2.
\end{aligned}
\end{align*}
Therefore, we have the desired representation \eqref{ineq-1}-\eqref{badgood}.
\end{proof}

\subsection{Proof of Lemma \ref{lem-zvh}}
First of all, using Lemma \ref{lem-zerovh} and Lemma \ref{lem-max} together with a change of variable $\xi\mapsto\xi+\mb X(t)$, we have
\beq\label{mainrhs}
\frac{d}{dt}\int_{\bbr} a \eta\big(U^{\mb X}|\wt U^{\mb X} \big) d\xi =\dot{\mb X} (t) \mb{Y}(U^{\mb X}) +\mathcal{B}(U^{\mb X}) - \mathcal{G}(U^{\mb X}),
\eeq
where note from \eqref{tilvh} that 
\[
\wt U^{\mb X}:= {\wt v^{\mb X} \choose \wt h^{\mb X}} = {(\wt v^R)^{\mb X}+\wt v^S-v_m \choose (\wt u^R)^{\mb X}+\wt h^S-u_m}.
\]
For the bad terms and good terms, we use the following notations:
\begin{align}
\begin{aligned}\label{sbg}
&\mathcal{B}(U):=\sum_{i=1}^5 \mb B_i(U) +\mb S_1(U)+\mb S_2(U),\\
&\mathcal{G}(U):=\mb G_1(U)+\mb G_2(U)+\mb G^R(U) +\mb D(U),
\end{aligned}
\end{align}
where
\begin{align*}
\begin{aligned}
&\mb B_1(U):= \frac{1}{2\sigma} \int_\bbr a_\xi \big|p(v)-p(\wt v^{\mb X})\big|^2 d\xi,\\
&\mb B_2(U):=  \s\int_\bbr a  (\wt v^S)_\xi p(v|\wt v^{\mb X}) d\xi,  \\
&\mb B_3(U):=-\int_\bbr a_\xi\frac{p(v)-p(\wt v^{\mb X})}{\gamma p(v)}  \partial_\xi \big(p(v)-p(\wt v^{\mb X})\big)  d\xi ,\\
&\mb B_4(U):= \int_\bbr a_\xi \big(p(v)-p(\wt v^{\mb X})\big)^2  \frac{\partial_{\xi} p(\wt v^{\mb X})}{\gamma p(v)p(\wt v^{\mb X})}   d\xi,\\
&\mb B_5(U):= -\int_\bbr a \partial_\xi \big(p(v)-p(\wt v^{\mb X})\big)  \frac{p(\wt v^{\mb X})-p(v)}{\gamma p(v)p(\wt v^{\mb X})}\partial_{\xi} p(\wt v^{\mb X})  d\xi,\\
&\mb S_1(U):=\int_\bbr a (p(v)-p(\wt v^{\mb X})) \big(\ln \wt v^S -\ln \wt v^{\mb X}\big)_{\xi\xi} d\xi ,\\
&\mb S_2(U):=-\int_\bbr a (h-\wt h^{\mb X}) \big(p(\wt v^{\mb X})-p((\wt v^R)^{\mb X})-p(\wt v^S)\big)_\xi  d\xi,
\end{aligned}
\end{align*}
and
\begin{align*}
\begin{aligned}
&\mb G_1(U):= \frac{\sigma}{2}\int_\bbr a_\xi \left| h-\wt h^{\mb X}-\frac{p(v)-p(\wt v^{\mb X})}{\sigma}\right|^2 d\xi,\\
&\mb G_2(U):= \sigma  \int_\bbr  a_\xi Q(v|\wt v^{\mb X}) d\xi,\\
&\mb G^R(U):= \int_\bbr a  (\wt u^R_\xi)^{\mb X}  p(v|\wt v^{\mb X}) d\xi, \\
&\mb D(U):= \int_\bbr \frac{a}{\gamma p(v)} |\partial_\xi \big(p(v)-p(\wt v^{\mb X})\big)|^2 d\xi.
\end{aligned}
\end{align*}
For notational simplicity in this section, we omit the dependence of the solution on the shift, i.e., $(v,h)=(v^{\mb X},h^{\mb X})$.\\
First, note from \eqref{apri-ass} with the change of variable $\xi\mapsto\xi+\mb X(t)$ that
\beq\label{smpressure}
\|p(v)-p(\wt v^{\mb X})\|_{L^\infty((0,T)\times\bbr)}\le C\|v-\wt v^{\mb X}\|_{L^\infty((0,T)\times\bbr)} \le C \|v-\wt v^{\mb X}\|_{L^\infty(0,T;H^1(\bbr))} \le C\eps_1.
\eeq
Since the diffusion term $\mb D$ is related to the small perturbation of pressure, we will perform the Taylor expansion near $p(\wt v^{\mb X})$ for the leading order terms and then use Lemma \ref{lem-poin} on the sharp Poincar\'e inequality in the following lemma.\\
For $\mb Y$, we have from \eqref{relative_e} and \eqref{sec-ent} that
\begin{align*}
\begin{aligned}
\mb Y(U)&= - \int_{\bbr} \!a_\xi \eta(U|\wt U^{\mb X} ) d\xi +\int_\bbr a \nabla^2\eta(\wt U^{\mb X}) (\wt U^S)_\xi  (U-\wt U^{\mb X}) d\xi\\
&= -\int_{\bbr} \!a_\xi \left(\frac{|h-\wt h^{\mb X}|^2}{2} +Q(v|\wt v^{\mb X})  \right) d\xi \\
&\quad +\int_{\bbr} a  {\wt h}^S_\xi (h-\wt h^{\mb X}) d\xi-\int_{\bbr} a p^\prime(\wt v^{\mb X}) {\wt v}^S_\xi (v-\wt v^{\mb X}) d\xi.
\end{aligned}
\end{align*}
We decompose the functional $\mb Y$ as follows:
\[
\mb Y:= \sum_{i=1}^6\mb{Y}_i,
\]
where
\begin{align*}
\begin{aligned}
&\mb Y_1(U):= \int \frac{a}{\sigma} \wt h^S_\xi(p(v)-p(\wt v^{\mb{X}}))d\xi,\\
&\mb Y_2(U):= -\int a p^\prime(\wt v^S) \wt v^S_\xi (v-\wt v^{\mb{X}}) d\xi,\\
&\mb Y_3(U):=\int  a \wt h^S_\xi \left( h-\wt h^{\mb{X}}-\frac{p(v)-p(\wt v^{\mb{X}})}{\sigma}\right)d\xi,\\
&\mb Y_4(U):=-\int a (p'(\wt v^{\mb{X}})-p'(\wt v^S)) \wt v^S_\xi (v-\wt v^{\mb{X}})  d\xi,\\
&\mb Y_5(U):=  -\frac12\int_{\bbr} \!a_\xi \left( h-\wt h^{\mb{X}}-\frac{p(v)-p(\wt v^{\mb{X}})}{\sigma}\right)\left( h-\wt h^{\mb{X}}+\frac{p(v)-p(\wt v^{\mb{X}})}{\sigma}\right)d\xi, \\
&\mb Y_6(U):= -\int  a_\xi Q(v|\wt v^{\mb{X}})d\xi-\int \frac{a_\xi}{2\sigma^2}(p(v)-p(\wt v^{\mb{X}}))^2d\xi.
\end{aligned}
\end{align*}
Notice from \eqref{X(t)} that 
\beq\label{defxy}
\dot{\mb{X}}(t)=-\frac{M}{\delta_S} (\mb{Y}_1+\mb{Y}_2),
\eeq
and so,
\begin{equation}\label{XY}
\dot{\mb{X}}(t)\mb{Y}=-\frac{\delta_S}{M} |\dot{\mb{X}}(t)|^2+\dot{\mb{X}}(t)\sum_{i=3}^6\mb{Y}_i.
\end{equation}

\subsubsection{Leading order estimates}
\begin{lemma}\label{lem-sharp}
There exists $C>0$ such that 
\begin{align*}
\begin{aligned}
&-\frac{\deltas}{2M} |\dot{\mb X}|^2 + \mb B_1+\mb B_2 -\mb G_2 -\frac{3}{4}\mb D \\&\le -C\int |(\wt v^S)_\xi| |p(v)-p(\wt v^{\mb{X}})|^2 d\xi +C\int |a_\xi| |p(v)-p(\wt v^{\mb{X}})|^3 d\xi \\
&\quad+C\int |a_\xi| |(\wt v^R)^{\mb{X}} -v_m | |p(v)-p(\wt v^{\mb{X}})|^2 d\xi.
\end{aligned}
\end{align*}
\end{lemma}
\begin{proof}
We first rewrite the main terms in terms of the new variables $y$ and $w$:
\begin{equation}\label{omega}
w :=p(v)-p(\wt v^{\mb{X}}),
\end{equation}
and
\begin{equation}\label{y-xi}
y:=\frac{p(v_m)-p(\wt v^S(\xi))}{\deltas}.
\end{equation}
Note that
\begin{equation}
\frac{dy}{d\xi} =-\frac{1}{\delta_S}p(\wt v^S)_\xi>0,
\end{equation}
and the change of variable $\xi\in\bbr\mapsto y\in (0,1)$ will be used below.\\
Note also that $a(\xi)=1+\lam y$ and so $a'(\xi)=\lam (dy/d\xi)$.\\
To perform the sharp estimates, we will consider the $O(1)$-constants:
\[
\s_m:=\sqrt{-p'(v_m)},\qquad \alpha_m:=\frac{\gamma+1}{2\gamma\s_m p(v_m)},
\]
which are indeed independent of the small constants $\deltas,\deltar$, since $\frac{v_+}{2}\le v_m\le v_+$. \\
Note that
\beq\label{sm1}
|\sigma-\sigma_m|\le C\delta_S,
\eeq
with together with $\s_m^2 = -p'(v_m)=\gamma p(v_m)^{\frac1\gamma+1}$ implies
\beq\label{sm2}
 |\s_m^2 -|p'(\wt v^S)||\le C\deltas,\quad \left|\frac{1}{\sigma_m^2}-\frac{p(\wt v^S)^{-\frac1\gamma-1}}{\gamma}\right|\le C\delta_S,\quad \left|\frac{1}{\sigma_m^2}-\frac{p(\wt v^{\mb X})^{-\frac1\gamma-1}}{\gamma}\right|\le C\delta_0 .
\eeq

\noindent$\bullet$ {\bf Estimate on $-\frac{\deltas}{2M} |\dot{\mb X}|^2$ :}
First, to estimate the term $-\frac{\deltas}{2M} |\dot{\mb X}|^2$, we will estimate $\mb Y_1, \mb Y_2$ due to \eqref{defxy}.\\
By the change of variable, we have
\[
\mb Y_1=-\frac{\deltas}{\s^2}\int_0^1 a w dy.
\]
Using \eqref{sm1} and $|a-1|\le \lam$, we have
\beq\label{y1}
\left|\mb Y_1 + \frac{\deltas}{\s_m^2}\int_0^1 w dy \right| \le C\deltas (\lam+\delta_0) \int_0^1|w| dy.
\eeq
For 
\[
\mb Y_2 = -\int a p(\wt v^S)_\xi (v-\wt v^{\mb{X}}) d\xi = \deltas\int_0^1 a (v-\wt v^{\mb{X}}) dy,
\] 
we observe that since (by considering $v=p(v)^{-\frac1\gamma}$)
\[
\left|v-\wt v^{\mb{X}}+\frac{p(\wt v^{\mb{X}})^{-\frac{1}{\gamma}-1}}{\gamma} (p(v)-p(\wt v^{\mb{X}}) )\right| \le C|p(v)-p(\wt v^{\mb{X}})|^2,
\]
it follows from \eqref{sm2} and \eqref{smpressure} that
\[
\left| v-\wt v^{\mb{X}} +\frac{1}{\s_m^2}(p(v)-p(\wt v^{\mb{X}})) \right| \le C(\delta_0+\eps_1)|p(v)-p(\wt v^{\mb{X}})|.
\]
This implies
\beq\label{y2}
\left|\mb Y_2 + \frac{\deltas}{\s_m^2}\int_0^1 w dy \right| \le C\deltas (\lam+\delta_0+\eps_1)\int_0^1|w| dy.
\eeq
Therefore, by \eqref{defxy}, \eqref{y1} and \eqref{y2}, we have
\[
\left| \dot{\mb X} -\frac{2M}{\s_m^2} \int_0^1 w dy\right| = \left| \sum_{i=1}^2\frac{M}{\deltas} \left(\mb Y_i + \frac{\deltas}{\s_m^2}\int_0^1 w dy \right) \right| \le C (\lam+\delta_0+\eps_1)\int_0^1|w| dy,
\]
which yields
\[
\left( \left|\frac{2M}{\s_m^2} \int_0^1 w dy\right| - |\dot{\mb X} | \right)^2 \le  C(\lam+\delta_0+\eps_1)^2 \int_0^1|w|^2 dy.
\]
This and the algebraic inequality $\frac{p^2}{2}-q^2 \le (p-q)^2$ for all $p,q\ge 0$ imply
\[
\frac{2M^2}{\s_m^4} \left(\int_0^1 w dy\right)^2 - |\dot{\mb X}|^2 \le  C(\lam+\delta_0+\eps_1)^2 \int_0^1|w|^2 dy.
\]
Thus,
\beq\label{gxest}
-\frac{\deltas}{2M} |\dot{\mb X}|^2 \le -\frac{M\deltas}{\s_m^4} \left(\int_0^1 w dy\right)^2 +C\ds(\lam+\delta_0+\eps_1)^2 \int_0^1|w|^2 dy.
\eeq
\noindent$\bullet$ {\bf Change of variable for $\mb B_1, \mb B_2$ :}
By the change of variable, we have
\[
\mb B_1 =\frac{\lam}{2\s}\int_0^1 w^2 dy,
\]
which together with \eqref{sm1} yields
\beq\label{b1s}
\mb B_1\le \frac{\lam}{2\s_m}\int_0^1 w^2 dy + C\lam\deltas \int_0^1 w^2 dy.
\eeq
For $\mb B_2$, using $(\wt v^S)_\xi =p(\wt v^S)_\xi/p'(\wt v^S)$ and  the change of variable, we have
\[
\mb B_2 = \s \deltas \int_0^1 (1+\lam y)\frac{1}{|p'(\wt v^S)|} p(v|\wt v^{\mb{X}}) dy.
\]
Using \eqref{p-est1} with \eqref{smpressure}, we have
\beq\label{b2s}
\mb B_2 \le \s \deltas(1+\lam) \int_0^1 \left(\frac{1}{|p'(\wt v^S)|}  \left(\frac{\gamma+1}{2\gamma p(\wt v^{\mb{X}})} + C\eps_1\right) |p(v)-p(\wt v^{\mb{X}})|^2 \right) dy,
\eeq
which together with \eqref{sm1}-\eqref{sm2} yields
\[
\mb B_2 \le \deltas  \alpha_m (1+C(\delta_0+\lam+\eps_1)) \int_0^1 w^2 dy.
\]
\noindent$\bullet$ {\bf Change of variable for $\mb G_2$ :}
For $\mb G_2$, we first use \eqref{Q-est11} with \eqref{smpressure} to split it into two parts:
\begin{align}
\begin{aligned} \label{ig2}
\mb G_2 &\ge \sigma \int_\bbr  a_\xi  \frac{p(\wt v^{\mb X})^{-\frac{1}{\gamma}-1}}{2\gamma}|p(v)-p(\wt v^{\mb X})|^2 d\xi -\sigma  \int_\bbr  a_\xi \frac{1+\gamma}{3\gamma^2} p(\wt v^{\mb X})^{-\frac{1}{\gamma}-2}(p(v)-p(\wt v^{\mb X}))^3  d\xi\\
&=\underbrace{ \sigma \int_\bbr  a_\xi  \frac{p(\wt v^S)^{-\frac{1}{\gamma}-1}}{2\gamma}|p(v)-p(\wt v^{\mb X})|^2 d\xi }_{=:\mathcal{G}_2} -\sigma  \int_\bbr  a_\xi \frac{1+\gamma}{3\gamma^2} p(\wt v^{\mb X})^{-\frac{1}{\gamma}-2}(p(v)-p(\wt v^{\mb X}))^3  d\xi
\\
&\quad + \frac{\s}{2\gamma} \int_\bbr  a_\xi  \left(p(\wt v^{\mb X})^{-\frac{1}{\gamma}-1} - p(\wt v^S)^{-\frac{1}{\gamma}-1} \right) |p(v)-p(\wt v^{\mb X})|^2 d\xi.
\end{aligned}
\end{align}
We only do the change of variable for the good term $\mathcal{G}_2$ as follows: by \eqref{sm1}-\eqref{sm2} and the change of variable,
\[
\mathcal{G}_2 \ge  \frac{1}{2\s_m}  (1 -C\deltas)  \int_\bbr  a_\xi |p(v)-p(\wt v^{\mb X})|^2 d\xi = \frac{\lam}{2\s_m}  (1 -C\deltas) \int_0^1 w^2 dy.
\]
This and \eqref{b1s} yield
\beq\label{b1-g2}
\mb B_1 -\mathcal{G}_2 \le C \lam \deltas  \int_0^1 w^2 dy.
\eeq
\noindent$\bullet$ {\bf Change of variable for $\mb D$ :}
First, using $a\ge 1$ and the change of variable, we have
\[
\mb D\ge  \int_\bbr \frac{1}{\gamma p(v)} |\partial_\xi \big(p(v)-p(\wt v^{\mb X})\big)|^2 d\xi = \int_0^1 |\partial_y w|^2 \frac{1}{\gamma p(v)} \Big(\frac{dy}{d\xi}\Big) dy.
\]
Integrating \eqref{hVS} over $(-\infty,\xi]$ yields
\[
(\ln \wt v^S)_\xi = -\s (\wt v^S -v_m) - \frac{p(\wt v^S)-p(v_m)}{\s}.
\]
Since
\[
\deltas  \frac{1}{\gamma p(\wt v^S)} \Big(\frac{dy}{d\xi}\Big) =  \frac{-p(\wt v^S)_\xi}{\gamma p(\wt v^S)} =(\ln \wt v^S)_\xi,
\]
we have
\begin{align*}
\begin{aligned}
\deltas  \frac{1}{\gamma p(\wt v^S)} \Big(\frac{dy}{d\xi}\Big) &=-\s (\wt v^S -v_m) - \frac{p(\wt v^S)-p(v_m)}{\s}\\
&=-\frac{1}{\s}\left( \s^2(\wt v^S -v_m) + (p(\wt v^S)-p(v_m)) \right),
\end{aligned}
\end{align*}
which together with $\s^2=\frac{p(v_m)-p(v_+)}{v_+-v_m}$ yields
\begin{align*}
\begin{aligned}
\deltas  \frac{1}{\gamma p(\wt v^S)} \Big(\frac{dy}{d\xi}\Big)
&=-\frac{1}{\s(v_+ -v_m)}\left( (p(v_m)-p(v_+)) (\wt v^S -v_m) + (v_+-v_m)(p(\wt v^S)-p(v_m)) \right)\\
&=-\frac{1}{\s(v_+ -v_m)}\bigg( (p(\wt v^S)-p(v_+)) (\wt v^S -v_m) +  (\wt v^S -v_m) (p(v_m)-p(\wt v^S)) \\
&\qquad +  (\wt v^S- v_m ) (p(\wt v^S)-p(v_m)) + (v_+- \wt v^S)(p(\wt v^S)-p(v_m)) \bigg) \\
&=-\frac{1}{\s(v_+ -v_m)}\left( (p(\wt v^S)-p(v_+)) (\wt v^S -v_m) + (v_+- \wt v^S)(p(\wt v^S)-p(v_m)) \right).
\end{aligned}
\end{align*}
Since $y=\frac{p(v_m)-p(\wt v^S)}{\deltas}$ and $1-y=\frac{p(\wt v^S)-p(v_+)}{\deltas}$,
\[
\frac{1}{y(1-y)}  \frac{1}{\gamma p(\wt v^S)} \Big(\frac{dy}{d\xi}\Big) =\frac{\deltas}{\s(v_+-v_m)}\left( \frac{v_m-\wt v^S}{p(v_m)-p(\wt v^S)}-\frac{v_+-\wt v^S}{p(v_+)-p(\wt v^S)} \right).
\]
Since the right-hand side above is the same as the one in the proofs of \cite[Appendix B]{KV-2shock} and \cite[Lemma 3.1]{Kang-V-NS17}), we have
\[
\left| \frac{1}{y(1-y)}  \frac{1}{\gamma p(\wt v^S)} \Big(\frac{dy}{d\xi}\Big) -\frac{\deltas p''(v_m)}{2|p'(v_m)|^2 \s_m} \right|\leq C\deltas^2.
\]
In addition, since \eqref{smpressure} yields $C^{-1}\le p(v)\le C$ and
\[
\left|\frac{p(\wt v^S)}{p(v)} -1 \right| \le C |\wt v^S -v| \le  C (|\wt v^S -\wt v| +|\wt v - v| ) \le C(\delta_0 + \eps_1),
\]
we have
\begin{eqnarray*}
\mb D&\ge&  \int_0^1 |\partial_y w|^2\frac{p(\wt v^S)}{p(v)} \frac{1}{\gamma p(\wt v^S)} \Big(\frac{dy}{d\xi}\Big) dy\\
&\ge& (1-C\delta_0 -C\eps_1) \left(\frac{\deltas p''(v_m)}{2|p'(v_m)|^2 \s_m} -C\deltas^2 \right)   \int_0^1y(1-y)  |\partial_y w|^2  \, dy.
\end{eqnarray*}
Since 
\[
\frac{p''(v_m)}{2|p'(v_m)|^2 \s_m} =\frac{\gamma+1}{2\gamma\s_m p(v_m)} =\alpha_m,
\]
we have
\[
\mb D\ge \deltas\alpha_m(1-C(\delta_0+\eps_1))  \int_0^1y(1-y)  |\partial_y w|^2 dy.
\]
\noindent$\bullet$ {\bf Conclusion :}
First, by \eqref{b2s}, \eqref{b1-g2} and the above estimates, we have
\[
\begin{array}{ll}
\di \mb B_1+\mb B_2 -\mathcal{G}_2 -\frac{3}{4}\mb D\\
\di \le  \deltas\alpha_m \left( (1+C(\delta_0+\lam+\eps_1)) \int_0^1 w^2 dy -\frac34(1-C(\delta_0+\eps_1))  \int_0^1y(1-y)  |\partial_y w|^2 dy \right),
 \end{array}
\]
which together with the smallness of $\lam,\delta_0,\eps_1$ yields
\[
\mb B_1+\mb B_2 -\mathcal{G}_2 -\frac{3}{4}\mb D\le \deltas\alpha_m \left( \frac98 \int_0^1 w^2 dy -\frac58  \int_0^1y(1-y)  |\partial_y w|^2 dy \right).
\]
Using Lemma \ref{lem-poin} and the fact that for $\bar w:=\int_0^1 w dy$,
\[
\int_0^1 |w-\bar w|^2 dy = \int_0^1w^2 dy -{\bar w}^2,
\]
we have
\[
 \mb B_1+\mb B_2 -\mathcal{G}_2 -\frac{3}{4}\mb D \le -\frac{\deltas\alpha_m}{8}  \int_0^1 w^2 dy +\frac{5\deltas\alpha_m}{4}  \left(\int_0^1 w dy\right)^2.
\]
Since the specific $O(1)$-constant $M$ satisfies
\beq\label{definitionM}
M=\frac{5}{4}\s_m^4\alpha_m,
\eeq
it holds from \eqref{gxest} and \eqref{ig2} that
\begin{align*}
\begin{aligned}
&-\frac{\deltas}{2M} |\dot{\mb X}|^2 + \mb B_1+\mb B_2 -\mb G_2 -\frac{3}{4}\mb D \\
&\quad\le -\frac{\alpha_m}{16}  \int_0^1 w^2 \deltas dy 
+\sigma  \int_\bbr  a_\xi \frac{1+\gamma}{3\gamma^2} p(\wt v^{\mb X})^{-\frac{1}{\gamma}-2}(p(v)-p(\wt v^{\mb X}))^3  d\xi\\
&\qquad-\frac{\s}{2\gamma} \int_\bbr  a_\xi  \left(p(\wt v^{\mb X})^{-\frac{1}{\gamma}-1} - p(\wt v^S)^{-\frac{1}{\gamma}-1} \right) |p(v)-p(\wt v^{\mb X})|^2 d\xi ,
\end{aligned}
\end{align*}
which implies the desired estimate.
\end{proof}

\subsubsection{Proof of Lemma \ref{lem-zvh}}
First of all, we use \eqref{mainrhs}, \eqref{sbg}, \eqref{XY} to have
\begin{align*}
\begin{aligned}
\frac{d}{dt}\int_{\bbr} a \eta\big(U|\wt U^{\mb X} \big) d\xi &=-\frac{\deltas}{2M} |\dot{\mb X}|^2 + \mb B_1+\mb B_2 -\mb G_2 -\frac{3}{4}\mb D \\
&\quad -\frac{\delta_S}{2M} |\dot{\mb{X}}|^2+\dot{\mb{X}}\sum_{i=3}^6\mb{Y}_i +\sum_{i=3}^5 \mb B_i +\mb S_1 +\mb S_2 - \mb G_1- \mb G^R - \frac14\mb D.
\end{aligned}
\end{align*}
Using Lemma \ref{lem-sharp} and the Young's inequality, we find that there exist $C_1, C>0$ such that
\begin{align*}
\begin{aligned}
\frac{d}{dt}\int_{\bbr} a \eta\big(U|\wt U^{\mb X} \big) d\xi 
&\le - C_1\int |(\wt v^S)_\xi| |p(v)-p(\wt v^{\mb{X}})|^2 d\xi +\underbrace{C\int |a_\xi| |p(v)-p(\wt v^{\mb{X}})|^3 d\xi}_{=:K_1} \\
&\quad +\underbrace{ C\int |a_\xi| |(\wt v^R)^{\mb{X}} -v_m | |p(v)-p(\wt v^{\mb{X}})|^2 d\xi}_{=:K_2} \\
&\quad -\frac{\delta_S}{4M} |\dot{\mb{X}}|^2 +\frac{C}{\deltas}\sum_{i=3}^6|\mb{Y}_i|^2 +\sum_{i=3}^5 \mb B_i +\mb S_1 +\mb S_2 - \mb G_1- \mb G^R - \frac14\mb D.
\end{aligned}
\end{align*}
In what follows, to control the above bad terms, we will use the above good terms $\mb G_1,\mb G^R, \mb D$ and
\beq\label{gs-1}
\mb G^S:=\int |(\wt v^S)_\xi| |p(v)-p(\wt v^{\mb{X}})|^2 d\xi.
\eeq
Note that from \eqref{good1} and \eqref{gs-1}, it is obvious that $\mb G^S=G^S$ with the change of variables $\xi\mapsto\xi+\mb X(t)$.

\noindent$\bullet$ {\bf Estimate on the cubic term $K_1$ :}
For simplicity, we use the notation $w=p(v)-p(\wt v^{\mb X})$ as in \eqref{omega}. 
We first use \eqref{d-weight} and the interpolation inequality to have
\begin{align*}
\begin{aligned}
K_1 &\le C\frac{\lam}{\deltas} \int  \|w\|_{L^\infty(\bbr)}^2 |(\wt v^S)_\xi| |w| d\xi\\
&\le C\frac{\lam}{\deltas} \|w\|_{L^\infty(\bbr)}^2 \sqrt{\int  |(\wt v^S)_\xi| w^2 d\xi} \sqrt{\int  |(\wt v^S)_\xi| d\xi}\\
&\le C\frac{\lam}{\sqrt{\deltas}}  \|w_\xi\|_{L^2(\bbr)} \|w\|_{L^2(\bbr)} \sqrt{\int |(\wt v^S)_\xi| w^2 d\xi}.
\end{aligned}
\end{align*}
Using \eqref{lamsmall}, \eqref{apri-ass} with \eqref{pressure2}, we have
\begin{align*}
\begin{aligned}
K_1 &\le  C\eps_1 \|w_\xi\|_{L^2(\bbr)} \sqrt{\int  |(\wt v^S)_\xi| w^2 d\xi}\\
&\le C\eps_1    \|w_\xi\|_{L^2(\bbr)}^2 +C\eps_1 \int  |(\wt v^S)_\xi| w^2 d\xi \le  \frac{1}{40}(\mb D + C_1\mb G^S).
\end{aligned}
\end{align*}
\noindent$\bullet$ {\bf Estimate on the term $K_2$ :}
Likewise, using  \eqref{d-weight} and the interpolation inequality,
\begin{align*}
\begin{aligned}
K_2 &\le C\frac{\lam}{\deltas} \|w\|_{L^4(\bbr)}^2 \| |(\wt v^S)_\xi| |(\wt v^R)^{\mb{X}} -v_m |\|_{L^2(\bbr)}  \\
&\le C\frac{\lam}{\deltas} \|w_\xi\|_{L^2(\bbr)}^{1/2}  \|w\|_{L^2(\bbr)}^{3/2} \| |(\wt v^S)_\xi| |(\wt v^R)^{\mb{X}} -v_m |\|_{L^2(\bbr)} .
\end{aligned}
\end{align*}
Using \eqref{apri-ass}, Lemma \ref{lemma2.2}, \eqref{lamsmall} and Young's inequality, it holds that
\begin{align*}
\begin{aligned}
K_2 &\le C\eps_1 \|w_\xi\|_{L^2(\bbr)}^{1/2} \frac{\lam}{\deltas} \deltas^{3/2} \deltar e^{-C\deltas t} \le C\eps_1 \|w_\xi\|_{L^2(\bbr)}^{1/2}  \deltas \deltar e^{-C\deltas t} \\
&\le C\eps_1 \|w_\xi\|_{L^2(\bbr)}^2 + C\eps_1 \deltas^{4/3} \deltar^{4/3} e^{-C\deltas t} \le  \frac{1}{40} \mb D +C\eps_1 \deltas^{4/3} \deltar^{4/3} e^{-C\deltas t}.
\end{aligned}
\end{align*}

\noindent$\bullet$ {\bf Estimates on the terms $\mb Y_i$ :}
Since
\[
|\mb Y_3| \le C\frac{\deltas}{\lam} \int |a_\xi| \left| h-\wt h^{\mb{X}}-\frac{p(v)-p(\wt v^{\mb{X}})}{\sigma}\right|d\xi \le C\frac{\deltas}{\sqrt\lam} \sqrt{\mb G_1},
\]
we have
\[
\frac{C}{\deltas}|\mb{Y}_3|^2 \le C\frac{\deltas}{\lam} \mb G_1\le\frac14 \mb G_1.
\]
Using \eqref{rel_Q} and \eqref{Q-est1}, we have
\[
|\mb Y_4|\le C\int | (\wt v^R)^{\mb{X}} -v_m| |\wt v^S_\xi| |v-\wt v^{\mb{X}}| d\xi \le C\deltar \int |\wt v^S_\xi| w d\xi \le C\deltar\sqrt{\deltas} \sqrt{\int |\wt v^S_\xi| w^2 d\xi },
\]
and so
\[
\frac{C}{\deltas}|\mb{Y}_4|^2 \le C\deltar^2 \mb G^S \le \frac{C_1}{40} \mb G^S.
\]

For $\mb Y_5$, we first estimate $h-\wt h^{\mb{X}}$ in terms of $u-\wt u^{\mb{X}}$ and $v-\wt v^{\mb{X}}$ (using the definition of $h$ in \eqref{h} and $\wt h$ in \eqref{huvx}) as follows.
Observe that
\begin{align}
\begin{aligned} \label{newph}
|h-\wt h^{\mb{X}}|&\le |u-\wt u^{\mb{X}}| + |(\ln v)_\xi-(\ln\wt v^S)_\xi|\\
 &\le  |u-\wt u^{\mb{X}}| +C (|(v-\wt v^{\mb{X}})_\xi|+|\wt v^{\mb{X}}_\xi||v-\wt v^S|+|(\wt v^R_\xi)^{\mb{X}}|)\\
&\le |u-\wt u^{\mb{X}}| +C (|(v-\wt v^{\mb{X}})_\xi|+|\wt v^{\mb{X}}_\xi||v-\wt v^{\mb{X}}|+|\wt v^S_\xi||(\wt v^R)^{\mb{X}}-v_m|+|(\wt v^R_\xi)^{\mb{X}}|),
\end{aligned}
\end{align}
which together with the wave interaction estimates in Lemma \ref{lemma2.2} and Lemma \ref{lemma1.2} implies 
\[
\|h-\wt h^{\mb{X}}\|_{L^2(\bbr)}\leq C\Big[\|u-\wt u^{\mb{X}}\|_{L^2(\bbr)}+\|v-\wt v^{\mb{X}}\|_{H^1(\bbr)}+\delta_R\Big].
\]
Then, by using \eqref{apri-ass},
\beq\label{hhe}
\|h-\wt h^{\mb{X}}\|_{L^\infty(0,T;L^2(\bbr))}\leq  C(\varepsilon_1+\delta_R).
\eeq
This together with \eqref{apri-ass} and $\|a_\xi\|_{L^\infty}\le C\lam\deltas$ yields
\[
\begin{array}{ll}
\di |\mb Y_5|\le C|\mb{G}_1|^{\frac12}\|a_\xi\|^{\frac12}_{L^\infty}\big[\|h-\wt h^{\mb{X}}\|_{L^\infty(0,T;L^2(\bbr))}+\|v-\wt v^{\mb{X}}\|_{L^\infty(0,T;L^2(\bbr))}\big]\\[2mm]
\di \qquad\leq  C(\varepsilon_1+\delta_R)(\lambda\delta_S)^{\frac12}\mb{G}_1^{\frac12},
\end{array}
\]
and so
\[
\frac{C}{\deltas}|\mb{Y}_5|^2\le C \lambda(\varepsilon_1+\delta_R)^2\mb{G}_1\le\frac14 \mb G_1.
\]

Using \eqref{Q-est1} with \eqref{smpressure}, we have
\[
\frac{C}{\deltas}|\mb{Y}_6|^2 \le \frac{C}{\deltas} \left(\int  |a_\xi| w^2 d\xi\right)^2
\le \frac{C\lam^2}{\deltas^3} \left(\int |(\wt v^S)_\xi| w^2 d\xi\right)^2.
\]
Thus, by  \eqref{apri-ass} with \eqref{pressure2}, we have
\[
\frac{C}{\deltas}|\mb{Y}_6|^2 \le \frac{C\lam^2}{\deltas} \|w\|_{L^2(\bbr)}^2 \int |(\wt v^S)_\xi| w^2 d\xi \le C\eps_1^2  \int |(\wt v^S)_\xi| w^2 d\xi\le \frac{C_1}{40} \mb G^S.
\]

\noindent$\bullet$ {\bf Estimates on the terms $\mb B_i$ :}
Using the Young's inequality, we have
\[
|\mb B_3(U)| \le \frac{1}{32}\mb D + C \int_\bbr |a_\xi|^2 w^2 d\xi \le \frac{1}{32}\mb D + \lam^2  \int_\bbr |(\wt v^S)_\xi| w^2 d\xi \le \frac{1}{40}(\mb D +C_1\mb G^S).
\]
For $\mb B_4, \mb B_5$, we use the facts that 
\[
|\partial_{\xi} p(\wt v^{\mb X})|\le C( |\wt v^S_\xi|+|(\wt u^R_\xi)^{\mb X}|) \quad \mbox{by Lemma \ref{lemma1.2}},
\]
and 
\[
|p(v)-p(\wt v^{\mb X})|^2 \le C p(v|\wt v^{\mb X})\quad \mbox{by \eqref{rel_p} and \eqref{pressure2}}.
\]
Then,
\[
|\mb B_4(U)| \le  C\lam\deltas  \int_\bbr ( |\wt v^S_\xi|+|(\wt u^R_\xi)^{\mb X}|)  \big(p(v)-p(\wt v^{\mb X})\big)^2 d\xi \le \frac18(C_1\mb G^S+\mb G^R).
\]
In addition, using Young's inequality and $\|(\wt u^R_\xi)^{\mb X}\|_{L^\infty} \le C\deltar $ by Lemma \ref{lemma1.2}, we have
\[
|\mb B_5(U)| \le \frac{1}{40} \mb D + C\delta_0  \int_\bbr ( |\wt v^S_\xi|+|(\wt u^R_\xi)^{\mb X}|)  \big(p(v)-p(\wt v^{\mb X})\big)^2 d\xi \le  \frac{1}{40} \mb D + \frac18(C_1\mb G^S+\mb G^R).
\]

\noindent$\bullet$ {\bf Estimates on the terms $\mb S_i$ :}
We first  compute that (using $\wt v^S,  \wt v^{\mb X}, (\wt v^R)^{\mb X} \in (v_+/2, 2v_+)$, $\wt v^{\mb X}=(\wt v^R)^{\mb X}+\wt v^S -v_m$, and Lemmas \eqref{lemma1.2}-\eqref{lemma1.3})
\begin{align}
\begin{aligned}\label{ccomp}
&| \big(\ln \wt v^S -\ln \wt v^{\mb X}\big)_{\xi\xi}| \\
&=\left|\wt v^S_{\xi\xi}\left(\frac{1}{\wt v^S}-\frac{1}{\wt v^{\mb X}}\right) + \frac{1}{\wt v^{\mb X}}\left(\wt v^S_{\xi\xi} -\wt v^{\mb X}_{\xi\xi} \right) -\frac{1}{(\wt v^S)^2} \left((\wt v^S_\xi)^2 - (\wt v^{\mb X}_\xi)^2 \right) -(\wt v^{\mb X}_\xi)^2 \left(\frac{1}{(\wt v^S)^2}-\frac{1}{(\wt v^{\mb X})^2} \right) \right| \\
&\le C \Big( |\wt v^S_{\xi\xi}| |(\wt v^R)^{\mb X} -v_m| +|(\wt v^R)^{\mb X}_{\xi\xi}| +|(\wt v^R)^{\mb X}_\xi | |\wt v^S_\xi | +|(\wt v^R)^{\mb X}_{\xi}|^2  \\
&\qquad\ +|\wt v^S_{\xi}|^2 \big|   |(\wt v^R)^{\mb X} -v_m|  \Big) \\
&\le C\big( |(\wt v^R)^{\mb X}_{\xi\xi}|+|(\wt v^R)^{\mb X}_{\xi}|^2 +( |\wt v^S_{\xi\xi}|+|\wt v^S_{\xi}|^2) |(\wt v^R)^{\mb X} -v_m| +|(\wt v^R)^{\mb X}_\xi | |\wt v^S_\xi | \big) ,
\end{aligned}
\end{align}
and
\beq\label{pcomp}
| \big(p(\wt v^{\mb X})-p((\wt v^R)^{\mb X})-p(\wt v^S)\big)_\xi | \le C\big( |(\wt v^R)^{\mb X}_\xi| |\wt v^S-v_m| +|\wt v^S_\xi|  |(\wt v^R)^{\mb X} -v_m| \big).
\eeq
Then,
\begin{align*}
\begin{aligned}
&|\mb S_1|+|\mb S_2| \\&\le C \int_\bbr |w| \big( |(\wt v^R)^{\mb X}_{\xi\xi}|+|(\wt v^R)^{\mb X}_{\xi}|^2 \big) d\xi \\
&\quad + C \int_\bbr (|w| + |h-\wt h^{\mb X}| ) \big( |\wt v^S_\xi| |(\wt v^R)^{\mb X} -v_m| +|(\wt v^R)^{\mb X}_\xi| |\wt v^S-v_m|+|(\wt v^R)^{\mb X}_\xi | |\wt v^S_\xi | \big) d\xi \\
&=: J_1 + J_2.
\end{aligned}
\end{align*}
Using the interpolation inequality and \eqref{apri-ass} with Young's inequality,
\begin{align}
\begin{aligned}\label{j1c}
J_1 &\le C \|w\|_{L^\infty} \|(\wt v^R)^{\mb X}_{\xi\xi}\|_{L^1} + C \|w\|_{L^2} \|(\wt v^R)^{\mb X}_{\xi}\|_{L^4}^2 \\
&\le  C\|w\|_{L^2}^{1/2}\|w_\xi\|_{L^2}^{1/2}  \|(\wt v^R)^{\mb X}_{\xi\xi}\|_{L^1} + C \|w\|_{L^2} \|(\wt v^R)^{\mb X}_{\xi}\|_{L^4}^2 \\
&\le  C\sqrt{\eps_1} \sqrt[4]{\mb D} \|(\wt v^R)^{\mb X}_{\xi\xi}\|_{L^1} + C\eps_1 \|(\wt v^R)^{\mb X}_{\xi}\|_{L^4}^2 \\
&\le \frac{1}{40} \mb D + C\eps_1^{2/3} \|(\wt v^R)^{\mb X}_{\xi\xi}\|_{L^1}^{4/3}+ C\eps_1 \|(\wt v^R)^{\mb X}_{\xi}\|_{L^4}^2.
\end{aligned}
\end{align}
For $J_2$, using \eqref{pressure2}, \eqref{hhe} and \eqref{apri-ass},
\[
\|w\|_{L^2}+\|h-\wt h^{\mb X}\|_{L^2}\le C (\varepsilon_1+\delta_R).
\]
Thus, 
\[
J_2 \le C(\eps_1 +\delta_R)\big\| |\wt v^S_\xi| |(\wt v^R)^{\mb X} -v_m|+|(\wt v^R)^{\mb X}_\xi| |\wt v^S-v_m| +|(\wt v^R)^{\mb X}_\xi | |\wt v^S_\xi | \big\|_{L^2}.
\]

\noindent$\bullet$ {\bf Conclusion :} From the above estimates, we have
\begin{align*}
\begin{aligned}
&\frac{d}{dt}\int_{\bbr} a \eta\big(U|\wt U^{\mb X} \big) d\xi 
\le -\frac{\delta_S}{4M} |\dot{\mb{X}}|^2  -\frac12 \mb G_1 -\frac{C_1}{2} \mb G^S - \frac{1}{8}\mb D \\
&\quad +C\eps_1 \deltas^{4/3} \deltar^{4/3} e^{-C\deltas t} + C\eps_1^{2/3} \|(\wt v^R)^{\mb X}_{\xi\xi}\|_{L^1}^{4/3} \\
&\quad + C\eps_1 \|(\wt v^R)^{\mb X}_{\xi}\|_{L^4}^2  +C(\eps_1+\delta_R) \big\| |\wt v^S_\xi| |(\wt v^R)^{\mb X} -v_m| +|(\wt v^R)^{\mb X}_\xi| |\wt v^S-v_m|+|(\wt v^R)^{\mb X}_\xi | |\wt v^S_\xi | \big\|_{L^2}.
\end{aligned}
\end{align*}
Integrating the above inequality over $[0,t]$ for any $t\le T$, we have
\begin{align*}
\begin{aligned}
&\sup_{t\in[0,T]}\int_{\bbr}  \eta\big(U|\wt U^{\mb X} \big) d\xi +\deltas\int_0^t|\dot{\mb{X}}|^2 ds +\int_0^t (\mb G_1+\mb G^S + \mb D ) ds\\
&\quad\le C\int_{\bbr} \eta\big(U_0|\wt U(0,\xi) \big) d\xi  + C\eps_1 \deltar^{4/3} + C\eps_1^{2/3} \int_0^t \|(\wt v^R)^{\mb X}_{\xi\xi}\|_{L^1}^{4/3}ds+C\eps_1\int_0^t \|(\wt v^R)^{\mb X}_{\xi}\|_{L^4}^2 ds\\
&\qquad +C (\eps_1+\delta_R) \int_0^t \big\| |\wt v^S_\xi| |(\wt v^R)^{\mb X} -v_m| +|(\wt v^R)^{\mb X}_\xi| |\wt v^S-v_m|+|(\wt v^R)^{\mb X}_\xi | |\wt v^S_\xi | \big\|_{L^2} ds.
\end{aligned} 
\end{align*}
Notice that by Lemma \ref{lemma1.2},
\[
\|(\wt v^R)_{\xi\xi}\|_{L^1} \le \left\{ \begin{array}{ll}
      \deltar&\quad\mbox{if } 1+t \le \deltar^{-1}\\
      \frac{1}{1+t} &\quad\mbox{if }  1+t \ge \deltar^{-1} , \\
        \end{array} \right.
\]
and
\[
\|(\wt v^R)_{\xi}\|_{L^4} \le \left\{ \begin{array}{ll}
      \deltar&\quad\mbox{if } 1+t \le \deltar^{-1}\\
      \deltar^{1/4}\frac{1}{(1+t)^{3/4}} &\quad\mbox{if }  1+t \ge \deltar^{-1} , \\
        \end{array} \right.
\]
Thus,
\beq\label{v21}
\int_0^\infty \|(\wt v^R)^{\mb X}_{\xi\xi}\|_{L^1}^{4/3}ds \le C\deltar^{1/3} ,\qquad \int_0^\infty \|(\wt v^R)^{\mb X}_{\xi}\|_{L^4}^2 ds \le C\deltar.
\eeq
In addition, since it follows from Lemma \ref{lemma2.2} that
\beq\label{v20}
\big\| |\wt v^S_\xi| |(\wt v^R)^{\mb X} -v_m| +|(\wt v^R)^{\mb X}_\xi| |\wt v^S-v_m|+|(\wt v^R)^{\mb X}_\xi | |\wt v^S_\xi | \big\|_{L^2} \le C\deltar\deltas e^{-C\deltas t},
\eeq
and so,
\beq\label{v22}
 \int_0^\infty \big\| |\wt v^S_\xi| |(\wt v^R)^{\mb X} -v_m| +|(\wt v^R)^{\mb X}_\xi| |\wt v^S-v_m|+|(\wt v^R)^{\mb X}_\xi ||\wt v^S_\xi |\big\|_{L^2} ds \le C\deltar,
\eeq
we have
\begin{align*}
\begin{aligned}
&\sup_{t\in[0,T]}\int_{\bbr}  \eta\big(U|\wt U^{\mb X} \big) d\xi +\deltas\int_0^t|\dot{\mb{X}}|^2 ds +\int_0^t (\mb G_1+\mb G^S + \mb D ) ds\\
&\quad\le C\int_{\bbr} \eta\big(U_0|\wt U(0,\xi) \big) d\xi  + C \deltar^{1/3}.
\end{aligned} 
\end{align*}
This implies the desired estimate \eqref{esthv} together with the new notations \eqref{good1}, where note that
\[
G_1(U)\sim \mb G_1(U),\quad G^S(U)=\mb G^S(U),\quad D(U)\sim \mb D(U).
\]

\section{Proof of Proposition \ref{prop2}}\label{sec-vu}
\setcounter{equation}{0}

In this section, we use the original system \eqref{NS-1} to estimate $\|u-\tilde u\|_{L^\infty(0,T;H^1(\bbr))}$, and then we complete the proof of Proposition \ref{prop2}.

\subsection{Estimates for $\|u-\tilde u\|_{L^2(\bbr)}$}

We first present the zeroth-order energy estimates for the system \eqref{NS-1}.

\begin{lemma}\label{lem-vu0}
Under the hypotheses of Proposition \ref{prop2}, there exists $C>0$ (independent of $\delta_0, \eps_1, T$) such that for all $ t\in (0,T]$,
\begin{align}
\begin{aligned}\label{estu0}
&\|v-\wt v\|_{H^1(\bbr)}^2 +\|u-\wt u\|_{L^2(\bbr)}^2 +\deltas\int_0^t|\dot{\mb{X}}|^2 ds \\
&\qquad+\int_0^t \left( G^S(U) + G^R(U) + D(U) + D_1(U) \right) ds\\
&\quad\le C\left( \|v_0-\wt v(0,\cdot)\|_{H^1(\bbr)}^2 +\|u_0 -\wt u(0,\cdot)\|_{L^2(\bbr)}^2 \right) + C\deltar^{1/3}  ,
\end{aligned}
\end{align}
where $G^S, D$ are as in \eqref{good1}, and 
\begin{align}
\begin{aligned}\label{good2}
& G^R(U):= \int_\bbr  \wt u^R_\xi p(v|\wt v) d\xi, \\
&D_1(U):=  \int_\bbr \big|(u-\wt u)_\xi \big|^2  d\xi.
\end{aligned}
\end{align}
\end{lemma}

\begin{proof}
First of all, as in Section \ref{ssec:ent}, we first rewrite \eqref{NS-1} into the form: 
\beq\label{system-u}
\partial_t U +\partial_\xi A(U)= \partial_\xi\Big(M(U) \partial_\xi\nabla\eta(U) \Big),
\eeq
where 
\[
U:={v \choose u},\quad A(U):={-\s v -u \choose -\s u+p(v)},\quad M(U):={0 \quad 0 \choose 0\quad \frac{1}{v}},
\]
and note that by the entropy $\eta(U):=\frac{u^2}{2}+Q(v)$ of \eqref{NS-1},
\[
\nabla\eta(U)={-p(v)\choose u}.
\]
By the above representation, the system \eqref{sws-u} can be written as
\beq\label{tilueq2}
\partial_t \wt U  +\partial_\xi A(\wt U)= \partial_\xi\Big(M(\wt U) \partial_\xi\nabla\eta(\wt U) \Big) -\dot{\mb{X}} \partial_\xi \big((\wt U^S)^{-\mb{X}} \big)+ \bmat{0}\\ {F_1+F_2}\emat ,
\eeq
where $F_1, F_2$ are as in \eqref{ff1f2}.\\

Then, applying the equality \eqref{genrhs} with $a\equiv1$ to the system \eqref{system-u}, we have
\[
\frac{d}{dt}\int_{\bbr} \eta\big(U(t,\xi)|\wt U(t,\xi) \big) d\xi  =\dot{\mb X} \mathcal{Y}(U)  +\sum_{i=1}^6 \mathcal{I}_i(U),
\]
\begin{align*}
& \mathcal{Y} (U) := \int_\bbr  \nabla^2\eta(\wt U) (\wt U^S)_\xi^{-\mb{X}}  (U-\wt U) d\xi,\\
&\mathcal{I}_1(U):=-\int_\bbr  \partial_\xi G(U;\wt U) d\xi,\\
&\mathcal{I}_2(U):=- \int_\bbr \partial_\xi \nabla\eta(\wt U) A(U|\wt U) d\xi,\\
&\mathcal{I}_3(U):=\int_\bbr \Big( \nabla\eta(U)-\nabla\eta(\wt U)\Big) \partial_\xi \Big(M(U) \partial_\xi \big(\nabla\eta(U)-\nabla\eta(\wt U)\big) \Big)  d\xi, \\
&\mathcal{I}_4(U):=\int_\bbr  \Big( \nabla\eta(U)-\nabla\eta(\wt U)\Big) \partial_\xi \Big(\big(M(U)-M(\wt U)\big) \partial_\xi \nabla\eta(\wt U) \Big)  d\xi, \\
&\mathcal{I}_5(U):=\int_\bbr (\nabla\eta)(U|\wt U)\partial_\xi \Big(M(\wt U) \partial_\xi \nabla\eta(\wt U) \Big)  d\xi,\\
&\mathcal{I}_6(U):=-\int_\bbr \nabla^2\eta(\wt U) (U-\wt U)   \bmat{0}\\ {F_1+F_2}\emat   d\xi.
\end{align*}
It follows from the above system that
\begin{align*}
&\mathcal{Y} =-\int  p^\prime(\wt v) ({\wt v}^S_\xi)^{-\mb{X}}  (v-\wt v) d\xi + \int   ({\wt u}^S_\xi)^{-\mb{X}}  (u-\wt u) d\xi=:\mathcal{Y}_1+\mathcal{Y}_2,\\
&\mathcal{I}_1= -\int_\bbr  \partial_\xi \big((p(v)-p(\wt v)) (u-\wt u)-\s \eta(U|\wt U) \big) d\xi =0,\\
&\mathcal{I}_2=- \int_\bbr \wt u_\xi p(v|\wt v) d\xi =- \underbrace{\int_\bbr  \wt u^R_\xi p(v|\wt v) d\xi}_{=:G^R} \underbrace{- \int_\bbr  ({\wt u}^S_\xi)^{-\mb{X}} p(v|\wt v) d\xi}_{=:\mathcal{I}_{21}} ,\\
&\mathcal{I}_3=\int_\bbr (u-\wt u)  \Big( \frac{1}{v} (u-\wt u)_\xi \Big)_\xi  d\xi = - \underbrace{ \int_\bbr \frac{1}{v} \big|(u-\wt u)_\xi \big|^2  d\xi}_{=: \mb D_1}, \\
&\mathcal{I}_4=\int_\bbr (u-\wt u)  \left(\left( \frac{1}{v}-\frac{1}{\wt v}\right) \wt u_\xi \right)_\xi  d\xi,\\
&\mathcal{I}_6=-\int_\bbr (u-\wt u)  \left(\left(\frac{ ({\wt u}^S_\xi)^{-\mb{X}} }{ ({\wt v}^S)^{-\mb{X}}} - \frac{ \wt u_\xi }{\wt v}  \right)_\xi +\big(p(\wt v)-p(\wt v^R)-p((\wt v^S)^{-\mb X})\big)_\xi \right)  d\xi.
\end{align*}
In addition, since $(\nabla\eta)(U|\wt U)={-p(v|\wt v) \choose 0}$, we have $\mathcal{I}_5=0$. \\
Since \eqref{rel_p} and \eqref{p-est1} yields
\[
|\mathcal{Y}_1| \le \sqrt{ \int |({\wt v}^S_\xi)^{-\mb{X}}| d\xi} \sqrt{\int  |({\wt v}^S_\xi)^{-\mb{X}}| |v-\wt v|^2 d\xi } \le \sqrt{\deltas} \sqrt{G^S},
\]
we have
\[
|\dot{\mb X}| |\mathcal{Y}_1 | \le \frac{\deltas}{4} |\dot{\mb X}|^2 + C G^S.
\]

To control $\mathcal{Y}_2$, we will use the follows estimate: as done in \eqref{newph},
\[
|u-\wt u| \le |h-\wt h| +C(|(v-\wt v)_\xi| +|\wt v_\xi||v-\wt v|+|(\wt v_\xi^S)^{-\mb{X}}||\wt v^R-v_m|+|\wt v^R_\xi|).
\]
In addition, using the fact that 
\[
(p(v)-p(\wt v))_\xi = p'(v)(v-\wt v)_\xi + \wt v_\xi (p'(v)-p'(\wt v)),
\]
and so,
\[
|(v-\wt v)_\xi | \le C |(p(v)-p(\wt v))_\xi | +C|\wt v_\xi| |v-\wt v|,
\]
we have
\begin{align*}
|\mathcal{Y}_2| &\le C \int  |({\wt v}^S_\xi)^{-\mb{X}}| \Big(\Big| h-\wt h-\frac{p(v)-p(\wt v)}{\sigma}\Big|+ |p(v)-p(\wt v)|\\
&\quad+ |(p(v)-p(\wt v))_\xi| +|\wt v_\xi||v-\wt v|+|(\wt v_\xi^S)^{-\mb{X}}||\wt v^R-v_m|+|\wt v^R_\xi|\Big) d\xi .
\end{align*} 
Then, using Lemma \ref{lemma2.2} to have
\begin{align*}
|\mathcal{Y}_2| &\le C \int  |({\wt v}^S_\xi)^{-\mb{X}}| \Big(\Big| h-\wt h-\frac{p(v)-p(\wt v)}{\sigma}\Big|+ |p(v)-p(\wt v)|\\
&\quad+ |(p(v)-p(\wt v))_\xi| +|\wt v_\xi||v-\wt v|+|(\wt v_\xi^S)^{-\mb{X}}||\wt v^R-v_m|+|\wt v^R_\xi|\Big) d\xi \\
&\le C\big(\frac{\deltas}{\sqrt\lam} \sqrt{ G_1}+ \sqrt{\deltas} \sqrt{ G^S}+\deltas\sqrt{ D}+\delta_S\delta_R e^{-C\delta_S t}\big).
\end{align*} 
Thus, 
\[
|\dot{\mb X}| |\mathcal{Y}_2 | \le \frac{\deltas}{4} |\dot{\mb X}|^2 + C\frac{\deltas}{\lam}  G_1 + C G^S + C\deltas D+C \delta_S\delta_R^2 e^{-C\delta_S t}.
\]
For $\mathcal{I}_2$, note first that $G^R\ge 0$ by $ \wt u^R_\xi >0$. 
Using Lemma \ref{lem-useful},
\[
|\mathcal{I}_{21}|\le C G^S.
\]
We will use the good terms $G^R$ and $\mb D_1$ to control $\mathcal{I}_4, \mathcal{I}_6$.\\
Using $ |\wt u^R_\xi|\le C\deltar,  |({\wt u}^S_\xi)^{-\mb{X}} |\le \deltas$ and Young's inequality, we have
\[
|\mathcal{I}_4|\le \int_\bbr |(u-\wt u)_\xi| \left|v-\wt v\right| \big(  |\wt u^R_\xi| + |({\wt u}^S_\xi)^{-\mb{X}} | \big)  d\xi \le \frac{1}{4} \mb D_1 + C\deltar G^R + C\deltas G^S.
\]
For $\mathcal{I}_6$, using \eqref{pcomp} and (as done in \eqref{ccomp})
\begin{align}
\begin{aligned} \label{f_3}
& \left|\left(\frac{ ({\wt u}^S_\xi)^{-\mb{X}} }{ ({\wt v}^S)^{-\mb{X}}}\right)_\xi  - \left(\frac{ \wt u_\xi }{\wt v}  \right)_\xi \right| \\
&\quad  \le C\bigg( |(\wt u^R)_{\xi\xi}|+|(\wt u^R)_{\xi}| |(\wt v^R)_{\xi}|  +( |(\wt u^S)^{-\mb{X}}_{\xi\xi}|+|(\wt u^S)^{-\mb{X}}_{\xi}||(\wt v^S)^{-\mb{X}}_{\xi}|) |\wt v^R-v_m| \\
&\quad\qquad +|(\wt u^R)_\xi | |(\wt v^S)^{-\mb{X}}_\xi | +|(\wt v^R)_\xi | |(\wt u^S)^{-\mb{X}}_\xi | \bigg)\\
&\quad  \le C\big( |(\wt u^R)_{\xi\xi}|+|(\wt u^R)_{\xi}|^2 +( |(\wt v^S)^{-\mb{X}}_{\xi\xi}|+|(\wt v^S)^{-\mb{X}}_{\xi}|^2) |\wt v^R-v_m| +|(\wt v^R)_\xi | |(\wt v^S)^{-\mb{X}}_\xi | \big) ,
\end{aligned}
\end{align}
we have
\begin{align*}
\begin{aligned}
\mathcal{I}_6 &\le C \int_\bbr |u-\wt u| \big( |(\wt u^R)_{\xi\xi}|+|(\wt u^R)_{\xi}|^2 \big) d\xi \\
&\quad + C \int_\bbr |u-\wt u| \big( |(\wt v^S)^{-\mb{X}}_{\xi}| |\wt v^R-v_m| +|(\wt v^R)_\xi| |(\wt v^S)^{-\mb X}-v_m|+ |(\wt v^R)_\xi | |(\wt v^S)^{-\mb{X}}_\xi | \big) d\xi \\
&=: Q_1 + Q_2.
\end{aligned}
\end{align*}
Using the same estimates as in \eqref{j1c} with \eqref{apri-ass}, we have
\begin{align*}
Q_1 &\le  C\|u-\wt u\|_{L^2}^{1/2}\|(u-\wt u)_\xi\|_{L^2}^{1/2}  \|(\wt u^R)_{\xi\xi}\|_{L^1} + C \|u-\wt u\|_{L^2} \|(\wt u^R)_{\xi}\|_{L^4}^2 \\
&\le  C\sqrt{\eps_1} \sqrt[4]{\mb D_1} \|(\wt u^R)_{\xi\xi}\|_{L^1} + C\eps_1 \|(\wt u^R)_{\xi}\|_{L^4}^2 \\
&\le \frac{1}{4} \mb D_1 + C\eps_1^{2/3} \|(\wt u^R)_{\xi\xi}\|_{L^1}^{4/3}+ C\eps_1 \|(\wt u^R)_{\xi}\|_{L^4}^2.
\end{align*}
Using \eqref{apri-ass}, we have
\[
Q_2 \le C\eps_1 \big\| |(\wt v^S)^{-\mb{X}}_{\xi}| |\wt v^R-v_m| +|(\wt v^R)_\xi| |(\wt v^S)^{-\mb X}-v_m|+ |(\wt v^R)_\xi | |(\wt v^S)^{-\mb{X}}_\xi |  \big\|_{L^2}.
\]
Therefore, from the above estimates, we find that for some constant $c_1>0$,
\begin{align*}
&\frac{d}{dt}\int_{\bbr} \eta\big(U(t,\xi)|\wt U(t,\xi) \big) d\xi  + \frac{1}{2} G^R + \frac{1}{2} \mb D_1\\
&\quad\le \frac{\deltas}{2} |\dot{\mb X}|^2 +C\frac{\deltas}{\lam}  G_1 + c_1  G^S + C\deltas D + C\eps_1^{2/3} \|(\wt v^R)_{\xi\xi}\|_{L^1}^{4/3}+ C\eps_1 \|(\wt v^R)_{\xi}\|_{L^4}^2\\
&\qquad  +C\eps_1\big\| |(\wt v^S)^{-\mb{X}}_{\xi}| |\wt v^R-v_m| +|(\wt v^R)_\xi| |(\wt v^S)^{-\mb X}-v_m|+ |(\wt v^R)_\xi | |(\wt v^S)^{-\mb{X}}_\xi |  \big\|_{L^2}
\\
&\qquad  +C \delta_S\delta_R^2 e^{-C\delta_S t}.
\end{align*}
Integrating the above inequality over $[0,t]$ for any $t\le T$, and using \eqref{v21}-\eqref{v22}, we have
\begin{align}
\begin{aligned}\label{f00}
&\int_{\bbr} \left(\frac{|u-\wt u|^2}{2} +Q(v|\wt v)\right) d\xi +\frac{1}{2} \int_0^t \left( G^R(U)+\mb D_1(U) \right) ds\\
&\quad\le \int_{\bbr} \left(\frac{|u_0-\wt u(0,\xi)|^2}{2} +Q(v_0|\wt v(0,\xi))\right)  d\xi    \\
& \quad\quad + \int_0^t \left( \frac{\deltas}{2} |\dot{\mb X}|^2+C\frac{\deltas}{\lam}  G_1 + c_1  G^S + C\deltas D \right) ds + C \deltar^{1/3} .
\end{aligned}
\end{align}
Therefore, multiplying \eqref{f00} by the constant $\frac{1}{2\max(1,c_1)}$, and then adding the result to \eqref{esthv}, together with the smallness of $\deltas/\lam, \deltas, \eps_1$, we have
\begin{align}\label{5e}
&\|v-\wt v\|_{L^2(\bbr)}^2+ \|h-\wt h\|_{L^2(\bbr)}^2 +\|u-\wt u\|_{L^2(\bbr)}^2 +\deltas\int_0^t|\dot{\mb{X}}|^2 ds \nonumber\\
&+\int_0^t \left( G^R+ G^S+ D+\mb D_1\right) ds\\
&\le C\big(\|v_0-\wt v(0,\cdot)\|_{L^2(\bbr)}^2+ \|(h-\wt h)(0,\cdot)\|_{L^2(\bbr)}^2 +\|u_0-\wt u(0,\cdot)\|_{L^2(\bbr)}^2 \big)  + C\deltar^{1/3},\nonumber
\end{align}
where we have used that (by Lemma \ref{lem-useful} and \eqref{smpressure})
\[
C^{-1}\big|v-\wt v \big|^2 \le Q(v|\wt v) \le C \big|v-\wt v \big|^2.
\]
Finally, to complete the proof, we will show that
\beq\label{5ee}
\|(v-\wt v)_\xi\|_{L^2(\bbr)}^2\leq C\Big[\|h-\wt h\|_{L^2(\bbr)}^2 +\|u-\wt u\|_{L^2(\bbr)}^2+\|v-\wt v\|_{L^2(\bbr)}^2+ \delta_R^2\Big],
\eeq
and 
\beq\label{5eee}
 \|(h-\wt h)(0,\cdot)\|_{L^2(\bbr)}^2 \le  C\Big[ \|v_0-\wt v(0,\cdot)\|_{H^1(\bbr)}^2 +\|u_0-\wt u(0,\cdot)\|_{L^2(\bbr)}^2+\delta_R^2 \Big].
\eeq
Using the definition of $h$ in \eqref{h} and $\wt h$ in \eqref{huvx}, we observe that
\[
(u-\wt u) -(h-\wt h) = \big(\ln v - \ln (\wt v^S)^{-\mb X}\big)_\xi = \frac{\big(v-(\wt v^S)^{-\mb X}\big)_\xi}{v} + \frac{(\wt v^S)^{-\mb X}_\xi\big( (\wt v^S)^{-\mb X} -v \big)}{v(\wt v^S)^{-\mb X}},
\]
which yields
\begin{align*}
(v-\wt v)_\xi &= \big(v-(\wt v^S)^{-\mb X}\big)_\xi -\big(\wt v-(\wt v^S)^{-\mb X}\big)_\xi \\
&= v(u-\wt u) -v(h-\wt h)  +\frac{(\wt v^S)^{-\mb X}_\xi\big((v-\wt v)+(\wt v^R-v_m) \big)}{(\wt v^S)^{-\mb X}} -\wt v^R_\xi.
\end{align*}
This with Lemma \ref{lemma1.2} and Lemma \ref{lemma2.2} implies \eqref{5ee}.\\
As in \eqref{newph}, we have
\[
 \|(h-\wt h)(0,\cdot)\|_{L^2(\bbr)}^2 \le  C\Big[ \|v_0-\wt v(0,\cdot)\|_{H^1(\bbr)}^2 +\|u_0-\wt u(0,\cdot)\|_{L^2(\bbr)}^2+\delta_R^2\|\wt v^S_\xi\|_{L^2(\bbr)}^2 +\|\wt v^R_\xi(0)\|_{L^2(\bbr)}^2  \Big],
\]
which together with Lemmas \ref{lemma1.3} and \ref{lemma1.2} implies \eqref{5eee}.\\
Hence, the combination of \eqref{5e},\eqref{5ee} and \eqref{5eee} implies the desired estimate.

\end{proof}

\subsection{Estimates for $\|\partial_\xi (u-\tilde u)\|_{L^2(\bbr)}$}
We here complete the proof of Proposition \ref{prop2}, by using the following lemma together with the following two estimates (by using Lemma \ref{lem-useful}) :
\begin{align*}
\begin{aligned}
&\mathcal{G}^S(U)=\int_\bbr |(\wt v^S)_\xi^{-\mb X} | |v-\wt v|^2 d\xi \le C G^S(U),\\
& \mathcal{G}^R(U)=\int_\bbr  | \wt u^R_\xi| \big|v-\wt v \big|^2 d\xi \le C G^R(U).
\end{aligned}
\end{align*}

\begin{lemma}\label{lem-u1}
Under the hypotheses of Proposition \ref{prop2}, there exist $C_1, C>0$ (independent of $\delta_0, \eps_1, T$) such that for all $ t\in (0,T]$,
\begin{align*}
\begin{aligned}
&\|v-\wt v\|_{H^1(\bbr)}^2 +\|u-\wt u\|_{H^1(\bbr)}^2 +\deltas\int_0^t|\dot{\mb{X}}|^2 ds \\
&+\int_0^t \left( G^S(U) + G^R(U) + D(U) + D_1(U) + D_2(U) \right) ds\\
&\le C\left( \|v_0-\wt v(0,\cdot)\|_{H^1(\bbr)}^2 +\|u_0 -\wt u(0,\cdot)\|_{H^1(\bbr)}^2 \right) + C\deltar^{1/3} ,
\end{aligned}
\end{align*}
where $G^S, D$ are as in \eqref{good1}, and $G^R, D_1$ are as in \eqref{good2}, and
\[
D_2(U):=  \int_\bbr \big|(u-\wt u)_{\xi\xi} \big|^2  d\xi.
\]
\end{lemma}
\begin{proof}
For notational simplicity, we set $\psi:=u-\wt u$. Then, it follows from the second equations of \eqref{NS-1} and \eqref{sws-u} that
\[
\psi_t-\s \psi_\xi -\dot{\mb{X}}(\wt u^S)^{-\mb{X}}_\xi+(p(v)-p(\wt v))_\xi=\left(\frac{u_\xi}{v} - \frac{\wt u_\xi}{\wt v}\right)_\xi -F_1-F_2.
\]
Multiplying the above equation by $-\psi_{\xi\xi}$ and integrating the result w.r.t. $\xi$, we have
\begin{align*}
&\frac{d}{dt}\int_{\bbr} \frac{|\psi_\xi|^2}{2} d\xi +\s \underbrace{\int_\bbr \left(\frac{|\psi_\xi|^2}{2} \right)_\xi d\xi}_{=0} \\
&\quad=    -\dot{\mb{X}} \int_\bbr (\wt u^S)^{-\mb{X}}_\xi \psi_{\xi\xi} d\xi + \int_\bbr (p(v)-p(\wt v))_\xi \psi_{\xi\xi} d\xi   \\
&\qquad - \int_\bbr \left(\frac{u_\xi}{v}- \frac{\wt u_\xi}{\wt v}\right)_\xi  \psi_{\xi\xi} d\xi + \int_\bbr (F_1+F_2) \psi_{\xi\xi} d\xi \\
&\quad=: J_1 + J_2 +J_3 +J_4. 
\end{align*}
First, we get a good term $$\mb D_2:= \int_\bbr \frac{1}{v} |\psi_{\xi\xi}|^2 d\xi$$  from $J_3$ as follows:
\begin{align*}
J_3 &= - \int_\bbr \frac{1}{v} |\psi_{\xi\xi}|^2 d\xi - \int_\bbr \left(\frac{1}{v}\right)_\xi \psi_\xi \psi_{\xi\xi} d\xi  - \int_\bbr \wt u_{\xi\xi} \left(\frac{1}{v} - \frac{1}{\wt v}\right)  \psi_{\xi\xi} d\xi  \\
&\quad - \int_\bbr \wt u_\xi \left(\frac{1}{v}- \frac{1}{\wt v}\right)_\xi  \psi_{\xi\xi} d\xi \\
&=: - \mb D_2 + J_{31}+ J_{32}+ J_{33}.
\end{align*}
We use the good terms $\mb D_2, D, D_2, G^S$ and $G^R$ to control the remaining terms as follows.\\
Using Young's inequality,
\[
|J_1| \le |\dot{\mb{X}}| \deltas^2 \int_\bbr |\psi_{\xi\xi}| d\xi \le \frac{\deltas}{2}|\dot{\mb{X}}|^2 + C\deltas^3 \mb D_2 \le \frac{\deltas}{2}|\dot{\mb{X}}|^2 +\frac18 \mb D_2,
\]
\[
|J_2|\le \frac18 \mb D_2 + C D.
\]
Using $\left(\frac{1}{v}\right)_\xi \le C|v_\xi| \le C(|(v-\wt v)_\xi| + |\wt v_\xi|)$, and the interpolation inequality and \eqref{apri-ass}, we have
\begin{align*}
|J_{31}| &\le \|(v-\wt v)_\xi\|_{L^2} \|\psi_\xi\|_{L^\infty}  \|\psi_{\xi\xi}\|_{L^2}  + \|\wt v_\xi\|_{L^\infty}   \|\psi_{\xi}\|_{L^2} \|\psi_{\xi\xi}\|_{L^2} \\
&\le C\eps_1 \|\psi_\xi\|_{L^2}^{1/2} \|\psi_{\xi\xi}\|_{L^2}^{1/2}  \|\psi_{\xi\xi}\|_{L^2}  + C(\deltas+\deltar) \|\psi_{\xi}\|_{L^2} \|\psi_{\xi\xi}\|_{L^2} \\
&\le C(\eps_1+\deltas+\deltar) \big( \|\psi_{\xi}\|_{L^2}^2 + \|\psi_{\xi\xi}\|_{L^2}^2 \big) \le \frac18 \mb D_2 + C(\eps_1+\deltas+\deltar) D_1.
\end{align*} 
Using $|(\wt u^R_{\xi\xi})| \leq C |(\wt u^R_{\xi})| $ (by Lemma \ref{lemma1.2}),
\[
|J_{32}| \le C\int_\bbr (|(\wt u^S_{\xi})| + |(\wt u^R_{\xi})| ) |v-\wt v| |\psi_{\xi\xi}| d\xi \le \frac18 \mb D_2 + C\deltas G^S + C\deltar G^R, 
\]
\begin{align*}
|J_{33}| &\le C\int_\bbr (|(\wt u^S_{\xi})| + |(\wt u^R_{\xi})| ) \big( |v-\wt v| +|(v-\wt v)_\xi| \big) |\psi_{\xi\xi}| d\xi \\
&\le  \frac18 \mb D_2 + C(\deltas + \deltar) ( G^S + G^R + D).
\end{align*} 
Using \eqref{f_3},
\begin{align*}
|J_4| & \le C\|\psi_{\xi\xi}\|_{L^2} \big\| |(\wt u^R)_{\xi\xi}|+|(\wt u^R)_{\xi}|^2 +( |(\wt v^S)^{-\mb{X}}_{\xi\xi}|+|(\wt v^S)^{-\mb{X}}_{\xi}|^2) |\wt v^R-v_m| +|(\wt v^R)_\xi | |(\wt v^S)^{-\mb{X}}_\xi | \big\|_{L^2} \\
&\le \frac18 \mb D_2 + C\|(\wt u^R)_{\xi\xi}\|_{L^2}^2 + C\|(\wt u^R)_{\xi}\|_{L^4}^4 +C \| |(\wt v^S)^{-\mb{X}}_{\xi}| |\wt v^R-v_m| +C|(\wt v^R)_\xi | |(\wt v^S)^{-\mb{X}}_\xi | \|_{L^2}^2.
\end{align*}
Therefore, we find that for some $c_2>0$,
\begin{align*}
\frac{d}{dt}\int_{\bbr} \frac{|\psi_\xi|^2}{2} d\xi &=  -\frac14 \mb D_2 + \frac{\deltas}{2}|\dot{\mb{X}}|^2 + c_2 D + C(\eps_1 + \deltas + \deltar) ( G^S + G^R + D_1) \\
&\quad +C \|(\wt u^R)_{\xi\xi}\|_{L^2}^2 +C \|(\wt u^R)_{\xi}\|_{L^4}^4 + C\| |(\wt v^S)^{-\mb{X}}_{\xi}| |\wt v^R-v_m| +|(\wt v^R)_\xi | |(\wt v^S)^{-\mb{X}}_\xi | \|_{L^2}^2.
\end{align*}
Integrating the above estimate over $[0,t]$ for any $t\le T$, and using \eqref{v20} and the fact that (by Lemma \ref{lemma1.2}) 
\[
\int_0^\infty \|(\wt u^R)_{\xi\xi}\|_{L^2}^2 ds \le C\deltar,\qquad \int_0^\infty \|(\wt u^R)_{\xi}\|_{L^4}^4 ds \le C\deltar^3,
\]
we have
\begin{align*}
\int_{\bbr} \frac{|(u-\wt u)_\xi|^2}{2} d\xi &\le \int_{\bbr} \frac{|(u_0-\wt u(0,\xi))_\xi|^2}{2} d\xi+\int_0^t\big[ -\frac14 \mb D_2 + \frac{\deltas}{2}|\dot{\mb{X}}|^2 \\
&\quad + c_2 D + C(\eps_1 + \deltas + \deltar) ( G^S + G^R + D_1) \big] ds+ C\deltar.
\end{align*}
Multiplying the above inequality by the constant $\frac{1}{2\max(1,c_2)}$, and then adding the result to \eqref{estu0}, together with the smallness of $\eps_1,\deltas,\deltar$, we have
\begin{align*}
&\|v-\wt v\|_{H^1(\bbr)}^2 +\|u-\wt u\|_{H^1(\bbr)}^2 +\deltas\int_0^t|\dot{\mb{X}}|^2 ds +\int_0^t \left( G^R+ G^S+ D+D_1+\mb D_2 \right) ds\\
&\quad\le C\big(\|v_0-\wt v(0,\cdot)\|_{H^1(\bbr)}^2 +\|u_0-\wt u(0,\cdot)\|_{H^1(\bbr)}^2 \big)  + C \deltar^{1/3} .
\end{align*}
This implies the desired result in Lemma \ref{lem-u1}.
\end{proof}

\noindent{\bf Conflict of Interest:} The authors declared that they have no conflicts of interest to this work.

\bibliography{KVW}

\end{document}